\documentclass[10pt]{amsart}

\usepackage{txfonts}
\usepackage[T1]{fontenc}

\usepackage[utf8]{inputenc} 
\usepackage{booktabs} 
\usepackage{array} 
\usepackage{paralist} 
\usepackage{verbatim} 
\usepackage{tikz}
\usetikzlibrary{arrows}
\usetikzlibrary{decorations.markings}
\usepackage{subfig}
\usepackage{tabularx}
\usepackage{enumitem}
\usepackage{amsmath,amsthm,amscd}
\usepackage{amssymb,amsfonts}
\usepackage{fancyhdr} 
\usepackage{pgf,tikz}
\usepackage{caption}
\usepackage{tikz-cd}
\usepackage{tikzscale}
\usetikzlibrary{patterns}
\usepackage{rotating}
\usepackage{xcolor}
\usepackage{lscape}
\usepackage{xspace}
\usepackage{xfrac}
\usepackage{blindtext}
\usepackage{bm}
\usepackage[ruled,vlined]{algorithm2e}

\usetikzlibrary{decorations.markings}
\usetikzlibrary{patterns}
\usetikzlibrary{calc}
\usepackage{soul}

\usepackage{tikz}
        \usetikzlibrary{patterns,hobby}
        \usepackage{pgfplots}
        \pgfplotsset{compat=1.6}

\usepackage{faktor} 
\usepackage{pgf,tikz}
\usetikzlibrary{arrows.meta}
\usetikzlibrary{decorations.markings}
\usetikzlibrary{patterns}

\usepackage{color}   
\usepackage{hyperref}
\hypersetup{
    colorlinks=true, 
    linktoc=all,     
    linkcolor=blue,  
    citecolor=blue,
    filecolor=blue,
    urlcolor=blue
}

\usepackage[T1]{fontenc}
\usepackage{hyperref}
\usepackage{nameref}

\newcommand{\mO}{\mathcal{O}}

\newcommand{\lcm}{\operatorname{lcm}}
\DeclareMathOperator{\res}{res}

\newcommand\CC{{\mathbb C}}
\newcommand\PP{{\mathbb P}}

\newcommand\calH{{\mathcal H}}
\newcommand\calR{{\mathcal R}}
\newcommand\calE{{\mathcal E}}

\newcommand\ZZ{{\mathbb Z}}
\newcommand\calF{{\mathcal F}}

\DeclareMathOperator\MS{MS}
\DeclareMathOperator\Res{Res}
\DeclareMathOperator\LG{LG}

\newcommand{\Id}{{\mathrm{Id}}}

\DeclareDocumentCommand{\rec}{ O{a} O{n}}{\left[ #1 \right]^{#2}}

\usepackage[
	    style=alphabetic,
            isbn=false,
            doi=false,
            url=false,
            backend=bibtex,
            maxnames = 5,
            sorting=nty,
           ]{biblatex}
\defbibheading{bibliography}[\bibname]{%
   \section*{References}%
}

\DeclareFieldFormat[article]{title}{{\it #1}}
\DeclareFieldFormat{journaltitle}{{\rm #1}}
\renewbibmacro{in:}{%
  \ifentrytype{article}{}{\printtext{\bibstring{in}\intitlepunct}}}
  \bibliography{biblio}

\newtheorem{thm}{Theorem}[section]
\newtheorem{cor}[thm]{Corollary}
\newtheorem{prop}[thm]{Proposition}
\newtheorem{lem}[thm]{Lemma}
\newtheorem{opprob}[thm]{Open Problem}
\newtheorem{conj}[thm]{Conjecture}
\theoremstyle{definition}
\newtheorem{defn}[thm]{Definition}
\theoremstyle{remark}
\newtheorem{rmk}[thm]{Remark}
\theoremstyle{definition}

\theoremstyle{definition}
\newtheorem{ex}[thm]{Example}
\theoremstyle{definition}

\numberwithin{equation}{section}
\title{Isoresidual curves}

\author{Dawei Chen} 
\address[Dawei Chen]{Department of Mathematics, Boston College, Chestnut Hill, MA 02467, USA}
\email{dawei.chen@bc.edu}

\author{Quentin Gendron}
\address[Quentin Gendron]{Instituto de Matem\'{a}ticas de la UNAM,
Ciudad Universitaria, Tenochtitlán, 04510,
M\'{e}xico}
\email{quentin.gendron@im.unam.mx}

\author{Miguel Prado} 
\address[Miguel Prado]{Goethe University, Frankfurt am Main, Hessen 60325, Germany}
\email{prado@math.uni-frankfurt.de}

\author{Guillaume Tahar}
\address[Guillaume Tahar]{Beijing Institute of Mathematical Sciences and Applications, Huairou District, Beijing, China}
\email{guillaume.tahar@bimsa.cn}

\date{\today}
\keywords{
Isoresidual fibration, Translation surfaces, Multi-scale compactification, Resonance arrangement, Gauss--Manin connection}

\begin{document}
\begin{abstract}
Given a partition $\mu$ of $-2$, the stratum $\mathcal{H}(\mu)$ parametrizes meromorphic differential one-forms on the Riemann sphere $\mathbb{CP}^{1}$ with~$n$ zeros and $p$ poles of orders prescribed by $\mu$. The isoresidual fibration is defined by assigning to each differential in $\mathcal{H}(\mu)$ its configuration of residues at the poles. In the case of differentials with $n=2$ zeros, generic isoresidual fibers are complex curves endowed with a canonical translation structure, which we describe extensively in this paper. Quantitative characteristics of the translation structure on isoresidual fiber curves, including the orders of the singularities and a period central charge encapsulating the linear dependence of periods on the underlying configuration of residues, provide rich discrete invariants for these fibers. We also determine the Euler characteristic of generic isoresidual fiber curves from intersection-theoretic computations, relying on the multi-scale compactification of strata of differentials. In particular, we describe a wall and chamber structure for the Euler characteristic of generic isoresidual fiber curves in terms of the partition $\mu$. Additionally, we classify the connected components of generic isoresidual fibers for strata in genus zero with an arbitrary number of zeros.
\end{abstract}
\maketitle
\setcounter{tocdepth}{1}
\tableofcontents

\section{Introduction}

For any family of positive integers $b_{1},\dots,b_{p}$, we denote by $\Omega(b_{1},\dots,b_{p})$ the moduli space of meromorphic one-forms on the Riemann sphere $\mathbb{CP}^{1}$ with (labeled) poles of orders $b_{1},\dots,b_{p}$ (up to biholomorphisms). The moduli space $\Omega(b_{1},\dots,b_{p})$ is stratified according to the orders $a_{1},\dots,a_{n}$ of the (labeled) zeros of differential forms.
Each stratum $\mathcal{H}(\mu)$ is characterized by a partition $\mu=(a_{1},\dots,a_{n},-b_{1},\dots,-b_{p})$ of $-2$ with $a_{1},\dots,a_{n} \geq 1$. It is known that $\mathcal{H}(\mu)$ is a complex-analytic orbifold of dimension $n+p-2$, whose underlying coarse space is a quasi-projective variety.
\par
For each $p \geq 2$, we define the \textit{residual space} $\mathcal{R}_{p}$ to be the complex vector space formed by the vectors $\lambda = (\lambda_{1},\dots,\lambda_{p})$ such that $\sum\limits_{j=1}^{p} \lambda_{j} = 0$. Each moduli space $\Omega(b_{1},\dots,b_{p})$ is endowed with a \textit{residual map}
$$\res\colon \Omega(b_{1},\dots,b_{p}) \to \mathcal{R}_{p} : \omega \mapsto (\Res_{1}\omega,\dots,\Res_{p}\omega)$$
that assigns to each differential $\omega$ the sequence of its residues at the poles. This map defines the \textit{isoresidual fibration} for the entire moduli space and restricts to each stratum $\mathcal{H}(\mu)$.
\smallskip
\par
The  in-depth  study of these maps began in \cite{getaab}, where the image of $\res$ is described for every stratum $\mathcal{H}(\mu)$ and in every genus. Then, in  \cite{GeTaIso}, the case of the {\em minimal strata} in genus zero, i.e., strata with a unique zero, was studied from a flat geometric perspective by two of the authors. The other two authors examined the same cases in \cite{ChPrIso} from an algebro-geometric point of view. In that case, these works show that the isoresidual fibration of such strata is a finite cover ramified over the \textit{resonance arrangement} of the residual space.
\begin{defn}\label{defn:resonanceintro}
Let $I$ be a nonempty proper subset of $\{1,\dots,p\}$. The \textit{resonance hyperplane} $A_{I}$ is the subspace defined by the equation $\sum\limits_{i \in I} \lambda_{i} = 0$ in the residual space $\mathcal{R}_{p}$. The union of all the resonance hyperplanes is the \textit{resonance arrangement} $\mathcal{A}_{p} \subset \mathcal{R}_{p}$.
\end{defn}
In the case of the minimal stratum, the generic degree of the isoresidual cover has been computed in Theorem~1.2 of \cite{GeTaIso} as 
\begin{equation}\label{eq:degree}
 f(a,p):=\frac{a!}{(a+2-p)!}\,,
\end{equation}
where $a = -2+\sum\limits_{j=1}^{p}b_{j}$ is the order of the unique zero. The cardinality of an arbitrary fiber over the resonance arrangement in $\mathcal{R}_{p}$ is given in Theorem 1.2 of \cite{ChPrIso}.
\par 
It is worth noting that this isoresidual cover has recently appeared in various other contexts, such as the dynamics of polynomial maps in \cite{Sugiyama}, the topology of configuration spaces in \cite{SaBraid}, and the Kadomtsev--Petviashvili hierarchy in \cite{BuRoCoun}. In higher genus, the analog of the isoresidual fibration is the isoperiodic foliation which has been extensively described for the stratum $\mathcal{H}(1,1,-2)$ in \cite{FaTaZa} (see \cite{KLS,CMDisoper,BGPA} for related questions).
\par
The goal of the present paper is to combine flat geometric and algebro-geometric approaches to study the isoresidual fiber in the strata of differentials on the Riemann sphere with $n=2$ zeros.
\par
\subsection{Main results}
\par 
We fix a stratum $\mathcal{H}(\mu)$ of differentials on $\mathbb{CP}^1$ with $n$ zeros and $p$ poles, where the orders of the singularities are prescribed by $\mu$. For a given residue configuration $\lambda=(\lambda_{1},\dots,\lambda_{p}) \in \mathcal{R}_{p}$, we denote by $\mathcal{F}_{\lambda}$ the isoresidual fiber in the stratum $\mathcal{H}(\mu)$ parametrizing differentials with the residue configuration $\lambda$. Note that $\mathcal{F}_{\lambda}$ depends on $\mu$, but we do not indicate it in order to keep the notation simple. We further denote by~$\overline{\mathcal{F}}_{\lambda}$ the closure of  $\mathcal{F}_{\lambda}$ in the multi-scale compactification of $\mathcal{H}(\mu)$ (see Section~\ref{sec:multi-scale}).
\par
Deformation of a differential in the stratum $\mathcal{H}(\mu)$ means changing the periods of the differential along (relative) homology classes. Inside an isoresidual fiber, the absolute periods (i.e., the residues at the poles) are fixed, so the only degree of freedom is from the relative periods between the zeros. In the case of $n=2$ zeros, this relative period serves as a local coordinate defining the translation structure on the isoresidual fiber,  and its differential gives rise to the canonical one-form $\omega_{\lambda}$ on $\calF_{\lambda}$.
\par
Our first main result describes the translation structure of the one-dimensional isoresidual fibers of the strata with $n=2$ zeros outside the resonance arrangement $\mathcal{A}_{p}$. We denote by $a = a_{1}+a_{2}$ the total order of the two zeros.
\par 
\begin{thm}\label{thm:MAIN1}
For a stratum $\mathcal{H}(a_{1},a_{2},-b_{1},\dots,-b_{p})$ and a configuration $\lambda \in \mathcal{R}_{p} \setminus \mathcal{A}_{p}$, the closure $\overline{\mathcal{F}}_{\lambda}$ of the isoresidual fiber $\mathcal{F}_{\lambda}$ is a (possibly disconnected) compact Riemann surface endowed with a meromorphic one-form $\omega_{\lambda}$. Additionally, $\omega_{\lambda}$ induces a translation structure on $\overline{\mathcal{F}}_{\lambda}$ that satisfies the following properties:
\begin{enumerate}
    \item $\mathcal{F}_{\lambda}$ coincides with the locus where $\omega_{\lambda}$ has neither a zero nor a pole; 
    \item $\omega_{\lambda}$ has $\frac{a!}{(a+2-p)!}$ zeros, each of order $a$;
    \item $\omega_{\lambda}$ always has poles whose orders and residues are described in Sections~\ref{subsub:SimplePoles} and~\ref{subsub:HigherPoles};
    \item $(\overline{\mathcal{F}}_{\lambda},\omega_{\lambda})$ has finitely many saddle connections;
    \item each saddle connection of $(\overline{\mathcal{F}}_{\lambda},\omega_{\lambda})$ has a period of the form $\sum\limits_{i \in I} \lambda_{i}$, where $I$ is a nonempty proper subset of $\{1,\dots,p\}$;
    \item if $\mathcal{F}_{\lambda}$ is disconnected, then all the components of $(\overline{\mathcal{F}}_{\lambda},\omega_{\lambda})$ belong to the same stratum component; 
    \item the stratum component containing $(\overline{\mathcal{F}}_{\lambda},\omega_{\lambda})$ is independent of $\lambda$.
\end{enumerate}
\end{thm}
\par 
Let us explain some parts of the theorem in the language of flat geometry. The zeros of $\omega_\lambda$ correspond to elements of the minimal stratum when the two zeros merge to form a single zero of order $a$. Hence, all the zeros of $\omega_\lambda$ have the same order, and there are $f(a,p)$ of them. The poles of $\omega_\lambda$ correspond to degenerations where the two zeros move infinitely far from each other in the flat metric of $\omega$.
The description of the number, orders, and residues of all the poles of a given isoresidual curve requires combinatorial work, which is done in Section~\ref{sub:Discretecombinatorics}. The list of possible poles of $\omega_{\lambda}$ crucially depends on the location of the partition $\mu$ in the following \textit{singularity pattern space}~$\mathcal{SP}_{p}$.
\par 
\begin{defn}\label{defn:SPS}
Given $p \geq 1$, the \textit{singularity pattern space} is the positive orthant $$\mathcal{SP}_{p}=\left\{ (x_{1},x_{2},y_{1},\dots,y_{p}) \in \mathbb{R}_{>0}^{p+2}~\mid~x_{1}+x_{2}= \sum\limits_{j=1}^p y_{j} \right\}$$ of real dimension $p+1$, endowed with the family $\mathcal{W}_{p}$ of hyperplanes 
\begin{equation}\label{eq:hyptype1}
 W_{1}(I)=\left\{\beta_{1}:=x_{1} - \sum\limits_{i \in I} y_{i} = 0\right\}, W_{2}(I)=\left\{\beta_{2}:=x_{2} - \sum\limits_{i \in I} y_{i} = 0\right\}
\end{equation}
and 
\begin{equation}\label{eq:hyptype2}
 W_{3}(I,K,L)=\left\{\beta_{3}:=x_{1} - \sum\limits_{i \in I \cup L} y_{i} -|K \setminus L| = 0\right\}
\end{equation}
and
\begin{equation}\label{eq:hyptype3}
 W_{4}(J,K,M)=\left\{\beta_{4}:=x_{2} - \sum\limits_{i \in J \cup M} y_{i} -|K \setminus M| = 0\right\}
\end{equation}
where $I \sqcup J \sqcup K$ is a partition of the index set of the poles $\lbrace{1,\dots,p \rbrace}$ into three disjoint subsets, and $L$ and~$M$ are (possibly intersecting) arbitrary subsets of $K$.
\end{defn}
\par 
Each partition $\mu$ determines a point $v_{\mu}=(a_{1}+1,a_{2}+1,b_{1},\dots,b_{p})\in \mathcal{SP}_{p}$. The sum of the orders of the poles of~$\omega_{\lambda}$ is given by a complicated combinatorial formula that depends on the chamber of $(\mathcal{SP}_{p},\mathcal{W}_{p})$ containing $v_{\mu}$. Nevertheless, the Euler characteristic of $\overline{\mathcal F}_\lambda$ possesses the following structure. 
\par 
\begin{thm}\label{thm:chambers}
For $\lambda \in \mathcal \mathcal{R}_{p} \setminus \mathcal{A}_{p}$ and $(a_1+1, a_2+1, b_1, \ldots, b_p)$ in each chamber of $\mathcal{SP}_p$, the Euler characteristic of $\overline{\mathcal F}_\lambda$ is a sum of homogeneous components of degree from $0$ up to $p-1$ in terms of the variables $a_{1}+1,a_{2}+1,b_{1},\dots,b_{p}$. 
\end{thm}
In Proposition~\ref{prop:chambers}, we provide a more detailed description of the above theorem. Moreover, in  Conjecture~\ref{conj:Euler}, 
we speculate that the main term in the formula of the Euler characteristic is always a homogeneous polynomial of degree exactly $p-1$ in the variables $a_{1}+1,a_{2}+1,b_{1},\dots,b_{p}$. We check this conjecture in a specific chamber in  Proposition~\ref{prop:one-chamber}.
\par 
In the case where all poles are simple, i.e., $b_i = 1$ for all $i$, we denote by $\mu = (a_1, a_2, \rec[-1][a+2])$ the partition $(a_{1},a_{2},-1,\dots,-1)$. The Euler characteristic of the generic isoresidual fiber is given as follows in that case.
\par 
\begin{thm}\label{thm:simple}
For $\mu = (a_1, a_2, \rec[-1][a+2])$ with $a=a_{1}+a_{2}$ and $\lambda\in \mathcal{R}_{p} \setminus \mathcal{A}_{p}$, we have
\begin{equation}\label{eq:genussimples}
 2g\left(\overline{\mathcal{F}}_{\lambda}\right)-2 = 
a! \left(a  - \frac{(a+2)(a+1)}{(a_1+1)(a_2+1)} \right)\,. 
\end{equation}
\end{thm}
\par
Note that, in general, we will show in part (2) of Theorem~\ref{thm:MAIN5} that the generic isoresidual fiber of $\calH(a_1, a_2, \rec[-1][a+2])$ is disconnected if and only if $a_1$ and $a_2$ are both even. In that case, the generic isoresidual fiber has two connected components, and~\eqref{eq:genussimples} gives the arithmetic genus of the disjoint union of these two components. 
\par 
For general $\mu$, although there is no simple formula to deduce the pole orders of $\omega_\lambda$ from $\mu$ for the corresponding generic isoresidual fiber, we can still describe some general properties. 
\par 
\begin{thm}\label{thm:MAIN3}
All the poles of $(\overline{\mathcal{F}}_{\lambda},\omega_{\lambda})$ are simple if and only if $\mu=(a_1, a_2, \rec[-1][a+2])$.

For $p \geq 3$, all the singularities of $(\overline{\mathcal{F}}_{\lambda},\omega_{\lambda})$ are of even orders if and only if $\mu=(2a_{1},2a_{2},-2b_{1},\dots,-2b_{p})$.
\end{thm}
\par 
If all the orders of singularities of the differential $\omega_{\lambda}$ are even, then the stratum containing it is usually disconnected (as shown in \cite{Bo}). In such a case, it would be interesting to determine which connected component of the stratum contains the isoresidual fiber.
\smallskip
\par 
Every stratum $\mathcal{H}(\mu)$ of meromorphic differentials in genus zero is connected. However, the generic isoresidual fibers can be disconnected in certain cases. There are two families of such disconnected isoresidual fibers, related to the topological invariants of strata of translation and dilation surfaces of higher genus (see \cite{Bo,ABW}). 
\begin{thm}\label{thm:MAIN5}
In strata of meromorphic one-forms on the Riemann sphere with $n \geq 2$ zeros, the generic isoresidual fibers are connected except for the following two families of strata:
\begin{enumerate}
    \item $\mathcal{H}(ka_{1},\dots,ka_{n},-kb_{1},\dots,-kb_{p-2},-1,-1)$ for some $k \geq 2$ and $a_{1},\dots,a_{n},b_{1},\dots,b_{p-2}$ being positive coprime integers, in which case the generic isoresidual fibers have $k$ connected components;
    \item $\mathcal{H}(2a_{1},\dots,2a_{n},-2b_{1},\dots,-2b_{p-2g},\rec[-1][2g])$ with $g \geq 2$, where generic isoresidual fibers have two connected components.
\end{enumerate}
\end{thm}
\par 
Besides the singularity pattern of $(\overline{\mathcal{F}}_{\lambda},\omega_{\lambda})$, the dependency of the periods of $\omega_{\lambda}$ on the underlying configuration of the residue tuple $\lambda = (\lambda_{1},\ldots,\lambda_{p})$ provides an additional discrete invariant. It is shown in Sections~\ref{subsub:SimplePoles} and~\ref{subsub:HigherPoles} that the poles of $\omega_{\lambda}$ (whose orders are already discrete invariants) have residues that are linear combinations of $\lambda_{1},\dots,\lambda_{p}$ with integer coefficients. These integer coefficients are also discrete invariants.
\par
The same statement holds, in fact, for the period of every relative homology class in $H_{1}(X \setminus P_{\omega},Z_{\omega})$, where $Z_{\omega}$ and $P_{\omega}$ are the sets of the zeros and poles of $\omega$, respectively. It follows that the monodromy of the isoresidual fibration around resonance hyperplanes is encapsulated in a \textit{Gauss--Manin system}, whose fiber is the relative homology group of the isoresidual fiber. The integer coefficients appearing in the formula of the periods form a \textit{period central charge} that commutes with the monodromy of the fibration (see Section~\ref{sec:GaussManin} for details).
\par 
\begin{rmk}
In the case of strata with $n \geq 3$ zeros, relative periods between the zeros still provide convenient local coordinates on isoresidual fibers and define a geometric structure that can serve as a reasonable higher-dimensional analog of a translation surface structure. Similarly to our computation of the Euler characteristic of one-dimensional isoresidual fibers, a geometric structure on higher-dimensional isoresidual fibers can provide a way to compute Chern and Hodge numbers of the (suitably compactified) fibers in terms of local invariants of the singular locus.
\par
As far as we know, there is no universally accepted notion of translation manifolds yet. Polyhedral K\"{a}hler manifolds introduced in \cite{Pan} have good geometric properties, but they are usually defined inductively by gluing polyhedra. In contrast, isoresidual fibers directly arise from complex-analytic data. What is still missing is a natural category of objects that possess the two-faceted quality of translation surfaces: a complex-analytic definition (complex structure with a differential) and an equivalent geometric definition (translation atlas with local models for singularities).
\end{rmk}

\subsection{Organization of the paper}

\begin{itemize}
    \item In Section~\ref{sec:Background}, we review the standard background on translation structures, meromorphic differentials, period coordinates on strata of differentials, and the classification of connected components of strata. We also review known counting formulas for differentials with zero residues and recall the results of \cite{GeTaIso} and \cite{ChPrIso} regarding the isoresidual fibration in the case of strata of differentials in genus zero with a unique zero.

    \item In Section~\ref{sec:multi-scale}, we review the multi-scale compactification of strata and describe the closure of subspaces given by linear equations between the residues. In particular, we describe the boundary  points of the generic isoresidual curves.
    
    \item In Section~\ref{sec:Translation}, we describe the canonical differential on the closure of the generic isoresidual curves whose translation structure is induced by the period atlas. In particular, we detail the local invariants of the singularities of the translation structure of isoresidual fibers in terms of the degeneration of the parametrized objects in the multi-scale compactification. We also provide a more precise description of their geometry in the case of isoresidual fibers over configurations of real residues. These results are summarized in Theorem~\ref{thm:MAIN1}, which is formally proved in Section~\ref{sub:formalMain1}. 
    
    \item In Section~\ref{sec:Euler}, 
    we analyze the wall and chamber structure for the expression of the Euler characteristic of generic isoresidual fibers. In particular, we prove Theorem~\ref{thm:chambers} which shows that the Euler characteristic consists of homogeneous components with respect to the chamber structure described in Definition~\ref{defn:SPS}. 

    \item In Section~\ref{sec:intersection},
    we provide an alternative computation for the Euler characteristic of generic isoresidual fibers through intersection theory on the multi-scale compactification. We  explicitly compute the Euler characteristic in the case where all poles are simple, proving Theorem~\ref{thm:simple}.

    \item In Section~\ref{sec:GaussManin}, we construct a Gauss--Manin system associated with the isoresidual fibration. We discuss the linear dependence of the periods of the translation structure in terms of the underlying configuration of residues, encapsulated in a period central charge, which serves as an arithmetic invariant of the fibers and commutes with the monodromy of the fibration.
    
    \item In Section~\ref{sec:Connected}, we provide a  classification of the connected components of one-dimensional generic isoresidual fibers, using rotation numbers and parity of spin structures. Drawing on the incidence relations between strata, we prove in Theorem~\ref{thm:MAIN5} the complete classification of connected components of generic isoresidual fibers for strata in genus zero with an arbitrary number of zeros.
    
    \item In Section~\ref{sec:Global}, we establish  arithmetic relations between the singularity pattern of a stratum $\mathcal{H}(\mu)$ and the singularity pattern of its generic isoresidual fibers ($\overline{\mathcal{F}}_{\lambda},\omega_\lambda)$, with a particular focus on proving Theorem~\ref{thm:MAIN3}.\newline
\end{itemize}
\par 
\paragraph{\bf Acknowledgements.} Research by D.C. is supported in part by NSF grant DMS-2301030, Simons Travel Support for Mathematicians, and a Simons Fellowship under Record ID SFI-MPS-SFM-00005694. D.C. thanks BIMSA and ICBS2024 for their invitation and hospitality,  where part of the work was completed. Research by Q.G. is supported in part by the Grant PAAPIT UNAM-DFG DA100124 "Conectividad y conectividad simple de los estratos". Research by M.P. is supported by the DFG-UNAM project MO 1884/3-1 and the Collaborative Research Centre TRR 326 ``Geometry and Arithmetic of Uniformized Structures.'' Research by G.T. is supported by the Beijing Natural Science Foundation (Grant IS23005) and the French National Research Agency under the project TIGerS (ANR-24-CE40-3604). The authors thank Gianluca Faraco, Samuel Grushevsky, Myeongjae Lee, Martin M\"oller, and Scott Mullane for inspiring discussions on relevant topics. The authors are also grateful to the anonymous referee for their valuable comments, which have contributed to improving the rigor and clarity of the paper.
\par 
\section{Meromorphic differentials and translation structures}\label{sec:Background}
\par
\subsection{Translation structures}
\par 
The details of the constructions presented in this section and the following one can be found in \cite{Zo,AMTransSurf,BCGGM3}.
\par 
\subsubsection{The regular locus}
\par 
On a compact Riemann surface $X$ endowed with a meromorphic one-form $\omega$, we denote by $X^{\ast}$ the complement of the zeros and poles of $\omega$ in $X$. Local primitives of~$\omega$ are injective around each point of $X^{\ast}$, providing an atlas of $\mathbb{C}$-valued local coordinates. Since local primitives of the same differential differ by a constant, the transition maps of this atlas are translations of the complex plane.
\par
Conversely, such an atlas defines a complex structure on $X^{\ast}$,  and the pullbacks of the one-form $dz$ on $\mathbb{C}$ via local charts globalize to a holomorphic differential $\omega$ on $X^{\ast}$.
\par
In a translation surface, for each slope $\theta$ in the circle $\mathbb{S}^{1}$ of directions, there is a foliation ${\rm Fol}_{\theta}$ of the surface, where the leaves are locally conjugated to oriented lines of slope $\theta$ in the translation charts.
\par 
\subsubsection{Local models for singularities}
\par 
At a zero of order $a$, a differential is locally represented  by pulling back $dz$ via a branched cover of degree $a+1$ of a disk, where the cover is totally ramified over the center of the disk. Consequently, the metric induced by the differential exhibits a conical singularity of angle $(a+1)2\pi$.
\par 
Before presenting local models for the poles of a meromorphic differential, we recall the following convention. The \textit{residue} of a differential $\omega$ at a pole is defined as the period of $\omega$ over a positively oriented simple loop around the pole. This convention differs from the usual one by a factor of $2\pi {\rm i}$. Since our approach emphasizes periods rather than coefficients, this convention is more suitable. It is important to note that the residue is a local invariant of a pole. In particular, for a differential on $\mathbb{CP}^{1}$, if the residue at every pole is real, then the period of any closed loop is also real.
\par 
In a neighborhood of a simple pole, the differential $\omega$ is of the form $\dfrac{rdz}{z}$, where $2\pi {\rm i} r$ is the residue at the pole. Its geometric interpretation is that of a semi-infinite cylinder, where the waist curves have a period equal to the residue $2\pi {\rm i} r$ (see Section 2.2 of \cite{Bo}). 
\par 
A {\em flat cone of type $a\geq0$} is a flat surface associated with the differential $\omega$ on $\PP^{1}$, which has a unique zero of order $a$ and a unique pole of order $a+2$. In the case where $a=0$, the flat cone has no conical singularity and features a unique pole of order $2$, corresponding to the flat plane where the pole of order two is located at infinity.
\par
The neighborhood of a pole of order $b>1$ with trivial residue is the complement of a compact neighborhood around the conical singularity of the flat cone of type $a-2$.
\par
To construct local models for poles of order $b>1$ with nontrivial residue, we start with a flat cone of type $a-2$ and remove an $\epsilon$-neighborhood of a semi-infinite line extending from the conical singularity, along with a neighborhood of the conical singularity itself. We then  identify the resulting boundaries of the neighborhood using an isometry. By rotating and rescaling, we obtain a pole of order $b$ with any nonzero residue. This construction is also explained in detail in~\cite{Bo}.
\smallskip
\par 
\subsubsection{Topological index of a loop in a translation surface}\label{sub:rotation}

Directions are preserved by translations, so the usual winding number can be generalized from the flat plane to translation surfaces. 
\begin{defn}
Let $\gamma$ be a smooth oriented loop in a translation surface punctured at the singularities. Assuming that $\gamma$ is parametrized by arc length, $\frac{\gamma'(t)}{|\gamma'(t)|}$ defines a continuous map from $\mathbb{S}^{1}$ to $\mathbb{S}^{1}$. As such, it has a well-defined \textit{topological index} ${\rm Ind}_{\gamma}$, which depends only on the homotopy class of $\gamma$.
\end{defn}
\par 
The topological index of a positively oriented loop around a singularity of order $a$ is $a+1$.
\par 
\subsection{Period coordinates for strata}\label{sec:periodcoor}
\par 
\subsubsection{The period atlas}
\par 
Denote by $Z_{\omega}$ the set of zeros of $\omega$ and by $P_{\omega}$ the set of poles. Any smooth path $\gamma$ joining two zeros of $\omega$ and avoiding the poles of $\omega$ represents a \textit{relative homology class} $[\gamma]$ in $H_{1}(X \setminus P_{\omega},Z_{\omega})$. The integrals of $\omega$ along an integral basis of $H_{1}(X \setminus P_{\omega},Z_{\omega})$ provide local complex-analytic coordinates for the corresponding stratum of differentials with the same orders of singularities as $\omega$, which are called {\em period coordinates}.
\par
For a stratum $\mathcal{H}(a_{1},\dots,a_{n},-b_{1},\dots,-b_{p})$ with $\sum\limits_{i=1}^n a_{i} - \sum\limits_{j=1}^p b_{j} =2g-2$ and $n,p \geq 1$, the \textit{period atlas} endows the stratum with the structure of a complex-analytic orbifold of complex dimension $2g+n+p-2$.
\par 
\subsubsection{Saddle connections}
\par 
In a translation surface $(X,\omega)$, a \textit{saddle connection} is an arc joining two zeros, whose interior does not contain any singularities and is locally conjugated to a segment of constant slope in the translation charts. In the flat metric induced by $\omega$, a saddle connection is a geodesic segment. A \textit{closed saddle connection} is a saddle connection where the two ends coincide.
\par
Every saddle connection represents a class in the relative homology group $H_{1}(X \setminus P_{\omega}, Z_{\omega})$.
\par 
\subsubsection{Cylinders}
\par 
It is well-known that translation surfaces of finite area (corresponding to holomorphic one-forms) have infinitely many saddle connections. However, meromorphic one-forms can have either finitely or infinitely many saddle connections, depending on the existence of invariant components of finite area in the directional foliation (see Section~5 in \cite{tahar} for details).
\par 
\begin{defn}
In a translation surface $(X,\omega)$, cutting along all the saddle connections sharing a given direction $\theta \in \mathbb{S}^{1}$ decomposes $X$ into connected components that we refer to as \textit{invariant components} because they are invariant under the directional flow in the direction $\theta$.
\end{defn}
\par 
Here, we are only interested in the case of meromorphic one-forms on $\mathbb{CP}^{1}$, whose characterization is simpler.
\par
We recall that in a translation surface, closed geodesics (loops locally conjugated to straight segments of constant slope in the translation charts) that are disjoint from the singularities form one-parameter families known as  \textit{cylinders}. An end of a cylinder can be either a simple pole (in the case of an end at infinity) or consist of a chain of parallel saddle connections.
\par 
\begin{prop}\label{prop:INFINITE-SC-Chara}
For a translation surface $(\mathbb{CP}^{1},\omega)$, where $\omega$ is a meromorphic one-form, one of the following statements holds:
\begin{itemize}
    \item $\omega$ admits finitely many saddle connections;
    \item $\omega$ admits infinitely many saddle connections,  and the accumulation points of the directions of these saddle connections in $\mathbb{S}^{1}$ coincide with the directions of closed geodesics in the cylinders of finite area in $(\mathbb{CP}^{1},\omega)$.
\end{itemize}
\end{prop}
\par 
\begin{proof}
It is proved in Proposition~5.10 of \cite{tahar} that the accumulation points of the directions of saddle connections are the directions of invariant components of the directional flow that have finite area. These components are either cylinders or minimal components. The latter only appears in genus at least one, so it cannot occur for $(\mathbb{CP}^{1},\omega)$. Conversely, in any cylinder of finite area, we can find arbitrarily long saddle connections joining the two boundary components, and their slope approaches the slope of the closed geodesics of the cylinder. 
\par
Finally, having infinitely many saddle connections and infinity many slopes of saddle connections in $\mathbb{S}^{1}$ are equivalent, as there is a topological upper bound on the maximal number of saddle connections sharing a given direction (this is a special case of Proposition~7.6 in \cite{tahar}).
\end{proof}
\par 
In this paper, we primarily focus on translation surfaces of genus zero with exactly $n=2$ conical singularities. In such  surfaces, there is at most one cylinder of finite area.
\par 
\begin{prop}
For a translation surface $(\mathbb{CP}^{1},\omega)$, where $\omega$ is a meromorphic one-form with exactly two zeros, there is at most one cylinder of finite area in $(\mathbb{CP}^{1},\omega)$.
\end{prop}
\par 
\begin{proof}
In genus zero, each closed geodesic of a cylinder of finite area disconnects the surface into two components. Each component contains one boundary component of the cylinder and therefore exactly one conical singularity.
\par
If $(\mathbb{CP}^{1},\omega)$ contains two such cylinders, they cannot be disjoint because $\omega$ would then have at least three zeros. In the case where these two cylinders intersect, we have two closed geodesics $\gamma$ and $\gamma'$, that intersect each other. This situation is also impossible because $\gamma$ decomposes $\mathbb{CP}^{1}$ into two connected components,~$X_{1}$ and $X_{2}$. Since $\gamma'$ is a periodic trajectory, it should cross $\gamma$ from $X_{1}$ to $X_{2}$ and then from $X_{2}$ to $X_{1}$. In particular, if $\gamma'$ enters $X_{2}$ by crossing $\gamma$ positively (resp. negatively), then it has to leave $X_{2}$ by crossing $\gamma$ negatively (positively). However, in a translation surface, the direction of a trajectory cannot change, so all intersections between the two trajectories have the same sign. This is a contradiction so there is no such pair of geodesics $\gamma$ and $\gamma'$
\end{proof}
\par 
\subsubsection{The action of ${\rm GL}_{2}^{+}(\mathbb{R})$}\label{subsub:action}
\par 
Given a translation surface $(X,\omega)$, elements of ${\rm GL}_{2}^{+}(\mathbb{R})$ act by composition with the coordinate functions induced by $\omega$. For a translation surface $(X,\omega)$ obtained by identifying parallel sides of a collection of (possibly infinite) polygons $P_{1},\dots,P_{t}$ in the real plane $\mathbb{R}^{2}$, given $g\in {\rm GL}_{2}^{+}(\mathbb{R})$, the image $g\cdot(X,\omega)$ is obtained by identifying the corresponding sides of the polygons $g(P_i)$,  where $g$ acts as a linear transformation on $\mathbb{R}^{2}$.
\par
Similarly, as $\mathbb{C}$ identifies with $\mathbb{R}^{2}$, the group ${\rm GL}_{2}^{+}(\mathbb{R})$ acts on the residual space $\mathcal{R}_p$. It is easy to check that the action of ${\rm GL}_{2}^{+}(\mathbb{R})$ preserves strata of meromorphic differentials and commutes with the residual map.
\par 
\subsection{Core}\label{sub:core}
\par 
In a translation surface induced by a meromorphic differential, every neighborhood of a pole has infinite area. However, the essential information of the translation structure is contained within a domain of finite area called the \textit{core}.
\par 
\begin{defn}\label{defn:core}
A subset $E$ of a translation surface $(X,\omega)$ is \textit{convex} if and only if every geodesic segment joining two points of $E$ lies entirely in~$E$.
\par
The \textit{convex hull} of a subset $F$ of a translation surface $(X,\omega)$ is the smallest closed convex subset of $X$ that contains~$F$.
\par
The \textit{core} of $(X,\omega)$ is the convex hull $core(X)$ of the zeros of $\omega$.
\end{defn}
\par 
The core separates the poles from one another. The following result demonstrates that the complement of the core has as many connected components as the number of poles (see Proposition~4.4 and Lemma~4.5 of \cite{tahar}).  
\par 
\begin{prop}\label{prop:decomppoles}
For a translation surface $(X,\omega)$, the boundary of the core $\partial\mathcal{C}(X)$ is a finite union of saddle connections. Moreover, each connected component of $X \setminus core(X)$ is a topological disk that contains a unique pole.
\end{prop}
We refer to these connected components as \textit{polar domains}.
\smallskip
\par
\subsection{Graphs for translation surfaces with real periods}\label{sub:RealGraph}
In this section we introduce two graphs associated to a translation surface with real period. Both of them could be useful depending to the context.  First we introduce a version of ribbon graphs suited for the description of the translation surfaces with real periods. Its definition is quite simple.
\begin{defn}
 Given a meromorphic differential $\omega$ with real periods on the Riemann sphere, its \textit{associated graph} is the embedded graph in the sphere whose vertices are its zeros and edges are its saddle connections, oriented in the increasing real direction. Moreover, we draw an half edge for every horizontal half infinite ray starting from a zero of $\omega$.
\end{defn}
Note that each face of the graph corresponds to a pole. Moreover, by convention, we will draw the graph in such a way that the unbounded face corresponds to a pole with positive residue. The order of the zero is half of the number of ray starting at its corresponding vertex minus one. The order of the pole is minus half of the number of half edges contained in the corresponding face minus one.

In the case of a single zero, this graph corresponds to the dual of the graph introduced in \cite{GeTaIso}.
\par
To conclude, the merging zeros of $\omega$ by shrinking a saddle connection corresponds to the shrinking of the corresponding edge in the associated graph.
\smallskip
\par
We now introduce the \textit{decorated graphs} of a translation surface with real periods, which is the dual notion of the ribbon graphs.  This generalizes the concept of \textit{decorated trees} from  \cite{GeTaIso}, to provide a combinatorial description of these translation surfaces. 
\par 
\begin{defn}\label{defn:abstractdecoratedtgraph}
A \textit{decorated graph} is an embedded graph in the topological sphere such that:
\begin{itemize}
    \item every face is labeled;
    \item every face is a topological disk;
    \item every vertex is labeled;
    \item edges are oriented;
    \item to every vertex is attached a nonnegative even number of unoriented half-edges;
    \item at a vertex, there is a nonnegative even number of half-edges between two adjacent edges with the same orientation and an odd number of half-edges between two edges of opposite orientation.
\end{itemize}
\end{defn}
\par 
Now, we show how to construct a decorated graph corresponding to a translation surface of genus zero with real periods.
\par
For any one-form $\omega$ on $\mathbb{CP}^{1}$ in a stratum $\mathcal{H}(a_{1},\dots,a_{n},-b_{1},\dots,-b_{p})$ with real periods, all the saddle connections are horizontal. Consequently, the \textit{vertical directional foliation} induced by $\omega$ on $\mathbb{CP}^{1}$ is straightforward to describe (see Proposition~5.5 in \cite{tahar} for more details). Vertical trajectories (oriented from bottom to top) can be:
\begin{itemize}
    \item \textit{Generic trajectories}: These are infinitely long and extend from one pole to another. 
    \item One of the $\sum\limits_{i=1}^n (a_{i}+1)$ \textit{critical trajectories} that go from a pole to a zero. 
    \item One of the $\sum\limits_{i=1}^n (a_{i}+1)$ \textit{critical trajectories} that go from a zero to a pole.
\end{itemize}
\par 
Generic trajectories assemble into one-parameter families that  sweep through three kinds of subsurfaces:
\begin{itemize}
    \item \textit{Open left half-planes} with a unique conical singularity on the right vertical boundary. 
    \item \textit{Open right half-planes} with a unique conical singularity on the left vertical boundary. 
    \item \textit{Infinite vertical strips} with a unique conical singularity at the same height on each boundary line.
\end{itemize}
\par 
We construct the \textit{decorated graph} $\mathfrak{gr}(\omega)$ in the following way:
\begin{itemize}
    \item The vertices are the poles of $\omega$. 
    \item For each infinite vertical strip, we draw one of the vertical trajectories (oriented from bottom to top). 
    \item For each open left or right half-plane, we draw a half-edge attached to the corresponding vertex.
\end{itemize}
\par 
We can immediately verify that:
\begin{itemize}
    \item There is a complete correspondence between the oriented edges of $\mathfrak{gr}(\omega)$, infinite vertical strips, and the horizontal saddle connections joining their sides. 
    \item Each face of $\mathfrak{gr}(\omega)$ is a topological disk containing exactly one zero of $\omega$. 
    \item To each pole of order $b_{j}$, there are $2b_{j}-2$ half-edges attached (half of them correspond to left half-planes, while the other half correspond to right half-planes). 
    \item The gluing of the boundaries of half-planes and strips is consistent with the orientation of the half-planes. Hence, there is a nonnegative even number of half-edges between two adjacent edges with the same orientation (at the vertex), and an odd number of half-edges between two edges of opposite orientation.
\end{itemize}

In a translation surface with only real periods, saddle connections are horizontal and can only meet at their endpoints. The union of the saddle connections cuts out $p$ polar domains, each of which is a topological disk containing a unique pole (see, for example, Lemma~4.10 of \cite{tahar}). The computation of the Euler characteristic formula shows that there are $n+p-2$ saddle connections, as well as the same number of infinite vertical strips. Each of these contains four sides of critical trajectories.
\par
Since the local model of a pole of order $b_{j}>1$ is the cyclic gluing of $b_{j}-1$ planes, there are exactly $\sum\limits_{j=1}^p (2b_{j}-2)$ left or right half-planes, each containing two sides of critical trajectories. Therefore, the total number of critical trajectories is $2n-4+ 2\sum\limits_{j=1}^p b_{j}$.
\par
Since a zero of order $a_{i}$ corresponds to a conical singularity with an angle of $(2a_{i}+2)\pi$, there are $a_{i}+1$ incoming and $a_{i}+1$ outgoing critical trajectories at this zero. Therefore, the total number of critical trajectories is equal to $2n+2\sum\limits_{i=1}^n a_{i}$. This is consistent with the identity $\sum\limits_{i=1}^n a_{i} - \sum\limits_{j=1}^p b_{j} = -2$.

\begin{ex}
A differential in $\calH(2,1,\rec[-1][5])$ is represented on the left side of Figure~\ref{fig:exgraph}. The associated ribbon and decorated graphs  are shown on the right side of the same figure, respectively on the top and on the bottom. Similarly, a differential in $\mathcal{H}(1,4,-1,-2,-4)$ and its associated graphs are presented in Figure~\ref{fig:exgraph2}.
\begin{figure}[ht]
 \centering
 \begin{tikzpicture}[,decoration={
    markings,
    mark=at position 0.4 with {\arrow[very thick]{>}}}]
\begin{scope}[]
\coordinate (a) at (-1,1);
\coordinate (b) at (0,1);

    \fill[fill=black!10] (a)  -- (b)coordinate[pos=.5](f) -- ++(0,-1) --++(-1,0) -- cycle;
 \draw  (a) -- (b);
 \draw (a) -- ++(0,-.9) coordinate (d)coordinate[pos=.5](h);
 \draw (b) -- ++(0,-.9) coordinate (e)coordinate[pos=.5](i);
 \draw[dotted] (d) -- ++(0,-.2);
 \draw[dotted] (e) -- ++(0,-.2);
\node[below] at (f) {$1$};
    \fill (a)  circle (1.5pt);
\fill[] (b) circle (1.5pt);
\node at (-.5,0) {$p_{2}$};
\end{scope}

\begin{scope}[xshift=1.2cm]
\coordinate (a) at (-1,1);
\coordinate (b) at (0,1);

    \fill[fill=black!10] (a)  -- (b)coordinate[pos=.5](f) -- ++(0,-1) --++(-1,0) -- cycle;
 \draw  (a) -- (b);
 \draw (a) -- ++(0,-.9) coordinate (d)coordinate[pos=.5](h);
 \draw (b) -- ++(0,-.9) coordinate (e)coordinate[pos=.5](i);
 \draw[dotted] (d) -- ++(0,-.2);
 \draw[dotted] (e) -- ++(0,-.2);
\node[below] at (f) {$2$};
    \fill (a)  circle (1.5pt);
\fill[] (b) circle (1.5pt);
\node at (-.5,0) {$p_{3}$};
\end{scope}

\begin{scope}[xshift=2.4cm]
\coordinate (a) at (-1,1);
\coordinate (b) at (0,1);

    \fill[fill=black!10] (a)  -- (b)coordinate[pos=.5](f) -- ++(0,-1) --++(-1,0) -- cycle;
 \draw  (a) -- (b) coordinate[pos=.2](p1)coordinate[pos=.1](q1)coordinate[pos=.6](q2);
 \draw (a) -- ++(0,-.9) coordinate (d)coordinate[pos=.5](h);
 \draw (b) -- ++(0,-.9) coordinate (e)coordinate[pos=.5](i);
 \draw[dotted] (d) -- ++(0,-.2);
 \draw[dotted] (e) -- ++(0,-.2);
\node[below] at (q1) {$3$};
\node[below] at (q2) {$3'$};
    \fill (a)  circle (1.5pt);
\fill[] (b) circle (1.5pt);
    \filldraw[fill=white] (p1)  circle (1.5pt);

    \node at (-.5,0) {$p_{4}$};
\end{scope}

\begin{scope}[xshift=3.6cm]
\coordinate (a) at (-1,1);
\coordinate (b) at (0,1);

    \fill[fill=black!10] (a)  -- (b)coordinate[pos=.5](f) -- ++(0,-1) --++(-1,0) -- cycle;
 \draw  (a) -- (b);
 \draw (a) -- ++(0,-.9) coordinate (d)coordinate[pos=.5](h);
 \draw (b) -- ++(0,-.9) coordinate (e)coordinate[pos=.5](i);
 \draw[dotted] (d) -- ++(0,-.2);
 \draw[dotted] (e) -- ++(0,-.2);
\node[below] at (f) {$4$};
    \filldraw[fill=white] (a)  circle (1.5pt);
    \filldraw[fill=white] (b) circle (1.5pt);

    \node at (-.5,0) {$p_{5}$};
\end{scope}

\begin{scope}[xshift=1.3cm,yshift=.4cm]
\coordinate (a) at (-2,1);
\coordinate (b) at (2,1);

    \fill[fill=black!10] (a)  -- (b)-- ++(0,1) --++(-4,0) -- cycle;
 \draw  (a) -- (b) coordinate[pos=.25](p1)coordinate[pos=.5](p2)coordinate[pos=.55](p3)coordinate[pos=.8](p4) coordinate[pos=.125](q1)coordinate[pos=.375](q2)coordinate[pos=.525](q3)coordinate[pos=.675](q4) coordinate[pos=.9](q5);
 \draw (a) -- ++(0,.9) coordinate (d)coordinate[pos=.5](h);
 \draw (b) -- ++(0,.9) coordinate (e)coordinate[pos=.5](i);
 \draw[dotted] (d) -- ++(0,.2);
 \draw[dotted] (e) -- ++(0,.2);
\node[above] at (q1) {$1$};
\node[above] at (q2) {$2$};
\node[above] at (q3) {$3$};
\node[above] at (q4) {$4$};
\node[above] at (q5) {$3'$};
    \fill (a)  circle (1.5pt);
\fill[] (b) circle (1.5pt);
    \fill (p1)  circle (1.5pt);
\fill[] (p2) circle (1.5pt);
    \filldraw[fill=white] (p3)  circle (1.5pt);
\filldraw[fill=white] (p4) circle (1.5pt);
\node at (0,2) {$p_{1}$};
\end{scope}

\begin{scope}[xshift=8cm,yshift=2.1cm]
\coordinate (a) at (-1,0);
\coordinate (b) at (1,0);

\draw (a)[postaction={decorate}]  .. controls ++(-110:1.4) and ++(-170:1.4) ..node[left]{$1$}  (a);
\draw (a)[postaction={decorate}]  .. controls ++(170:1.4) and ++(110:1.4) ..node[left]{$2$}  (a);
\draw (a)[postaction={decorate}]  .. controls ++(60:1.4) and ++(120:1.4) ..node[above]{$3$}  (b);
\draw (b)[postaction={decorate}]   .. controls ++(-30:1.4) and ++(30:1.4) .. node[right]{$4$} (b);
\draw (b)[postaction={decorate}]  .. controls ++(-120:1.4) and ++(-60:1.4) ..node[below]{$3'$}  (a);

\fill[] (a) circle (2pt);
\filldraw[fill=white] (b)  circle (2pt);
\end{scope}

\begin{scope}[xshift=8cm,yshift=-.8cm]
 \node[circle,draw] (a) at  (-2,1) {$p_{3}$};
\node[circle,draw] (b) at  (0,0) {$p_{1}$};
\node[circle,draw] (c) at  (-2,-1) {$p_{2}$};
 \node[circle,draw] (d) at  (1.5,0) {$p_{5}$};
  \node[circle,draw] (e) at  (3,0) {$p_{4}$};
\draw [postaction={decorate}]  (a) --node[above right] {$2$}  (b);
\draw[postaction={decorate}] (c) --node[below right] {$1$} (b);
\draw[postaction={decorate}] (d) --node[below] {$4$} (b);
\draw[postaction={decorate}] (e)  .. controls ++(120:1.4) and
++(60:1.4) ..node[below] {$3$}  (b);
\draw[postaction={decorate}] (e)  .. controls ++(-120:1.4) and
++(-60:1.4) ..node[below left] {$3'$}  (b);

\fill[] (-1,0) circle (1.5pt);
    \filldraw[fill=white] (2.25,0)  circle (1.5pt);
\end{scope}
\end{tikzpicture}
    \caption{A differential $\omega$ and the associated ribbon and decorated graphs  are shown, with extra labels on the edges given by the label of the corresponding saddle connection.}\label{fig:exgraph}
\end{figure}
\smallskip
\par
 \begin{figure}[ht]
\begin{tikzpicture}[scale=1,decoration={
    markings,
    mark=at position 0.4 with {\arrow[very thick]{>}}}]

    \begin{scope}[xshift=-.5cm,yshift=.5cm]
    \draw (-.25,0) coordinate (a) -- (.25,0) coordinate[pos=.5] (c) coordinate (b);

    \fill[fill=black!10] (a) -- (b) -- ++(0,-1) --++(-.5,0) -- cycle;
    \draw (a) -- (b);
    \draw  (a) -- ++(0,-1)  coordinate[pos=.5] (d);
    \draw  (b) -- ++(0,-1)  coordinate[pos=.5] (e);
\filldraw[fill=white](a) circle (2pt);
\filldraw[fill=white] (b) circle (2pt);

 \node[above] at (c) {$1$};
\node[left] at (d) {$5$};
\node[right] at (e) {$5$};

\node at (0,-1) {$p_{3}$};
    \end{scope}

\begin{scope}[xshift=2.3cm,yshift=0cm]
      \fill[fill=black!10] (0,0) circle (1cm);
      \draw (-.5,0) coordinate (a) -- (.5,0) coordinate[pos=.25] (c) coordinate[pos=.5] (d) coordinate[pos=.75] (e) coordinate (b);
\fill (a) circle (2pt);
\fill (b) circle (2pt);
\filldraw[fill=white] (d) circle (2pt);
\node[below] at (c) {$2$};
\node[below] at (e) {$2'$};
    \fill[white] (a) -- (b) -- ++(0,1) --++(-1,0) -- cycle;
 \draw  (b) --  (a);
 \draw (a) -- ++(0,.9) coordinate[pos=.5] (d);
 \draw (b) -- ++(0,.9)coordinate[pos=.5] (e);
\node[left] at (d) {$4$};
\node[right] at (e) {$4$};

\node at (0,-1) {$p_{2}$};
    \end{scope}

    \begin{scope}[yshift=3cm]
      \fill[fill=black!10] (0,0) circle (1.5cm);
 \draw (-1,0) coordinate (a) -- (.5,0) coordinate[pos=.17] (c) coordinate[pos=.5] (d) coordinate[pos=.83] (e) coordinate[pos=.33] (f) coordinate[pos=.66] (g) coordinate (b);
 \node[above] at (c) {$2$};
  \node[above] at (d) {$1$};
   \node[above] at (e) {$2'$};
  \fill[] (a) circle (2pt);
\fill[] (b) circle (2pt);
\filldraw[fill=white] (f) circle (2pt);
\filldraw[fill=white] (g) circle (2pt);
    \fill[white] (a) -- (b) -- ++(0,-1.5) --++(-1.5,0) -- cycle;
 \draw  (b) --  (a);
 \draw (a) -- ++(0,-1.4) coordinate[pos=.5] (f);
 \draw (b) -- ++(0,-1.4)coordinate[pos=.5] (g);
  \draw  (b) --++  (1,0) coordinate[pos=.5] (h) ;

\node[left] at (f) {$5$};
\node[right] at (g) {$5$};
\node[above] at (h) {$6$};\node[below] at (h) {$8$};

       \fill[fill=black!10] (2.5,.75) circle (.7cm);
  \fill (2.5,.75) circle (2pt);
  \draw (2.5,.75) --++(.7,0)  coordinate[pos=.5] (i) ;
\node[below] at (i) {$6$};\node[above] at (i) {$7$};

       \fill[fill=black!10] (2.5,-.75) circle (.7cm);
  \fill (2.5,-.75) circle (2pt);
  \draw (2.5,-.75) --++(.7,0) coordinate[pos=.5] (j) ;
  \node[below] at (j) {$7$};\node[above] at (j) {$8$};

  \node at (0,1) {$p_{1}$};
    \end{scope}
\begin{scope}[xshift=7.2cm,yshift=3cm]
\coordinate (a) at (-1,0);
\coordinate (b) at (1,0);

\foreach \i in {0,1,...,5}
\draw (a) --++(90+36*\i:.3);
\draw (a) [postaction={decorate}]  .. controls ++(60:1.4) and ++(120:1.4) .. node[above]{$2$}  (b);
\draw (b)[postaction={decorate}]   .. controls ++(30:1.4) and ++(-30:1.4) ..node[right]{$1$}   (b);
\draw (b)[postaction={decorate}]   .. controls ++(-120:1.4) and ++(-60:1.4) .. node[below]{$2'$}  (a);
\foreach \i in {0,1}
\draw (a) --++(-15+30*\i:.3);

\fill[] (a) circle (2pt);
\filldraw[fill=white] (b)  circle (2pt);
\end{scope}

\begin{scope}[xshift=6.5cm,yshift=.3cm]
\node[circle,draw] (b) at  (0,0) {$p_{1}$};
 \node[circle,draw] (d) at  (1.5,0) {$p_{3}$};
  \node[circle,draw] (e) at  (3,0) {$p_{2}$};
\draw[postaction={decorate}] (d) -- (b);
\draw[postaction={decorate}] (e)  .. controls ++(120:1.4) and
++(60:1.4) .. (b);
\draw[postaction={decorate}] (e)  .. controls ++(-120:1.4) and
++(-60:1.4) ..  (b);
\draw (e) --++(60:.5);
\draw (e) --++(-60:.5);
   \foreach \i in {0,1,...,5}
\draw (b) --++(90+36*\i:.5);
  
\fill[] (-1,0) circle (1.5pt);
    \filldraw[fill=white] (2.25,0)  circle (1.5pt);
\end{scope}
\end{tikzpicture}
 \caption{A differential in $\mathcal{H}(1,4,-1,-2,-4)$ and its associated graphs are presented.} \label{fig:exgraph2}
\end{figure}
\end{ex}
\par 
To conclude, we describe the effect of merging zeros of $\omega$ in terms of the associated graph $\mathfrak{gr}(\omega)$.
\begin{lem}\label{lem:defengraph}
 Given a differential $\omega\in\mathcal{H}(a_{1},a_{2},-b_{1},\dots,-b_{p})$ with real periods and its associated graph $\mathfrak{gr}(\omega)$, if we shrink a saddle connection between two zeros, the associated graph of the limit is obtained by eliminating the corresponding edge of $\mathfrak{gr}(\omega)$.  
\end{lem}
\par 
\begin{proof}
The decorated graph $\mathfrak{gr}(\omega)$ has $p$ vertices and $p$ edges,  and therefore a unique cycle that decomposes the underlying sphere into two connected components. The edges of the cycle correspond to the saddle connections joining the two zeros. Shrinking one of these saddle connections produces a translation surface $\omega_{0}$ with one unique zero and the $p-1$ remaining saddle connections. Then, $\mathfrak{gr}(\omega_{0})$ is a decorated tree where edges connect the same polar domains as in $\mathfrak{gr}(\omega)$. In other words,  $\mathfrak{gr}(\omega_{0})$ is obtained from $\mathfrak{gr}(\omega)$ by removing the edge corresponding to the saddle connection we have shrunk.
\end{proof}

\par 
\subsection{Classification of the connected components of strata}\label{sub:CCClassification}
\par 
A complete classification of connected components of strata of meromorphic differentials with prescribed orders of singularities has been provided in \cite{Bo}. Note that in the case of genus zero, every stratum of meromorphic differentials on $\mathbb C\mathbb P^1$ is connected.
\par
We first review the classification in genus one in terms of the rotation number of a translation surface (see Section~\ref{sub:rotation} for the definition of rotation numbers).
\par 
\begin{defn}
Let $(X,\omega)$ be a translation surface in a stratum $\mathcal{H}(a_{1},\dots,a_{n},-b_{1},\dots,-b_{p})$ of differentials in genus one, with $n,p \geq 1$. Let $(\alpha,\beta)$ be two simple loops forming a symplectic basis of the homology of the underlying torus. The \textit{rotation number} of $(X,\omega)$ is defined as  $\gcd(a_{1},\dots,a_{n},b_{1},\dots,b_{p},{\rm Ind}_{\alpha},{\rm Ind}_{\beta})$.
\end{defn}
\par 
\begin{thm}\label{thm:BoissyGenusOne}
Let $\mathcal{H}=\mathcal{H}(a_{1},\dots,a_{n},-b_{1},\dots,-b_{p})$ be a stratum of translation surfaces of genus one. In each connected component of $\mathcal{H}$, the rotation numbers of the translation surfaces are the same. Conversely, there is exactly one connected component for each positive divisor of $\gcd(a_{1},\dots,a_{n},b_{1},\dots,b_{p})$, with one exception: in the stratum $\mathcal{H}(a,-a)$ (where $n=p=1$),  there is no connected component corresponding to the rotation number $a$.
\end{thm}
\par 
For strata parameterizing translation surfaces of genus $g \geq 2$, their connected components are distinguished by \textit{hyperellipticity} and \textit{parity of spin structures}.
\par 
\begin{defn}
For a stratum $\mathcal{H}(\mu)$ of translation surfaces of genus $g \geq 2$, the signature $\mu$ is said to be:
\begin{itemize}
    \item \textit{hyperelliptic type} if $\mu$ is of the form $(a,a,-b,-b)$, $(2a,-b,-b)$, $(a,a,-2b)$, or $(a,a,-b,-b)$,  where~$a$ and $b$ are positive integers;
    \item \textit{even type} if $\mu$ is either of the form $(a_{1},\dots,a_{n},-b_{1},\dots,-b_{p})$ or   $(a_{1},\dots,a_{n},-1,-1)$, where the orders $a_{1},\dots,a_{n},b_{1},\dots,b_{p}$ are even.
\end{itemize}
\end{defn}
\par 
\begin{defn}
A stratum component $\mathcal{C}$ is said to be  \textit{hyperelliptic} if every translation surface $(X,\omega)$ in $\mathcal{C}$ admits a nontrivial involution $\tau$ of $X$ such that $\tau^{\ast} \omega = -\omega$ and $X / \tau$ is isomorphic to $\mathbb{CP}^{1}$.    
\end{defn}
\par 
\begin{defn}
In a stratum $\mathcal{H}(\mu)$ where $\mu$ is of even type, the \textit{parity of spin structure} of a translation surface $(X,\omega)$ of genus $g \geq 2$ is defined as the parity of $\sum\limits_{i=1}^{g} ({\rm Ind}_{\alpha_{i}}+1)({\rm Ind}_{\beta_{i}}+1)$, where $(\alpha_{1},\beta_{1},\dots,\alpha_{g},\beta_{g})$ is a symplectic basis of the homology of the underlying topological surface of genus $g$. This parity is a topological invariant of each connected component of $\mathcal{H}(\mu)$.    
\end{defn}
\par 
\begin{thm}\label{thm:BoissyHigherGenus}
Let $\mathcal{H}(\mu)=\mathcal{H}(a_{1},\dots,a_{n},-b_{1},\dots,-b_{p})$ be a stratum of translation surfaces of genus $g \geq 2$. We distinguish several cases as follows:\newline
If $\sum\limits_{j=1}^p b_{j}$ is odd, then $\mathcal{H}(\mu)$ is connected.\newline
If $\sum\limits_{j=1}^p b_{j} =2$ and $g=2$, then:
\begin{itemize}
    \item if $\mu$ is of hyperelliptic type, then
there are two components: one is hyperelliptic while the other is not (in this case, these two components are also distinguished by the parity of the spin structure);
    \item if $\mu$ is not of hyperelliptic type, then $\mathcal{H}(\mu)$ is connected.
\end{itemize}
If $\sum\limits_{j=1}^p b_{j} >2$ or $g>2$, then:
\begin{itemize}
    \item if $\mu$ is of hyperelliptic type, then $\mathcal{H}(\mu)$ has exactly one hyperelliptic  component and either one or two non-hyperelliptic components, depending on whether $\mu$ is of even type or not;   
    \item if $\mu$ is of even type, then $\mathcal{H}(\mu)$ has two non-hyperelliptic components distinguished by the parity of the spin structure;
    \item if $\mu$ is neither of hyperelliptic nor even type, then $\mathcal{H}(\mu)$ is connected.
\end{itemize}
\end{thm}
\par 
\subsection{Differentials with zero residues}\label{sub:counting}

For our later computations, we need to count the number of meromorphic differentials whose residues vanish.
\begin{defn}\label{def:residuless}
 The number of meromorphic differentials in $\calH(a_{1},a_{2},-b_{1},\dots,-b_{p})$
whose residues at the poles are all zero is denoted by $\Xi(a_{1},a_{2};b_{1},\dots,b_{p})$.
\end{defn}
 This number was initially studied in \cite{EMZ}, and an expression in terms of coefficient extraction from a generating series was given in Proposition~2.1 of \cite{CMSZ} as
$$(p-1)!\ \big[t^{a_1+1}\big]\prod_{i=1}^p \big(t + \cdots + t^{b_i-1}\big)\,,$$
where $\big[t^{j}\big]P(t)$ denotes the coefficient of degree $j$ of the polynomial $P$.
Equivalently, up to the factor $(p-1)!$, the above formula  
counts the number of tuples $(c_1, \ldots, c_p)$ such that 
$\sum\limits_{i=1}^p c_i = a_1 + 1$ and $1\leq c_i\leq b_i-1$ for all $i$. 
\par 
In the following, we prove that $\Xi$ has a piecewise polynomial structure, where the walls that separate the polynomial chambers are defined by identities of the form $a_1+1-b_{i_1}-\cdots -b_{i_k}-(p-k)= 0$.
\par 
\begin{prop}\label{prop:residueless}
For any positive integers $a_{1},a_{2},b_{1},\dots,b_{p}$ such that $\sum\limits_{j=1}^p b_j = a_{1}+a_{2}+2$, the number $\Xi(a_{1},a_{2};b_{1},\dots,b_{p})$ is given by the following piecewise polynomial formula:
\begin{equation}
    \label{eqn:residueless}
    (p-1)!\sum_{k=0}^{p-1} (-1)^k \sum_{a_1+1-b_{i_1}-\cdots -b_{i_k}-(p-k)\geq 0} \binom{a_1-b_{i_1}-\cdots-b_{i_k}+k}{p-1}\,. 
\end{equation}
In particular, $\Xi$ is a polynomial of degree $p-1$ in every chamber.
\end{prop}
\par 
\begin{proof}
To obtain the coefficient of $\big[t^{a_1+1}\big]\prod_{i=1}^p \big(t + \cdots + t^{b_i-1}\big)$, note the inclusion-exclusion relation: 
\begin{equation*}
     \big[t^{a_1+1}\big]\prod_{i=1}^p \big(t + \cdots + t^{b_i-1}\big) = \sum_{k=0}^{p-1} (-1)^k  \sum_{a_{1}+1-b_{i_1}-\cdots-b_{i_k}-(p-k)\geq 0} \big[t^{a_1+1}\big]\left(\prod_{j=1}^k \frac{t^{b_{i_j}}}{1-t}\prod_{i\not= i_j}^p\frac{t}{1-t}\right)\,.
\end{equation*}
\par 
The right-hand side first counts the coefficient of $\big[t^{a_1+1}\big]\prod_{i=1}^p \frac{t}{1-t}$ whenever $a_1+1-p\geq 0$. However, it can occur that we took more than $b_i-1$ copies of $t$ in the factor $(t+\cdots+t^{b_{i}-1}+\cdots)$ for $k$ different factors corresponding to $b_{i_1},\ldots, b_{i_k}$. This happens whenever $a_1+1-(b_{i_1}-1)-\cdots-(b_{i_k}-1)-p\geq 0$, and the corresponding extra coefficient is given by $\big[t^{a_1+1}\big]\left(\prod_{j=1}^k \frac{t^{b_{i_j}}}{1-t}\prod_{i\not= i_j}^p\frac{t}{1-t}\right)$. However, this case was also subtracted for every proper subset of $\{b_{i_1},\ldots, b_{i_k}\}$, so the number of times we have  already subtracted it is $\sum_{i=0}^{k-1}(-1)^i\binom{k}{i}=(-1)^{k-1}\binom{k}{k}$. Therefore, we need to add an extra term with coefficient $(-1)^k$ to ensure that this case is completely subtracted from the total sum. 
\par 
Finally, observe that
\begin{equation*}\label{eq:Counting}
    \big[t^{a_1+1}\big]\left(\prod_{j=1}^k \frac{t^{b_{i_j}}}{1-t}\prod_{i\not=i_j}^p\frac{t}{1-t}\right) = \binom{a_1-b_{i_1}-\cdots - b_{i_k}+k}{p-1}\,.
\end{equation*}
This expression counts the number of ways we can distribute $a_1+1$ objects into $p$ distinct boxes, where the boxes indexed by $b_{i_1},\ldots, b_{i_{k}}$ must contain at least $b_{i_j}$ objects respectively, and the $p-k$ remaining boxes must contain at least one object each.
\end{proof}
\par 
\subsection{The isoresidual fibration of the minimal strata}\label{sub:Minimal}
\par 
For the strata of meromorphic differentials on $\mathbb{CP}^{1}$ with a unique zero, the isoresidual fibration is a ramified cover. Its degree and ramification locus are described in the following statement, which is proved as Theorem~1.2 in \cite{GeTaIso}. 
\par 
\begin{thm}\label{thm:numerofibre}
For every stratum $\mathcal{H}(a,-b_{1},\dots,-b_{p})$ of meromorphic differentials on $\mathbb{CP}^{1}$, the isoresidual fibration is an unramified cover of degree $f(a,p):=\frac{a!}{(a+2-p)!}$ over $\mathcal{R}_{p} \setminus \mathcal{A}_{p}$.
\end{thm}
\par 
Elements of the generic isoresidual fibers of the strata $\mathcal{H}(a,-b_{1},\dots,-b_{p})$ over configurations of real residues are classified by the \textit{decorated graphs} introduced in Section~\ref{sub:RealGraph}. Since the corresponding translation surfaces have exactly $p-1$ saddle connections, these graphs are decorated trees.
\par 
\section{Multi-scale compactification of the isoresidual loci}\label{sec:multi-scale}
\par 
In this section, we recall the basics of the multi-scale compactification introduced in \cite{BCGGM3} and describe the closure of the loci of differentials with residues satisfying some linear relations. Moreover, we describe some properties for multi-scale differentials contained in the closure of the isoresidual loci within this compactification. Since the main Theorem~\ref{thm:closureMSD} of this section applies to any genus, we denote by $\mathcal H_{g}(\mu)$ a stratum of differentials of genus $g\geq0$ with signature $\mu$ when the results apply to such general cases.
\par 
\subsection{The multi-scale compactification of strata}\label{sub:multi-scale}
\par 
The multi-scale compactification $\MS(\mu)$ of the projectivized strata $\PP\mathcal H_{g}(\mu) := \mathcal H_{g}(\mu) / \mathbb C^*$ is constructed in \cite{BCGGM2}. Briefly speaking, a {\em multi-scale differential} $(X,\omega)$ is defined on an underlying stable pointed nodal Riemann surface $X$, where $\omega$ consists of a non-identically-zero differential $\omega_{i}$ on each irreducible component $X_{i}$ of $X$. Moreover, there is a total order that compares any two irreducible components of $X$, which encodes the information about vanishing rates when differentials from nearby smooth surfaces degenerate to these sub-surfaces $X_i$. In this sense, the total order induces a level structure on the dual graph of the nodal surface, which is called a {\em level graph}. Additionally, there are compatibility conditions that relate the zero and pole orders of the $\omega_{i}$ at the two branches of every node, as well as local and global residue conditions at the nodes, which we will review below. 

If a node has two simple poles at its branches, the corresponding edge in the dual graph is called {\em horizontal}; otherwise, it is called {\em vertical}. At a vertical edge, the multi-scale differential has a zero of order $k\geq0$ at the upper nodal point and a pole of order $-k-2\leq-2$ at the lower nodal point.
Moreover, the number $k+1$ is called the {\em prong number} of the edge. It encodes the number of ways to locally open up the node under the induced flat metric. 
\par
Next, we discuss the global residue condition. Consider a level $L$ and a component $Y$ of the part $\Gamma_{>L}$ of sub-surfaces lying strictly above $L$ in the level graph $\Gamma$. We consider the edges that connect $Y$ to the vertices of $\Gamma$ at level $L$, and denote by $e_{1},\dots,e_{b}$ the (lower) endpoints of these edges. 
If $Y$ contains a marked pole, we do not impose any condition to $Y$. Now, suppose $Y$ does {\em not} contain any marked poles. Then, we say that $Y$ satisfies the (usual) {\em global residue condition} if the following condition holds:
\begin{itemize}
 \item[{\bf GRC}:]  $\sum_{i=1}^{b}\Res_{e_{i}}(\omega) = 0$. 
\end{itemize}
We remark that the residue theorem for each vertex (i.e., the total sum of residues on a component is equal to zero) still needs to hold, which we implicitly impose as a preexisting condition.    
\par
Finally, the multi-scale compactification has a system of local coordinates, which extends the period coordinates recalled in Section~\ref{sec:periodcoor} to the boundary. These coordinates are called perturbed period coordinates and a variation of them corresponds to log period coordinates. These are defined and discussed in Sections~9.2 and~13.3 of \cite{BCGGM2} and Section~5 of \cite{BenBound}, respectively. These coordinates consist of two parts: one part is given by (usual) relative periods on the nodal curve and the other part is given by the parameters used to smooth the horizontal nodes and the nodes crossing a level of the multi-scale differential.

\subsection{The closure of loci of differentials with linear residue conditions}\label{sub:closure}

One can consider a similar residual map from the multi-scale compactification $\MS(\mu)$ to the projective residue space $\mathbb{P} \mathcal{R}_{p} = (\mathcal{R}_{p} \setminus \{ 0 \}) / \mathbb{C}^{*}$. Note that, in general, it is only a rational map; for example, it is undefined on the locus of residueless differentials. Indeed, in Corollary~\ref{cor:reszeros} below, we will show that in our case the closure of every isoresidual curve contains the zero-dimensional locus of residueless differentials.
\par
Following \cite[Section 4.1]{CMZ}, we now describe the closure of the isoresidual fibers in the moduli space of multi-scale differentials. This discussion can generally be applied to arbitrary genera and partition~$\mu$. Additionally, we treat the more general setting where we fix a linear subspace of the residues (i.e., not just a point in the projectivized residue space by fixing all residues).
\par 
Consider a stratum $\PP\mathcal H_{g}(\mu)$ of meromorphic differentials of genus $g$ with signature $\mu$. Let $\Lambda$ be a linear subspace in the residual space $\mathcal{R}_{p}$, and consider the subspace $\mathcal{F}_{\Lambda}$ of differentials whose residues lie in $\Lambda$. We define ${\Lambda}^{\vee}$ in the dual vector space $\mathcal{R}_{p}^{\vee}$ of $\mathcal{R}_{p}$ to be the vector space of homogeneous linear equations satisfied by all residue tuples in $\Lambda$.

We want to characterize when a multi-scale differential $(X,\omega)$ is contained in the closure of $\mathcal F_\Lambda$. To this aim, we will state the {\em generalized global residue condition} imposed to $(X,\omega)$ by $\Lambda^{\vee}$, called $\mathcal{E}_{\Lambda}$-GRC. 
Denote by $q_1,\ldots, q_p$ the marked poles in $\Gamma$. For each $q_i$, add a new vertex at level $\infty$ with a marked pole $q'_i$ 
and turn the original $q_i$ into an edge that connects to the new vertex. One can think of this new vertex as a semistable rational component with a pole at $q'_i$ of order equal to that of $q_i$. Denote by $\Gamma'$ the resulting level graph. For every level $L < \infty$ in $\Gamma'$, let $Y_1,\ldots, Y_s$ be the connected components of $\Gamma'_{>L}$. 
Then, the following conditions hold: 

\begin{itemize}
 \item[{\bf $\mathcal{E}_{\Lambda}$-GRC}:] 
\begin{enumerate}
\item Each $Y_i$ that does not contains any $q'_{j}$ satisfies the usual GRC. 

\item For every equation $f\in \Lambda^{\vee}$ that can be written of the form 
$$f = \sum_{i=1}^s a_i \left(\sum_{q'_j\in Y_i} {\rm Res}_{q'_j}\right)\,,$$ 
we require that 
the equation
$$\sum_{i=1}^s a_i \left(\sum_{e_j \in Y_{i, L}} \Res_{e_j} (\omega)\right) = 0$$
holds, where the inner summation ranges over the lower endpoints of the edges of $Y_i$ that connect to vertices at level $L$. 
 \end{enumerate}
 \end{itemize}
\par 
The idea behind the $\mathcal{E}_{\Lambda}$-GRC is the following. Suppose a one-parameter family of differentials $\{\omega_t\}$ in~$\mathcal F_{\Lambda}$ degenerates to the multi-scale differential $\omega$ as $t\to 0$, with the  vanishing rate $t^{\ell}$ as they approach the limit components on level $L$ for some $\ell \geq 0$. Then, given an equation $f$ in $\Lambda^{\vee}$, if it satisfies the description in condition (2), it can descend to impose a relation for the residues at level $L$ by applying Stokes' theorem to the limit differential $\lim\limits_{t\to 0}t^{-\ell }\omega_t$ on the subsurfaces $Y_i$ with boundary punctures at the poles. Note that those terms in $f$ whose corresponding poles $q_j$ connect to a vertex strictly below level $L$ do not appear in the $\mathcal{E}_{\Lambda}$-GRC for level $L$, as their vanishing rates are faster than $t^{\ell}$. Moreover, if a component $Y_{i}$ contains a pole whose residue term does not appear in $f$, then condition (2) can be non-empty only if $a_{i}=0$.
\par 
Additionally, we remark that up to a linear combination, there are indeed only finitely many non-trivial conditions imposed by the $\mathcal{E}_{\Lambda}$-GRC. To see this, if there are two independent relations in condition (2) that involve the same subset of $Y_i$ with nonzero coefficients, we can use their combination to produce another relation that involves a smaller subset of $Y_i$. In the end, we can reduce to a finite set of relations that form an echelon form for each level. 
\par
We illustrate the $\mathcal{E}_{\Lambda}$-GRC through the following example. 
\par
\begin{ex}
 Consider a stratum with four poles. We impose the residue relations $r_{1}+r_{2}+2r_{3}=0$ and $2r_{2} + r_3=0$ (together with the residue theorem relation $r_1+r_2+r_3+r_4 = 0$). In other words, the given subspace of residue tuples $\Lambda$ is spanned by $(3,1,-2,-2)$. 
 
 Consider a multi-scale differential $(X,\omega)$ in the boundary of this isoresidual fiber whose level graph $\Gamma$ is given on the left of Figure~\ref{fig:exclosure} and the associated level graph $\Gamma'$ is given on the right, where the marked poles of $\Gamma$ become the edges connecting to level $\infty$ in $\Gamma'$. 
    \begin{figure}[htb]
   \centering
\begin{tikzpicture}[scale=1]

\coordinate (a1) at (0,0);\fill (a1) circle (2pt);
\coordinate (a2) at (-.8,1);\fill (a2) circle (2pt);
\coordinate (a4) at (1,1);\fill (a4) circle (2pt);
\coordinate (a3) at (0,1);\fill (a3) circle (2pt);
\draw (a1) --node[left]{$e_{1}$} (a2);
\draw (a1) --node[right]{$e_{2}$} (a3);
\draw (a1) --node[right]{$e_{3}$} (a4);
\draw (a2) --++ (90:.3) node[above] {$q_{1}$};
\draw (a3) --++ (50:.3)node[above] {$q_{3}$};
\draw (a3) --++ (130:.3)node[above] {$q_{2}$};
\draw (a4) --++ (90:.3)node[above] {$q_{4}$};

\begin{scope}[xshift=5cm]
\coordinate (a1) at (0,0);\fill (a1) circle (2pt);
\coordinate (a2) at (-.8,1);\fill (a2) circle (2pt);
\coordinate (a4) at (1,1);\fill (a4) circle (2pt);
\coordinate (a3) at (0,1);\fill (a3) circle (2pt);
\coordinate (b1) at (-.8,2);
\coordinate (b2) at (-.3,2);
\coordinate (b3) at (.3,2);
\coordinate (b4) at (1,2);
\draw (a1) -- (a2);
\draw (a1) -- (a3);
\draw (a1) -- (a4);
\draw (a2) -- (b1);
\draw (a3) -- (b2);
\draw (a3) -- (b3);
\draw (a4) -- (b4);

\filldraw[fill=white] (b1) circle (2pt);
\filldraw[fill=white] (b2) circle (2pt);
\filldraw[fill=white] (b3) circle (2pt);
\filldraw[fill=white] (b4) circle (2pt);

\node at (3,0) {$L=-1$};
\node at (3,1) {$L=0$};
\node at (3,2) {$L=\infty$};
\end{scope}
 \end{tikzpicture}
\caption{The level graphs $\Gamma$ and $\Gamma'$,  illustrating the $\mathcal E_{\Lambda}$-GRC. The vertices at level infinity are pictured in white.}
\label{fig:exclosure}
\end{figure}

Note that all the poles $q_i$ are at level zero. Hence, 
the $\mathcal E_{\Lambda}$-GRC for level zero implies that the imposed residue relations still hold at level zero. Combining with the residue theorem at each vertex, we conclude that $r_1 = r_{2}=r_{3}= r_4 = 0$ for the residues of $\omega$ at each $q_i$ at level zero of $\Gamma$.  

Next, denote by $e_{i}$ for $i=1,2,3$ the residues of the multi-scale differential $\omega$ at the lower endpoints of the edges joining the bottom component of $\Gamma$ to level zero. Then, the $\mathcal E_{\Lambda}$-GRC at level $-1$ implies 
that there exists some $r\in\CC$ such that $(e_{1},e_{2},e_{3}) = (r,-r/3,-2r/3)$. To see this, suppose that $e_{1}=r$. Combining the two given equations gives $r_1 + 3(r_2 + r_3) + 0\ r_4 = 0$, which descends to the relation $e_1 + 3e_2 = 0$ by the $\mathcal E_{\Lambda}$-GRC. Therefore, $e_2 = - r/3 $, which further implies $e_3 = -2r/3$ by the residue theorem on the bottom vertex. 

Finally, note that the equations $2r_1 + 3r_4 = 0$ and 
$2(r_2+r_3) - r_4 = 0$ are also in $\Lambda^{\vee}$. If we run the $\mathcal E_{\Lambda}$-GRC to them at level $-1$, we still obtain the same conclusion $2e_1 + 3e_3 = 0$ and  
$2e_2 - e_3 = 0$ as before. 
\end{ex}

After the above preparation, we can state the main result of this section as follows. 
\begin{thm}\label{thm:closureMSD}
 A multi-scale differential is in the closure $\overline{\mathcal{F}}_{\Lambda}$
 of $\mathcal{F}_{\Lambda}$ if and only if it satisfies the $\mathcal{E}_{\Lambda}$-GRC at every level.
 Moreover, $\overline{\mathcal{F}}_{\Lambda}$ is a smooth orbifold with normal crossing boundary. Finally, the codimension of a boundary stratum in $\overline{\mathcal{F}}_{\Lambda}$ is equal to the number of horizontal edges plus the number of levels below zero.  
\end{thm} 

We remark that the statement of Theorem~\ref{thm:closureMSD} generalizes  \cite[Proposition 4.2]{CMZ}, where partial sum residue conditions therein are replaced by arbitrary linear residue conditions in our setting. 

\begin{proof}
 Note that the $\mathcal{E}_{\Lambda}$-GRC condition contains the usual GRC. In particular, such multi-scale differentials can be smoothed into the entire stratum interior. Therefore, we only need to verify the extra residue conditions imposed by $\mathcal{E}_{\Lambda}$. First, we remark that every  horizontal node can be smoothed locally and independently. This is because the defining linear equations of $\mathcal F_{\Lambda}$ consists of loops around the marked poles, which do not cross any horizontal node (see \cite{BDG}). 
 We can thus assume that the concerned multi-scale differential has only vertical nodes.
 
 The necessity of the $\mathcal{E}_{\Lambda}$-GRC condition has already been explained in the paragraph below the statement of the $\mathcal{E}_{\Lambda}$-GRC condition. Here we provide a more formal argument, 
 following \cite[Proposition 6.3]{BCGGM3}. Take a pointed topological surface $\Sigma$ with $\mathcal{V}\subset \Sigma$ as a disjoint union of simple closed curves such that the degenerate surface with dual graph $\Gamma$ is obtained by
pinching the curves in $\mathcal{V}$ to form the corresponding nodes. The residue assignment $\rho$ can be identified with an element of $H^1(\mathcal{V})$.  It satisfies the usual GRC if and only if it factors through level quotients.  Now, consider an imposed linear residue relation $f$ as an element of $H^1(\Sigma \setminus P, Z)$. In order to smooth the multi-scale differential while preserving $f$, we need its image under $H^1(\Sigma \setminus P, Z) \to H^1(\mathcal{V})$ to contain $\rho$. This shows that the $\mathcal{E}_{\Lambda}$-GRC condition is necessary.
\smallskip
\par
 Conversely, we will show that the $\mathcal{E}_{\Lambda}$-GRC condition is also sufficient for smoothing the multi-scale differential while preserving the equations in $\Lambda^{\vee}$. The upshot is that since the imposed equations are linear, if we have two solutions, then their linear combination remains to be a solution.
 
 As in \cite{BCGGM}, the smoothing is level by level from the top to the bottom. Hence, we consider a level~$L$ and suppose by induction that all the edges whose lower endpoints are at a level strictly above $L$ have been smoothed. Then, consider the connected components $Y_i$ of $\Gamma'_{>L}$, and we will smooth their edges that connect to level $L$ under the assumption that the $\calE_{\Lambda}$-GRC is satisfied. 
 
Consider the poles $q'_{j}$ that lie in the components $Y_{i}$. In order to prove the result, we claim that it suffices to find a residue assignment $r_{j}$ for each pole $q'_{j}$ that satisfies the following conditions:  
 \begin{itemize}
  \item[a)] The~$r_{j}$ satisfy the equations in $\Lambda^{\vee}$.
  \item[b)] For each $Y_{i}$, we have $\sum_{q'_j\in Y_i}r_{j}=\sum_{e_j \in Y_{i, L}} \Res_{e_j} (\omega)$.
 \end{itemize}
 
 Indeed, if we can find the assignments $r_{j}$ as in the above, then we can find in each component~$Y_{i}$ a modifying differential with these residues at the $q_{j}'$ such that the sum of the residues of these modifying differentials and the residues of $\omega$ at the nodes is zero. Moreover, note that if $Y_i$ contains no marked poles at all, the usual GRC is satisfied.  Consequently, we can smooth this level in such a way that the residues at the $q_{j}'$ are proportional to the $r_{j}$. The linearity of $\Lambda^{\vee}$ implies that the corresponding equations are still satisfied. It follows that we can smooth the multi-scale differential in such a way that the equations of $\Lambda^{\vee}$ are satisfied by the residues after smoothing.
 \smallskip
 \par
 Therefore, it suffices to prove the existence of such residue assignments as in the above. We fix the residues of $\omega$ at the edges $e_{j}$ and consider the set $\calR$ of residue assignments satisfying part b) of the condition. More precisely, let $w_{i}=\sum_{e_j \in Y_{i, L}} \Res_{e_j} (\omega) \in \CC$. Then $\calR$ is the solution set of the linear system 
 $$y_{i}(r_1,\ldots, r_p)\coloneqq\sum_{q'_j\in Y_i}r_{j}=w_{i}\,$$
 given by all connected components $Y_i$ of $\Gamma'_{>L}$. 
 Note that $\calR$ is a nonempty affine subspace of the total residue space $\mathcal{R}_{p}$ (where $\calR$ is not a linear subspace if some $w_i$ is nonzero). We will show that if $\omega$ satisfies the $\calE_{\Lambda}$-GRC, then for all linear subspaces $V$ of dimension $k$ in $\Lambda^{\vee}$, the locus of residue assignments in $\calR\subset \calR_p$ that are compatible with $V$ (in the sense of satisfying the equations in $V$) remains to be a nonempty affine subspace. 
 To this aim, we apply induction on $k$ starting from $k=0$, and the final case of $k=\dim \Lambda^{\vee}$ thus provides the desired residue assignments to conclude the proof. 
 
 If $k=0$, then $V=0$. In this case, part a) of the residue assignment holds trivially and thus $\calR_{V}=\calR$.
 Now, suppose by induction that given a $k$-dimensional subspace $V\subset \Lambda^{\vee}$, 
 the locus of residue assignments $\calR_{V}\subset \calR$ is a nonempty affine subspace in $\calR_p$.  
 Next, we consider a subspace $V'$ of dimension $k+1$, where $V\subset V'\subset \Lambda^{\vee}$.

 Let $f\in V'\setminus V$ and consider the system of equations given by $y_{i}(r_1,\ldots, r_p)=w_{i}$ and $f_{i}(r_1,\ldots, r_p)=0$, where the $f_{i}$ form a basis of $V$. If adding the equation $f(r_1,\ldots, r_p)=0$ increases the rank of the system, then the linear hyperplane defined by $f = 0$ is not parallel to the affine subspace $\calR_V$. Therefore, their intersection is a nonempty  codimension-one affine subspace contained in $\calR_{V}$. Otherwise, i.e., when the rank is preserved, since $f\notin V$, there exists a relation of the form 
 $$f+\sum \lambda_{i}f_{i} = \sum a_{i}y_{i}\,,$$
 with at least one coefficient $a_{i}\neq 0$.
 Since $\omega$ satisfies the $\calE_{\Lambda}$-GRC, its condition 2) implies that  
 $\sum a_{i} w_{i}=0$. Therefore, the linear hyperplane defined by $f = 0$ contains the affine subspace $\calR_V$. It follows that 
 the solution set $\calR_{V'}$ of the linear system after adding the equation $f = 0$ is still given by $\calR_{V}$, which is nonempty. 
\smallskip
\par
Finally, under the system of local coordinates along the boundary, recalled at the end of Section~\ref{sub:multi-scale}, the closure $\overline{\mathcal{F}}_{\Lambda}$ is still defined by linear equations at each boundary point. Hence, it remains smooth. The level-wise opening parameters and local horizontal-node smoothing parameters imply the desired normal crossing boundary structure as well as the codimension count. Alternatively, these claims follow from the general description of linear subvarieties (see e.g., \cite[Corollary 1.8]{BDG}), as the defining  equations of~$\mathcal F_\Lambda$ consist of loops around each marked pole and thus do not cross any horizontal nodes. 
\end{proof}

\subsection{The closure of isoresidual fibers}\label{sub:isoclosure}

Since we are interested in the isoresidual fibers, we restrict to the case when $\Lambda = \mathbb{C}^{\ast} \cdot \lambda$, where $\lambda = (\lambda_1,\ldots, \lambda_p)$. In this case, we abuse notation and also write $\lambda$ instead of~$\Lambda$.
Note that when $\lambda=(0,\dots,0)$, we have $\mathcal{E}_{\lambda}^{\vee}=\mathcal{R}_{p}^{\vee}$. If  $\lambda$ has at least one nonzero entry, then~$\mathcal{E}_{\lambda}^{\vee}$ is the hyperplane generated by the equations of the form
\begin{equation}\label{eq:relres}
 f_{i,j}(q_{1},\dots,q_{p}):=\lambda_{i}\Res(q_{j})-\lambda_{j}\Res(q_{i})=0\,.
\end{equation}

Now, we describe some behavior of the residues in the closure of isoresidual fibers.

\begin{cor}\label{cor:reszero}
Consider a multi-scale differential in the closure $\overline{\mathcal{F}}_{\lambda}$. If two poles are on different levels, then the pole on the higher level must have zero residue. Moreover, if two poles $q_i$ and $q_j$ are at the same level, then their residues satisfy $f_{i,j} = 0$ in \eqref{eq:relres}. 
\end{cor}

\begin{proof}
For the first claim, if $\lambda_i = 0$ in $\lambda$, then the limit residue remains to be zero in the multi-scale differential. Next,
consider two poles $q_{i}$ and $q_{j}$ such that $\lambda_{i}$ and $\lambda_{j}$ are both nonzero. Suppose that $q_{i}$ lies strictly above $q_{j}$.  
Then, running the $\mathcal{E}_{\lambda}$-GRC for the level of $q_i$ implies that the equation $f_{i,j} = 0$ holds with the residue at $q_{j}$ set to be $0$. This implies that $\lambda_{j}\Res(q_{i})-\lambda_{i}\cdot 0=0$, and hence $\Res(q_{i})=0$. 
 \par
 For the other claim, suppose that the two poles $q_{i}$ and $q_{j}$ are at the same level. Then, running the $\mathcal{E}_{\lambda}$-GRC for that level implies that their residues satisfy Equation~\eqref{eq:relres}. 
\end{proof}

Next, we show that if an isoresidual fiber has positive dimension, i.e., if $n\geq2$ or $g\geq1$, then the closure of an isoresidual fiber contains the locus of differentials with zero residues (up to a scalar multiple). 
\begin{cor}\label{cor:reszeros}
Given $\lambda\in \mathcal R_p$,  
 suppose $(X,\omega)$ is contained in the isoresidual fiber $\overline{\calF}_{\lambda}$, where $X$ is smooth. If there exists $i$ such that $\lambda_i \neq 0$ and ${\rm Res}(p_i) =0$ for $\omega$, 
then every residue of $\omega$ is zero. Conversely, the locus of such residueless differentials (up to a scalar multiple) is contained in $\overline{\calF}_{\lambda}$ for every $\lambda\in \mathcal R_p$. 
\end{cor}

\begin{proof}
For the first claim, note that all $f_{i,j}\colon\lambda_i {\rm Res}(q_j) - \lambda_j {\rm Res}(q_i)=0$ in Equation~\eqref{eq:relres} are satisfied by differentials on smooth curves contained in $\mathcal F_\lambda$. If $\lambda_i \neq 0$ and ${\rm Res}(q_i)= 0$, it follows from $f_{i,j} = 0$ that ${\rm Res}(q_j)= 0$ for all $j$. 

For the other claim, residueless differentials on smooth curves satisfy $\calE_{\lambda}$-GRC for all $\lambda \in \mathcal R_p$. Hence, they are contained in $\overline{\calF}_{\lambda}$ for every $\lambda$. 
\end{proof}

Finally, we show that the operation of splitting a zero can be performed within any isoresidual fiber.
\begin{cor}\label{cor:breakingup}
 A zero of order $a_1+a_2$ in a translation surface can be locally split into two nearby zeros of respective orders $a_1$ and $a_2$, without affecting the translation structure outside a small flat geometric neighborhood containing the zeros.
 \end{cor}
 This operation is illustrated in detail in Section~8.1 of \cite{EMZ}.

 \begin{proof}
We reinterpret this operation using multi-scale differentials.
Let $(X_{1},\omega_{1})$ be a differential with a zero~$z_{0}$ of order $a_0=a_{1}+a_{2}$ whose residues are given by $\lambda$. Take another differential $(\mathbb C\PP^1, \omega_2)$ in the stratum $\mathcal{H}(a_{1},a_{2},-a_0-2)$. Identifying $z_0$ with the pole $q_0$ of $\omega_{2}$, we obtain a multi-scale differential $(X,\omega)$ by taking the unique equivalence class of prong-matchings $\sigma$ at the node. The operation of breaking up the zero $z_{0}$ is the smoothing of the multi-scale differential $(X,\omega)$ into the respective stratum. By the residue theorem, the residue at the pole of $\omega_{2}$ vanishes, and hence the $\calE_{\lambda}$-GRC holds. Therefore, by Theorem~\ref{thm:closureMSD}, the multi-scale differential $(X,\omega)$ lies in the closure of the isoresidual fiber $\calF_{\lambda}$.
\end{proof}

\subsection{The multi-scale boundary of isoresidual curves}\label{sub:boundary}
\par 
In this section, we focus on the case of $g=0$ and~$\mu$ with exactly two zero orders. We will analyze in detail how $\overline{\mathcal F}_\lambda$ intersects the boundary of the multi-scale compactification $\MS(\mu)$, especially when $\lambda$ is generic.

First, we bound the types of level graphs that can appear in the boundary of $\overline{\mathcal F}_\lambda$. 

\begin{lem}\label{lem:residues}
Suppose $g=0$ and $\mu$ has exactly two nonnegative entries. Given $\lambda\in \mathcal{R}_{p}$, suppose the level graph $\Gamma$ of a multi-scale differential in $\overline{\mathcal{F}}_{\lambda}$ is not a single point (i.e., the underlying curve is not smooth). Then, either $\Gamma$ has two levels without any horizontal edges, or $\Gamma$ is a one-level graph with exactly one horizontal edge. 
\end{lem}

\begin{proof}
By assumption, $\mathcal F_{\lambda}$ is one-dimensional. Therefore, each of its boundary strata has codimension one. The claim thus follows from Theorem~\ref{thm:closureMSD}. 
\end{proof}

Now, we can classify all possible degenerations in a generic isoresidual curve. By Lemma~\ref{lem:residues}, for a level graph $\Gamma$ in the boundary of $\overline{\mathcal F}_\lambda$, either $\Gamma$ has two levels without horizontal edges, or $\Gamma$ has one level with exactly one horizontal edge. 
We start by analyzing the case of $\Gamma$ with two levels.
\begin{prop}
\label{prop:F-boundary}
Suppose $g=0$ and $\mu$ has exactly two nonnegative entries. 
For $\lambda \in \mathcal{R}_{p} \setminus \mathcal{A}_{p}$, 
let $(X,\omega)$ be a multi-scale differential with more than one level in the boundary of $\overline{\mathcal F}_\lambda$. Then there are two cases: 
\begin{itemize}
\item[(1)] The level graph of $X$ has two levels, where the bottom level has one vertex containing the zeros $z_1$ and $z_2$, while the top level has one vertex containing all the poles. Additionally, the residues at the top level poles are specified by $\lambda$. 
\par 
\item[(2)] The level graph of $X$ has two levels: the top level has one vertex containing a subset of poles $K$, while the bottom level has two vertices containing subsets of poles $I$ (with $z_1$) and $J$ (with $z_2$), respectively, where $I\sqcup J \sqcup K = \{1, \ldots, p\}$. Moreover, the top level differential has zero residue at every pole, while the residues of the bottom level multi-scale differential are specified by $\lambda$. 
\end{itemize}
\end{prop}
\par 
We call the graphs of type (2) above ``cherry graphs'' because of their appearance as shown in Figure~\ref{fig:cherry}. Note that $K$ cannot be empty (otherwise it would correspond to a horizontal degeneration). However, if either $I$ or $J$ is empty, then the corresponding graph is not stable. Its stabilization, in the case when $I=\emptyset$, is shown in (2') of Figure~\ref{fig:cherry}. Moreover,  if both $I$ and $J$ are empty, then the stable model becomes a single~$\mathbb{CP}^1$ with two zeros and residueless poles as explained in Corollary~\ref{cor:reszeros}.
   \begin{figure}[htb]
   \centering
\begin{tikzpicture}[scale=1]
\coordinate (a1) at (0,0);\fill (a1) circle (2pt);
\coordinate (a2) at (0,-1);\fill (a2) circle (2pt);
\draw (a1) -- (a2);
\draw[dotted] (145:.25) .. controls ++(50:.1) and ++(130:.1) .. (35:.25);
\draw (a1) --++ (20:.3)node[right] {$p_{p}$};
\draw (a1) --++ (160:.3) node[left] {$p_{1}$};
\draw (a2) --++ (-125:.3) node[below left] {$a_{1}$};
\draw (a2) --++ (-55:.3)node[below right] {$a_{2}$};
\node at (0,-1.9) {(1)};
\begin{scope}[xshift=4cm]
\coordinate (a1) at (0,0);\fill (a1) circle (2pt);
\coordinate (a2) at (-.7,-1);\fill (a2) circle (2pt);
\coordinate (a3) at (.7,-1);\fill (a3) circle (2pt);
\draw (a1) -- (a2);
\draw (a1) -- (a3);
\draw[dotted] (145:.25) .. controls ++(50:.1) and ++(130:.1) ..node[above]{$K$} (35:.25);
\draw (a1) --++ (20:.3);
\draw (a1) --++ (160:.3);
\draw (a2) --++ (-125:.3);
\draw (a2) --++ (125:.3);
\draw[dotted] (a2)+(-145:.25) .. controls ++(120:.1) and ++(60:.1) ..node[left]{$I$} ++(145:.25);
\draw (a2) --++ (-55:.3) node[below] {$a_{1}$};
\draw (a3) --++ (-125:.3) node[below] {$a_{2}$};
\draw (a3) --++ (-55:.3);
\draw (a3) --++ (55:.3);
\draw[dotted] (a3)+(35:.25) .. controls ++(-60:.1) and ++(-120:.1) ..node[right]{$J$} ++(-35:.25);
\node at (0,-1.9) {(2)};
\end{scope}
\begin{scope}[xshift=8cm]
\coordinate (a1) at (0,0);\fill (a1) circle (2pt);
\coordinate (a3) at (.7,-1);\fill (a3) circle (2pt);
\draw (a1) -- (a3);
\draw[dotted] (145:.25) .. controls ++(50:.1) and ++(130:.1) ..node[above]{$K$} (35:.25);
\draw (a1) --++ (20:.3);
\draw (a1) --++ (160:.3);
\draw (a1) --++ (-120:.3) node[below] {$a_{1}$};
\draw (a3) --++ (-125:.3) node[below] {$a_{2}$};
\draw (a3) --++ (-55:.3);
\draw (a3) --++ (55:.3);
\draw[dotted] (a3)+(35:.25) .. controls ++(-60:.1) and ++(-120:.1) ..node[right]{$J$} ++(-35:.25);
\node at (0,-1.9) {(2')};
\end{scope}
 \end{tikzpicture}
\caption{The types of level graphs in Proposition~\ref{prop:F-boundary}.}
\label{fig:cherry}
\end{figure}
\par
\begin{proof}
By Lemma~\ref{lem:residues}, in this case the level graph $\Gamma$ has exactly two levels with no horizontal edges. Since every lower component has to contain a marked zero, there are at most two lower components. 
Moreover, for a top component, it either has at least two zeroes (including both marked zeros or zero edges), or its residues are not all equal to zero (see \cite[Theorem 1.2 (i)]{getaab}).
\smallskip
\par 
First, consider the case that the graph $\Gamma$ has a unique lower component. Then, this component either contains one or both zeros. If the unique lower component contains both zeros, then $\Gamma$ cannot have two or more top components. Otherwise, there would be a partial sum vanishing induced by the residue theorem on each top component, which is not possible since by hypothesis the residue tuple $\lambda$ is generic. Hence, $\Gamma$ has a unique top component. Moreover, the top component contains all the poles. Otherwise, there would still be a partial sum vanishing of the residues on top, which contradicts the generic choice of $\lambda$. Therefore, this case  corresponds to merging the two zeros, as illustrated in type (1). 

If the unique lower component contains only one zero, then again there must be a unique top component. By stability, the lower component contains at least one marked pole, and hence all the residues on the top component vanish by Corollary~\ref{cor:reszero}. This corresponds to the case where either $I$ or $J$ is empty, as illustrated in type (2') (for $J$ nonempty). 
\smallskip
\par
Next, consider the case that $\Gamma$ has exactly two lower components. Since each of the lower components contains a unique zero and they are not joined by any horizontal edges, it follows that the top level has a unique component. 
Then, $\Gamma$ must be a cherry graph, and by stability, every component contains at least one marked pole. Therefore, this is the case when $I$, $J$, and $K$ are all nonempty, as illustrated in type (2). 
Finally, the description of residues for each case follows from Corollary~\ref{cor:reszero}. 
\end{proof}
\par 
\begin{rmk}
Note that in the above description, the assumption that  $\lambda$ is not contained in $\mathcal{A}_{p}$ is necessary. Otherwise, for non-generic $\lambda$, for example, if some residues are prescribed to be zero, when the corresponding poles move to the lower level, they do not constrain the top-level residues to be zero.
\end{rmk}
\par 
\begin{rmk}
We observe that in the cases of Proposition~\ref{prop:F-boundary}, once the level graph and the prescribed residues are given, the top-level differential (with exactly two zeros and residueless poles) and the bottom-level differential (with residues specified by $\lambda$) each have only finitely many choices (up to a scale multiple). This makes sense, because the intersection locus of the one-dimensional isoresidual fiber $\overline{\mathcal F}_\lambda$ with the boundary of $\MS(\mu)$ must be zero-dimensional. 
\end{rmk}
\par 
Next, we describe how $\overline{\mathcal F}_\lambda$ intersects the horizontal boundary of $\MS(\mu)$. 
\par 
\begin{prop}\label{prop:F-horizontal}
Suppose $g=0$ and $\mu$ has exactly two nonnegative entries. 
For $\lambda \in \mathcal{R}_{p}$,  
suppose $\overline{\mathcal F}_\lambda$ contains a multi-scale differential $(X,\omega)$ with a horizontal edge. Then $X$ consists of two irreducible components,  $X_1$ and $X_2$, where $X_1$ contains the zero of order $a_1$ and the subset of poles labeled in $I$, while $X_2$ contains the zero of order $a_2$ and the subset of poles labeled in $J$. Furthermore, $I\sqcup J = \{1,\ldots,p\}$, and $a_1 - \sum\limits_{i \in I} b_i + 1 = 0$. Moreover, the residues of~$\omega$ at the marked poles are determined by~$\lambda$.
\end{prop}
\par 
\begin{proof} 
In this case, Lemma~\ref{lem:residues} implies that $X$ has exactly two components joined by a unique horizontal node, which yields the desired configuration. Moreover, note that the residues at the marked poles do not all degenerate to zero (by the residue theorem and the fact that the horizontal node has non-zero residue). Hence, Corollary~\ref{cor:reszero} implies that the residues at the marked poles are still determined by $\lambda$. 
\end{proof}

\section{The canonical translation structure of isoresidual curves}\label{sec:Translation}
\par 
In this section we first introduce the translation structure on isoresidual fibers in Section~\ref{sub:translationStructure} and study its singularities in Sections~\ref{subsub:zeros} to~\ref{subsub:HigherPoles}. The following sections give more details on this structure and in the last Section~\ref{sub:formalMain1}, we give the formal proof of Theorem~\ref{thm:MAIN1}.
\par
\subsection{The translation structure on isoresidual fibers}\label{sub:translationStructure}
\par 
Consider a stratum $\mathcal{H}(a_{1},a_{2},-b_{1},\dots,-b_{p})$ of meromorphic differentials on $\mathbb{CP}^{1}$ with two zeros. The isoresidual fibers are defined by fixing the periods of differentials along close loops encircling the poles. In terms of period coordinates, the unique local parameter remaining to deform $(\mathbb{CP}^{1},\omega)$ is the relative period of the differential along a path joining the two zeros. We recall that the zeros are labeled: the one of order $a_{i}$ is denoted by $z_{i}$ for $i=1,2$. 
\par
Any homotopy class of paths in $\mathbb{CP}^1$ punctured at the poles has at least one geodesic representative. Therefore, any homotopy class of paths from $z_{1}$ to $z_{2}$ has a representative that is a chain of saddle connections. Consequently, there is at least one saddle connection from $z_{1}$ to $z_{2}$.
\par 
Consider the isoresidual fiber $\mathcal{F}_{\lambda}$ determined by the configuration $\lambda = (\lambda_{1},\dots,\lambda_{p})$.
Let $\omega\in\mathcal{F}_{\lambda}$ and $\gamma$ be a path joining the two zeros $z_{1}$ and $z_{2}$ of $\omega$. In a neighborhood of $\omega$ in $\mathcal{F}_{\lambda}$, the period $z_{\gamma} = \int_{\gamma} \omega$ can serve as a local complex coordinate for $\mathcal{F}_{\lambda}$.
\par
For any other path $\gamma'$ joining the two zeros of $\omega$, the difference $z_{\gamma}-z_{\gamma'}$ is the period of the differential along the closed loop $\gamma \cup -\gamma'$. This latter period is a linear combination of residues. Since the residues are fixed in $\mathcal{F}_{\lambda}$, the two local coordinates thus differ by a constant. Therefore, the differential $dz_{\gamma}$ on $\mathcal{F}_{\lambda}$ does not depend on the choice of $\gamma$ and hence defines a canonical \textit{translation structure} on $\mathcal{F}_{\lambda}$.
\par
The relative period $z_{\gamma}$ is a locally injective complex function on an open subset of $\mathcal{F}_{\lambda}$. In particular, $d z_{\gamma}$ (defined on the entire fiber $\mathcal{F}_{\lambda}$) has no zero or pole. On the other hand, zeros and poles of $d z_{\gamma}$ appear on the boundary of the closure $\overline{\mathcal{F}}_{\lambda}$ in the multi-scale compactification of the stratum as given in Section~\ref{sub:translationStructure}. We will prove that the differential $\omega_{\lambda}$ extends to the closure $\overline{\mathcal F}_\lambda$ of $\mathcal{F}_{\lambda}$ as a meromorphic one-form. 
\par 
\subsection{Zeros of $\omega_{\lambda}$}\label{subsub:zeros}
\par 
First, we will describe the boundary points that are zeros of $\omega_{\lambda}$. Recall that $a=a_{1}+a_{2}$.
\begin{prop}\label{prop:zeros}
Given a stratum $\mathcal{H}(a_{1},a_{2},-b_{1},\dots,-b_{p})$, for any $\lambda \in \mathcal{R}_{p} \setminus \mathcal{A}_{p}$, the boundary points of $(\overline{\mathcal{F}}_{\lambda}, \omega_\lambda)$ corresponding to Proposition~\ref{prop:F-boundary} (1) are the zeros of $\omega_{\lambda}$.
\par
These zeros are of order $a$ and there are $f(a,p)=\frac{a!}{(a+2-p)!}$ of them, corresponding to elements in the zero-dimensional generic isoresidual fiber of the minimal stratum $\mathcal{H}(a,-b_{1},\dots,-b_{p})$.
\end{prop}
\par 
\begin{proof}
Let $(X,\omega)$ be the multi-scale differential in the  boundary of $\overline{\mathcal F}_\lambda$ corresponding to the merging of the two zeros $z_1$ and $z_2$ into a zero $z$ of order $a_1 + a_2$. The stable curve $X$ is a union of two $\mathbb{CP}^1$ attached in such a way that $z$ is glued to the other component containing $z_1$ and $z_2$ at a node. The breakup of $z$ into two nearby zeros of order $a_1$ and $a_2$ is described in Corollary~\ref{cor:breakingup}. The absolute period coordinates of the stratum at this point are given by the residues of the poles and by the opening parameter $t$ of the node. Indeed, since we consider the projectivized period joining $z_{1}$ and $z_{2}$, this is zero-dimensional. Now, the differential $\omega_{\lambda}$ on the isoresidual fiber is given by the derivative of the period joining the two zeros. The lower level differential is multiplied by $t^{a+1}$ when smoothed, as explained in \cite{BCGGM2}, especially in Sections~8.2 and~9.2. Therefore, the differential of the period $ct^{a+1}$, where $c$ is a constant, has the form $c(a+1)t^{a}dt$, which has a zero of order~$a$. The number of such zeros corresponds to the cardinality of the zero-dimensional generic isoresidual fiber in $\mathcal{H}(a,-b_{1},\dots,-b_{p})$.
\end{proof}

\begin{rmk}
 In terms of flat geometry, this proposition can be seen as follows. When splitting $z$ into two zeros of orders $a_{1}$ and $a_{2}$,  there are $a +1 = a_1 + a_2 + 1$ horizontal positive directions to open it up at $z$. Each of these yields a real positive saddle connection joining $z_1$ and $z_2$. Therefore, under the translation structure~$\omega_\lambda$ induced by the primitive of such a saddle connection, $(X,\omega)$ is a conical singularity of angle $2\pi (a+1)$ in the isoresidual fiber.
\end{rmk}

\begin{ex}
A  differential in $\mathcal{H}(2,3,-1,-2,-4)$ is shown in Figure~\ref{fig:isores1}. Its isoresidual deformations are given by the variation of $\gamma$, while fixing all the other saddle connections.
 \begin{figure}[ht]
\begin{tikzpicture}[scale=1,decoration={
    markings,
    mark=at position 0.5 with {\arrow[very thick]{>}}}]
    
    \begin{scope}[xshift=-1.5cm,yshift=-.5cm]
    \draw (-.25,0) coordinate (a) -- (.25,0) coordinate[pos=.5] (c) coordinate (b);

    \fill[fill=black!10] (a) -- (b) -- ++(0,1) --++(-.5,0) -- cycle;
    \draw  (a) -- ++(0,1)  coordinate[pos=.5] (d);
    \draw  (b) -- ++(0,1)  coordinate[pos=.5] (e);
\filldraw[fill=white](a) circle (2pt);
\filldraw[fill=white] (b) circle (2pt);
 
 \node[below] at (c) {$1$};
\node[left] at (d) {$5$};
\node[right] at (e) {$5$};
    \end{scope}
    
\begin{scope}[xshift=2.3cm,yshift=0cm]
      \fill[fill=black!10] (0,0) circle (1cm);
      \draw (-.25,0) coordinate (a) -- (.25,0) coordinate[pos=.5] (c) coordinate (b);
\filldraw[fill=white] (a) circle (2pt);
\filldraw[fill=white] (b) circle (2pt);
    \fill[white] (a) -- (b) -- ++(0,-1) --++(-.5,0) -- cycle;
 \draw  (b) --  (a);
 \draw (a) -- ++(0,-.9) coordinate[pos=.5] (d);
 \draw (b) -- ++(0,-.9)coordinate[pos=.5] (e);
\node[above] at (c) {$2$};
\node[left] at (d) {$4$};
\node[right] at (e) {$4$};
    \end{scope} 
    
    \begin{scope}[yshift=-3cm]
      \fill[fill=black!10] (0,0) circle (1.5cm);
 \draw (-.75,0) coordinate (a) -- (.25,0) coordinate[pos=.25] (c) coordinate[pos=.5] (d) coordinate[pos=.75] (e) coordinate (b);
    \fill[white] (a) -- (b) -- ++(0,1.5) --++(-1,0) -- cycle;
 \draw  (b) --  (a);
 \draw (b) -- ++ (-30:.8) coordinate[pos=.6] (i) coordinate (j);
 \draw (a) -- ++(0,1.4) coordinate[pos=.5] (f);
 \draw (b) -- ++(0,1.4)coordinate[pos=.5] (g);
  \draw  (j) --++  (.5,0) coordinate[pos=.5] (h) ;
  \filldraw[fill=white] (a) circle (2pt);
\filldraw[fill=white] (b) circle (2pt);
\filldraw[fill=white] (d) circle (2pt);
  \node[below] at (c) {$1$};
    \node[below] at (e) {$2$};
  \node[above] at (i) {$\gamma$};
    \fill (j) circle (2pt);
  
\node[left] at (f) {$5$};
\node[right] at (g) {$5$};
\node[above] at (h) {$6$};\node[below] at (h) {$8$};

       \fill[fill=black!10] (2.5,.75) circle (.7cm);    
  \fill (2.5,.75) circle (2pt);
  \draw (2.5,.75) --++(.7,0)  coordinate[pos=.5] (i) ;
\node[below] at (i) {$6$};\node[above] at (i) {$7$};
  
       \fill[fill=black!10] (2.5,-.75) circle (.7cm);    
  \fill (2.5,-.75) circle (2pt);
  \draw (2.5,-.75) --++(.7,0) coordinate[pos=.5] (j) ;
  \node[below] at (j) {$7$};\node[above] at (j) {$8$};
    \end{scope}    
\end{tikzpicture}
 \caption{A differential in $\mathcal{H}(2,3,-1,-2,-4)$ such that the isoresidual deformations are given by varying $\gamma$ only.} \label{fig:isores1}
\end{figure}
Now, suppose that $\gamma$ is very short and we start rotating it in the positive direction. It first meets the vertical ray labeled by $5$, and then the saddle connection labeled by $1$. In particular, after a rotation of angle $2\pi$, it does not come back to its original position. Only after a rotation of angle $12\pi$, it comes back to its original position. Therefore, this corresponds to a zero of order $5$ under the canonical translation structure of the isoresidual fiber. 
\end{ex}

\par 
\subsection{Simple poles of $\omega_{\lambda}$}\label{subsub:SimplePoles}
\par 
The intersection of $\overline{\mathcal{F}}_{\lambda}$ with the horizontal boundary of $\MS(\mu)$ is described in Proposition~\ref{prop:F-horizontal}. In the following, we will prove that each of these boundary points is a simple pole of the translation structure of $(\overline{\mathcal{F}}_{\lambda},\omega_{\lambda})$.
\par 
\begin{prop}\label{prop:simple-pole}
In a stratum $\mathcal{H}(a_{1},a_{2},-b_{1},\dots,-b_{p})$, for any $\lambda \in \mathcal{R}_{p} \setminus \mathcal{A}_{p}$, the boundary points of $\overline{\mathcal{F}}_{\lambda}$ that correspond to  multi-scale differentials with a horizontal edge are simple poles of $\omega_{\lambda}$.
\par
For such a boundary point, the underlying nodal curve $X$ consists of two irreducible components, $X_1$ and~$X_2$. Here, $X_1$ contains the zero of order $a_1$ and the subset of poles labeled by $I$, while $X_2$ contains the zero of order~$a_2$ and the subset of poles labeled by $J$. We have $I\sqcup J = \{1,\ldots,p\}$ and $a_1 - \sum\limits_{i \in I} b_i + 1 = 0$. We denote by $p_{I}$ and~$p_{J}$ the numbers of marked poles in $X_{1}$ and $X_{2}$, respectively. 
\par
For each such partition $I \sqcup J$, there are 
\begin{equation}
 \frac{a_{1}!}{(a_{1}-p_{I}+1)!} \cdot \frac{a_{2}!}{(a_{2}-p_{J}+1)!}
\end{equation}
of these boundary points. The residue of $\omega_{\lambda}$ at each of these simple poles is given by $r_{J}\coloneqq\sum\limits_{j \in J} \lambda_{j}$.
\end{prop}
\par 
\begin{proof}
The counting of such simple poles of $\omega_\lambda$ follows from the formula in Theorem~\ref{thm:numerofibre}. Now let $(X,\omega)$ be a multi-scale differential described in Proposition~\ref{prop:F-horizontal}. We use the log period coordinates as given in Sections~10.10 and~13.3 of \cite{BCGGM2}. In particular, from the last equation of that section, if the opening parameter of the node is $t$, then there exist constants $c_{1}$ and $c_{2}$ such that 
$$\int_{z_{1}}^{z_{2}}\omega = r_{J}\log\left(\frac{t}{c_{1}}\right) + c_{2} \,.$$
Near $(X,\omega)$, the differential $\omega_{\lambda}$ is thus given by
$$\omega_{\lambda} = d\left(\int_{z_{1}}^{z_{2}}\omega\right)=\frac{r_{J}}{t}dt\,.$$
\end{proof}
\par 
\begin{rmk}
The flat geometric way to understand the above result is the following.  The multi-scale differential $(X,\omega)$ has an infinite cylinder $C$ corresponding to the simple pole $q$, where the marked poles in the subsets $I$ and $J$ are separated on the two sides of $q$, respectively. Since the residues are determined by $\lambda$ and do not vary in the isoresidual fiber, the residue of $q$, i.e., the period of the cylinder core curve, is fixed as $\sum\limits_{j \in J} \lambda_{j}$ (up to orientation). Therefore, the deformation of $(X,\omega)$ into $\mathcal F_\lambda$ is parametrized by deforming the cylinder $C$ back to a finite cylinder while preserving the period of the core curve. A local neighborhood of $(X,\omega)$ under the translation structure of $\omega_\lambda$ can be identified with a local neighborhood of $q$ in $(X,\omega)$.
\end{rmk}

\begin{ex}
 A flat geometric picture of an isoresidual deformation around a simple pole is given in Figure~\ref{fig:polesimplefibre}. The simple pole is formed by letting the height of $\gamma$ go to infinity.
 \begin{figure}[ht]
\begin{tikzpicture}[scale=1,decoration={
    markings,
    mark=at position 0.5 with {\arrow[very thick]{>}}}]
    
    \begin{scope}[xshift=-1.5cm,yshift=-.5cm]
    \draw (-.25,0) coordinate (a) -- (.25,0) coordinate[pos=.5] (c) coordinate (b);

    \fill[fill=black!10] (a) -- (b) -- ++(0,1) --++(-.5,0) -- cycle;
    \draw  (a) -- ++(0,1)  coordinate[pos=.5] (d);
    \draw  (b) -- ++(0,1)  coordinate[pos=.5] (e);
\filldraw[fill=white](a) circle (2pt);
\filldraw[fill=white] (b) circle (2pt);
 
 \node[below] at (c) {$1$};
\node[left] at (d) {$5$};
\node[right] at (e) {$5$};
    \end{scope}
    
\begin{scope}[xshift=2.3cm,yshift=0cm]
      \fill[fill=black!10] (0,0) circle (1cm);
      \draw (-.25,0) coordinate (a) -- (.25,0) coordinate[pos=.5] (c) coordinate (b);
\filldraw[fill=white] (a) circle (2pt);
\filldraw[fill=white] (b) circle (2pt);
    \fill[white] (a) -- (b) -- ++(0,-1) --++(-.5,0) -- cycle;
 \draw  (b) --  (a);
 \draw (a) -- ++(0,-.9) coordinate[pos=.5] (d);
 \draw (b) -- ++(0,-.9)coordinate[pos=.5] (e);
\node[above] at (c) {$2$};
\node[left] at (d) {$4$};
\node[right] at (e) {$4$};
    \end{scope} 
    
    \begin{scope}[yshift=-3cm]
      \fill[fill=black!10] (0,0) circle (1.5cm);
  \node[below] (0,0) {$9$};
 \draw (-.5,0) coordinate (a) -- (.5,0) coordinate[pos=.25] (c) coordinate[pos=.5] (d) coordinate[pos=.75] (e) coordinate (b);
  \fill[] (a) circle (2pt);
\fill[] (b) circle (2pt);
    \fill[white] (a) -- (b) -- ++(0,1.5) --++(-1,0) -- cycle;
 \draw  (b) --  (a);
 \draw (a) -- ++(0,1.4) coordinate[pos=.5] (f);
 \draw (b) -- ++(0,1.4)coordinate[pos=.5] (g);
  \draw  (b) --++  (1,0) coordinate[pos=.5] (h) ;

\node[left] at (f) {$5$};
\node[right] at (g) {$5$};
\node[above] at (h) {$6$};\node[below] at (h) {$8$};

\filldraw[fill=black!10] (-.5,2) coordinate (o)   -- ++(1,0)coordinate[pos=.5]    (l)coordinate (p)   -- ++(80:1.1) coordinate[pos=.5] (m) coordinate (q)    -- ++(-1,0)   coordinate (r) coordinate[pos=.25] (t)  coordinate[pos=.5] (n)  coordinate[pos=.75] (s) -- ++(-100:1.1) coordinate[pos=.5] (k);   
   \fill[] (o) circle (2pt);
\fill[] (p) circle (2pt);
\filldraw[fill=white] (q) circle (2pt);
\filldraw[fill=white] (r) circle (2pt);
\filldraw[fill=white] (n) circle (2pt);
 
 \node[above] at (s) {$1$};\node[above] at (t) {$2$};
  \node[right] at (m) {$\gamma$}; \node[left] at (k) {$\gamma$}; \node[below] at (l) {$9$};
  
       \fill[fill=black!10] (2.5,.75) circle (.7cm);    
  \fill (2.5,.75) circle (2pt);
  \draw (2.5,.75) --++(.7,0)  coordinate[pos=.5] (i) ;
\node[below] at (i) {$6$};\node[above] at (i) {$7$};
  
       \fill[fill=black!10] (2.5,-.75) circle (.7cm);    
  \fill (2.5,-.75) circle (2pt);
  \draw (2.5,-.75) --++(.7,0) coordinate[pos=.5] (j) ;
  \node[below] at (j) {$7$};\node[above] at (j) {$8$};
    \end{scope}    

    \draw[->] (3.5,-2.5) -- (6,-2.5) coordinate[pos=.5] (a);
    \node[above] at (a) {$\Im(\gamma) \to +\infty$};

\begin{scope}[xshift=8cm]
 
\begin{scope}[xshift=-1.5cm,yshift=-.5cm]
\draw (-.25,0) coordinate (a) -- (.25,0) coordinate[pos=.5] (c) coordinate (b);

    \fill[fill=black!10] (a) -- (b) -- ++(0,1) --++(-.5,0) -- cycle;
    \draw  (a) -- ++(0,1)  coordinate[pos=.5] (d);
    \draw  (b) -- ++(0,1)  coordinate[pos=.5] (e);
\filldraw[fill=white](a) circle (2pt);
\filldraw[fill=white] (b) circle (2pt);
 
 \node[below] at (c) {$1$};
\node[left] at (d) {$5$};
\node[right] at (e) {$5$};
    \end{scope}
    
\begin{scope}[xshift=2.3cm,yshift=0cm]
      \fill[fill=black!10] (0,0) circle (1cm);
      \draw (-.25,0) coordinate (a) -- (.25,0) coordinate[pos=.5] (c) coordinate (b);
\filldraw[fill=white] (a) circle (2pt);
\filldraw[fill=white] (b) circle (2pt);
    \fill[white] (a) -- (b) -- ++(0,-1) --++(-.5,0) -- cycle;
 \draw  (b) --  (a);
 \draw (a) -- ++(0,-.9) coordinate[pos=.5] (d);
 \draw (b) -- ++(0,-.9)coordinate[pos=.5] (e);
\node[above] at (c) {$2$};
\node[left] at (d) {$4$};
\node[right] at (e) {$4$};
    \end{scope} 
    
    \begin{scope}[xshift=.2cm,yshift=-2.5cm]
     \fill[fill=black!10] (-.5,2) coordinate (o)   -- ++(1,0)coordinate[pos=.5]   (l)coordinate (p)   -- ++(90:1.1) coordinate[pos=.5] (m) coordinate (q)    -- ++(-1,0)   coordinate (r) coordinate[pos=.25] (t)  coordinate[pos=.5] (n)  coordinate[pos=.75] (s) -- ++(-90:1.1) coordinate[pos=.5] (k);   
\draw (o) -- (r) -- (q) -- (p);
  \node[right] at (m) {$\gamma_{2}$}; \node[left] at (k) {$\gamma_{2}$};

\filldraw[fill=white] (q) circle (2pt);
\filldraw[fill=white] (r) circle (2pt);
\filldraw[fill=white] (n) circle (2pt);
 \node[above] at (s) {$1$};\node[above] at (t) {$2$};
      \end{scope}
      
    \begin{scope}[yshift=-4cm]
      \fill[fill=black!10] (0,0) circle (1.5cm);
  \node[below] (0,0) {$9$};
 \draw (-.5,0) coordinate (a) -- (.5,0) coordinate[pos=.25] (c) coordinate[pos=.5] (d) coordinate[pos=.75] (e) coordinate (b);
  \fill[] (a) circle (2pt);
\fill[] (b) circle (2pt);
    \fill[white] (a) -- (b) -- ++(0,1.5) --++(-1,0) -- cycle;
 \draw  (b) --  (a);
 \draw (a) -- ++(0,1.4) coordinate[pos=.5] (f);
 \draw (b) -- ++(0,1.4)coordinate[pos=.5] (g);
  \draw  (b) --++  (1,0) coordinate[pos=.5] (h) ;

\node[left] at (f) {$5$};
\node[right] at (g) {$5$};
\node[above] at (h) {$6$};\node[below] at (h) {$8$};

\fill[fill=black!10] (-.5,1.8) coordinate (o)   -- ++(1,0)coordinate[pos=.5]   (l)coordinate (p)   -- ++(90:1.1) coordinate[pos=.5] (m) coordinate (q)    -- ++(-1,0)   coordinate (r) coordinate[pos=.25] (t)  coordinate[pos=.5] (n)  coordinate[pos=.75] (s) -- ++(-90:1.1) coordinate[pos=.5] (k);   
\draw (r) -- (o) -- (p) -- (q);

   \fill[] (o) circle (2pt);
\fill[] (p) circle (2pt);

  \node[right] at (m) {$\gamma_{1}$}; \node[left] at (k) {$\gamma_{1}$}; \node[below] at (l) {$9$};
  
       \fill[fill=black!10] (2.5,.75) circle (.7cm);    
  \fill (2.5,.75) circle (2pt);
  \draw (2.5,.75) --++(.7,0)  coordinate[pos=.5] (i) ;
\node[below] at (i) {$6$};\node[above] at (i) {$7$};
  
       \fill[fill=black!10] (2.5,-.75) circle (.7cm);    
  \fill (2.5,-.75) circle (2pt);
  \draw (2.5,-.75) --++(.7,0) coordinate[pos=.5] (j) ;
  \node[below] at (j) {$7$};\node[above] at (j) {$8$};
    \end{scope}   
\end{scope}

\end{tikzpicture}
 \caption{A isoresidual deformation in $\mathcal{H}(2,3,-1,-2,-4)$ leading to a simple pole of its flat structure.} \label{fig:polesimplefibre}
\end{figure}
Note that the action of the parabolic element $\begin{pmatrix}
   1 & s \\
   0 & 1 
\end{pmatrix}$
on the cylinder bounded by the saddle connections $\gamma$, $1$, $2$ and $9$ is trivial, where~$s$ is the length of the saddle connection $9$. This illustrates the fact that the corresponding pole has a residue equal to $s$ as explained in Section~\ref{subsub:SimplePoles}.
\end{ex}
\par 
\subsection{Higher-order poles of $\omega_{\lambda}$}\label{subsub:HigherPoles}
\par 
Recall that for generic residues the intersection points of $\overline{\mathcal{F}}_{\lambda}$ with the vertical boundary of $\MS(a_{1},a_{2},-b_{1},\dots,-b_{p})$ are described in Proposition~\ref{prop:F-boundary}. Case (1) corresponds to the zeros of $\omega_{\lambda}$ as shown in Section~\ref{subsub:zeros}. Now we treat case (2) and show that these are poles of order at least two in the translation structure of $(\overline{\mathcal{F}}_{\lambda},\omega_{\lambda})$. These points are described by a partition $I \sqcup J \sqcup K$ of the set of the poles. Their level graphs are cherry graphs as pictured in Figure~\ref{fig:cherry}. The top level component $X_{\rm top}$ contains the subset $K$, while the bottom level has:
\begin{itemize}
    \item one vertex $X_{1}$ that contains a subset of poles $I$ with the zero $z_1$ of order $a_{1}$;
    \item one vertex $X_{2}$ that contains a subset of poles $J$ with the zero $z_{2}$ of order $a_{2}$. 
\end{itemize}
We denote by $p_{I}$, $p_{J}$, and $p_{K}$ the numbers of marked poles in $X_{1}$, $X_{2}$, and $X_{\rm top}$, respectively. We further define
\begin{equation}\label{eq:prongs}
 c_{1}=a_{1}+1-\sum\limits_{i \in I} b_{i} \quad \text{and} \quad c_{2}=a_{2}+1-\sum\limits_{j \in J} b_{j}\,.
\end{equation}
Finally, recall from Section~\ref{sub:counting} that $\Xi(\mu)$ is the number of residueless differentials in the stratum $\calH(\mu)$.
\par 
\begin{prop}\label{prop:flat-boundary}
Given a stratum $\mathcal{H}(a_{1},a_{2},-b_{1},\dots,-b_{p})$, $\lambda \in \mathcal{R}_{p} \setminus \mathcal{A}_{p}$ and a boundary point of $(\overline{\mathcal{F}}_{\lambda}, \omega_\lambda)$ corresponding to the partition of poles $I \sqcup J \sqcup K$. Up to reordering the indices, we assume that the poles in~$K$ are the poles of order $b_{1},\ldots,b_{p_{K}}$. We have the following cases:
\par 
{\bf Both $I$ and $J$ are nonempty.} There are
\begin{equation}\label{eq:numbnoempty}
 \frac{a_{1}!}{(a_{1}-p_{I}+1)!} \cdot \frac{a_{2}!}{(a_{2}-p_{J}+1)!} \cdot \gcd (c_1, c_2) \cdot \Xi(c_{1}-1,c_{2}-1;b_{1},\dots,b_{p_{K}})
\end{equation}
of these boundary points. Each of them is a pole of order
 $1 + \lcm (c_1, c_2)$
 for $\omega_{\lambda}$, with residue 
\begin{equation}\label{eq:resIJnonempty}
 -\frac{c_{2}}{\gcd(c_{1},c_{2})}\sum\limits_{i \in I} \lambda_{i} + \frac{c_{1}}{\gcd(c_{1},c_{2})}\sum\limits_{j \in J} \lambda_{j}\,.
\end{equation}
\par
{\bf $I$ is empty and $J$ is nonempty.} There are
\begin{equation}
  \frac{a_{2}!}{(a_{2}-p_{J}+1)!} \cdot \Xi(a_{1},c_{2}-1;b_{1},\dots,b_{p_{K}})
\end{equation}
of these boundary points. Each of them is a pole of order $1 + c_{2}$ with residue
\begin{equation}\label{eq:residuecase2}
 \sum\limits_{j \in J} \lambda_{j}\,.
\end{equation}
\par
{\bf $I$ is nonempty and $J$ is empty.} There are
\begin{equation}
 \frac{a_{1}!}{(a_{1}-p_{I}+1)!} \cdot \Xi(c_{1}-1,a_{2};b_{1},\dots,b_{p_{K}})
\end{equation}
of these boundary points. Each of them is a pole of order $1 + c_{1}$ with residue
\begin{equation}
 -\sum\limits_{i \in I} \lambda_{i} \,.
\end{equation}
\par
{\bf Both $I$ and $J$ are empty.} There are $\Xi(a_{1},a_{2};b_{1},\dots,b_{p})$ of these boundary points. Each of them is a double pole with zero residue.
\end{prop}
\par 
As a general remark, if we transpose the zeros $z_{1}$ and $z_{2}$, we can multiply $\omega_{\lambda}$ by $-1$ and transpose the subsets $I$ and $J$. This symmetry is reflected in  the local invariants of the poles of $\omega_{\lambda}$.
\begin{proof}
We first focus on the case when $I$, $J$, and $K$ are all nonempty. We start by counting the number of multi-scale differentials associated to this partition in the closure of the isoresidual fiber. The data of a multi-scale differential has various pieces. First, it has two differentials in the lower level strata with a unique zero respectively of order $a_{i}$ and residues of poles prescribed. This gives the first two terms of Equation~\eqref{eq:numbnoempty}. On the top level, we have a differential with two zeros and poles of orders $b_{1},\dots,b_{p_{K}}$ and zero residues. Their number gives the last term of  Equation~\eqref{eq:numbnoempty}. Additionally, a multi-scale differential contains the data of an equivalence class of {\em prong-matchings}  (see Section~5.4 of \cite{BCGGM3}).  In the case of a cherry graph with the prong numbers $c_1$ and $c_2$ at the two vertical edges, there are $\gcd (c_1, c_2)$ of such classes, giving the remaining term of  Equation~\eqref{eq:numbnoempty}. 
\par 
Next, we determine the flat geometric structure at such points. In order to avoid heavy notations, 
we write only the relevant part of the log period coordinates. The isoresidual curve is parametrized by the opening parameter $t$ of both nodes. The period of a path $\gamma$ joining $z_{1}$ and $z_{2}$ is a multi-valued function (since the residue corresponding to such pole is not zero in general). Nevertheless, since we consider a one-parameter family, Corollary~6.3 of \cite{BenBound} shows that its log period $\psi_{\gamma}(t)$ is a holomorphic function of~$t$. Moreover, it gives a formula for the value of this log period that we can compute in our context (following the notation of \cite[Definition~5.1]{BenBound}). 
We denote by $m_{i}=\lcm (c_1, c_2)/c_{i}$ for $i=1,2$.  The rescaling parameter of the top component is equal to $1$, so $t_{\lceil\top(\gamma)\rceil}=1$. Now consider the edge $e_{1}$ between the top component and the bottom component containing $z_{1}$ and the corresponding vanishing cycle $\lambda_{e_{1}}$. The intersection pairing $\langle\gamma,\lambda_{e_{1}}\rangle = 1$. The residue $r_{e_{1}}(t)$ is given by $t^{a}\sum_{i\in I}\lambda_{i}$, where $a = \lcm (c_1, c_2)$. Finally, $\sigma_{e}$ is given by $m_{i}$ based on the formula at the beginning of \cite[Section~7.4]{BenBound}. Doing the same for $e_{2}$ and noting that $\langle\gamma,\lambda_{e_{2}}\rangle = -1$, we obtain that the log period of $\gamma$ on the isoresidual curve is 
\begin{equation*}
 \psi_{\gamma}(t) = \int_{\gamma}\omega(t) - \lambda_{I}t^{a}m_{1}\log(t)+\lambda_{J}t^{a}m_{2}\log(t)\
\end{equation*}
where $\lambda_I = \sum_{i\in I}\lambda_{i}$ and $\lambda_J = \sum_{j\in J}\lambda_{j}$. 
Since $\psi_{\gamma}(0)\neq0$ (by the second part of Corollary~6.3 of \cite{BenBound}), the period of $\gamma$ is given by
$$t^{-a}\psi_{\gamma}(t) - \lambda_{I}m_{1}\log(t)+\lambda_{J}m_{2}\log(t)\,.$$
Therefore, its derivative 
$$\left(t^{-a-1}(-a\psi_{\gamma}(t)+t\psi_{\gamma}'(t)) + \frac{-\lambda_{I}m_{1} + \lambda_{J}m_{2}}{t}\right)dt$$
has a pole of order $a+1$ with residue as given by Equation~\eqref{eq:resIJnonempty}.
\smallskip
\par 
If $I$ is empty and $J$ is nonempty, the stable model of the level graph has exactly one top vertex, one bottom vertex, and one vertical edge, where the top vertex contains $z_1$ and the poles in $K$, the bottom vertex contains $z_2$ and the poles in $J$, and the prong number of the edge is $c_2$. In this case,  there is a unique prong-matching equivalence class. The same argument as in the preceding case justifies the desired claim. The case when $I$ is nonempty and $J$ is empty follows by symmetry. 
\smallskip
\par 
If $I$ and $J$ are both empty, the stable model of the level graph consists of a single vertex only, where the residue of every pole becomes zero. The number of such differentials with zero residues is computed in Section~\ref{sub:counting} as  
$\Xi(a_1, a_2; b_1, \ldots, b_p)$. Consider such a residueless differential $(\mathbb{CP}^{1},\omega_{0})$. At this point, the stratum has the usual period coordinates $(r_{1},\dots,r_{p-1},t)$, where $r_{i}$ is the residue at the pole $q_{i}$ and $t$ is a period between $z_{1}$ and $z_{2}$. Hence, the differential $\omega_0$ 
corresponds to $(0,\dots,0,c)$ under these coordinates, where $c$ is nonzero, and nearby points in the isoresidual fiber correspond to the coordinates $(s\lambda_{1},\dots,s\lambda_{p-1}, f(s))$, where $f(s)$ is a holomorphic function such that $f(0)=c$. Therefore, the differential $\omega_{\lambda}$ is locally given by $d\left(\frac{f(s)}{s}\right) = \frac{f'(s)s-f(s)}{s^{2}}ds$, where the Taylor series of the numerator at $0$ is of the form $f(0) + o(s^{2})$. Hence, $\omega_\lambda$ has a pole of order $2$ whose residue vanishes. 
\end{proof}
\par
\begin{ex}
 The flat geometric intuition behind  Proposition~\ref{prop:flat-boundary} is that the two zeros can move arbitrarily far away from each other under the flat metric. An example of an isoresidual deformation in the stratum $\mathcal{H}(2,3,-1,-2,-4)$ illustrating the second case is given in Figure~\ref{fig:poleordresup}.
 \begin{figure}[ht]
\begin{tikzpicture}[scale=1,decoration={
    markings,
    mark=at position 0.5 with {\arrow[very thick]{>}}}]
    
      \begin{scope}[xshift=-1.5cm,yshift=-.5cm]
\draw (-.25,0) coordinate (a) -- (.25,0) coordinate[pos=.5] (c) coordinate (b);

\fill[fill=black!10] (a) -- (b) -- ++(0,1) --++(-.5,0) -- cycle;
\draw  (a) -- ++(0,1)  coordinate[pos=.5] (d);
\draw  (b) -- ++(0,1)  coordinate[pos=.5] (e);
\fill(a) circle (2pt);
\fill(b) circle (2pt);
 
 \node[below] at (c) {$1$};
\node[left] at (d) {$2$};
\node[right] at (e) {$2$};
    \end{scope}
    
\begin{scope}[xshift=2.3cm,yshift=0cm]
      \fill[fill=black!10] (0,0) circle (1cm);
      \draw (-.25,0) coordinate (a) -- (.25,0) coordinate[pos=.5] (c) coordinate (b);
      \draw (-.75,0) coordinate (z) -- (a) coordinate[pos=.5] (x);
\filldraw[fill=white] (z) circle (2pt);
\fill (a) -- (-.25,-.09) arc [start angle=-90, end angle=-180, radius=.9mm] -- ++(.09,0);
\fill (b) circle (2pt);
    \fill[white] (a) -- (b) -- ++(0,-1) --++(-.5,0) -- cycle;
 \draw  (b) --  (a);
 \draw (a) -- ++(0,-.9) coordinate[pos=.5] (d);
 \draw (b) -- ++(0,-.9)coordinate[pos=.5] (e);
\node[above] at (a) {$\beta$};
\node[below] at (x) {$\alpha$};
\node[left] at (d) {$4$};
\node[right] at (e) {$4$};
    \end{scope} 
    
    \begin{scope}[yshift=-3cm]
      \fill[fill=black!10] (0,0) circle (1.5cm);
  \fill  (0,0) circle (2pt);
 \draw (-.5,0) coordinate (a) -- (.5,0) coordinate[pos=.25] (c) coordinate[pos=.5] (d) coordinate[pos=.75] (e) coordinate (b);
 \draw (-1,0) coordinate (z) -- (a) coordinate[pos=.5] (y);
 \draw (z) -- ++(-.5,0)coordinate[pos=.5] (x);
    \filldraw[fill=white] (z) circle (2pt);
    \node[below] at (x) {$8$};    \node[above] at (x) {$9$};
     \node[below] at (e) {$1$}; 
        \node[below] at (a) {$\beta$};    \node[above] at (y) {$\alpha$};
\fill (a) -- (-.5,.09) arc [start angle=90, end angle=180, radius=.9mm] -- ++(.09,0);

\fill[] (b) circle (2pt);
    \fill[white] (a) -- (b) -- ++(0,1.5) --++(-1,0) -- cycle;
 \draw  (b) --  (a);
 \draw (a) -- ++(0,1.4) coordinate[pos=.5] (f);
 \draw (b) -- ++(0,1.4)coordinate[pos=.5] (g);
  \draw  (b) --++  (1,0) coordinate[pos=.5] (h) ;

\node[left] at (f) {$5$};
\node[right] at (g) {$5$};
\node[above] at (h) {$6$};\node[below] at (h) {$7$};

       \fill[fill=black!10] (2.5,.75) circle (.7cm);    
  \fill (2.5,.75) circle (2pt);
  \draw (2.5,.75) --++(.7,0)  coordinate[pos=.5] (i) ;
\node[below] at (i) {$6$};\node[above] at (i) {$7$};
  
       \fill[fill=black!10] (2.5,-.75) circle (.7cm);    
         \draw (2.5,-.75) --++(-.7,0) coordinate[pos=.5] (j) ;
  \filldraw[fill=white] (2.5,-.75) circle (2pt);
  \node[below] at (j) {$9$};\node[above] at (j) {$8$};
    \end{scope}

    \draw[->] (3.5,-2.5) -- (6,-2.5) coordinate[pos=.5] (a);
    \node[above] at (a) {$\alpha,\beta \to \infty$};
\begin{scope}[xshift=8cm]
 
   \begin{scope}[xshift=-1.5cm,yshift=-.5cm]
\draw (-.25,0) coordinate (a) -- (.25,0) coordinate[pos=.5] (c) coordinate (b);

\fill[fill=black!10] (a) -- (b) -- ++(0,1) --++(-.5,0) -- cycle;
\draw  (a) -- ++(0,1)  coordinate[pos=.5] (d);
\draw  (b) -- ++(0,1)  coordinate[pos=.5] (e);
\fill(a) circle (2pt);
\fill(b) circle (2pt);
 
 \node[below] at (c) {$1$};
\node[left] at (d) {$2$};
\node[right] at (e) {$2$};
    \end{scope}
 
 \begin{scope}[xshift=2.3cm,yshift=0cm]
      \fill[fill=black!10] (0,0) circle (1cm);
      \draw (-.25,0) coordinate (a) -- (.25,0) coordinate[pos=.5] (c) coordinate (b);
      \draw  (a) -- ++ (-.75,0) coordinate[pos=.6] (x);
\fill (a) -- (-.25,-.09) arc [start angle=-90, end angle=-180, radius=.9mm] -- ++(.09,0);
\fill (b) circle (2pt);
    \fill[white] (a) -- (b) -- ++(0,-1) --++(-.5,0) -- cycle;
 \draw  (b) --  (a);
 \draw (a) -- ++(0,-.9) coordinate[pos=.5] (d);
 \draw (b) -- ++(0,-.9)coordinate[pos=.5] (e);
\node[above] at (a) {$\beta''$};
\node[below] at (x) {$\alpha''$};
\node[left] at (d) {$4$};
\node[right] at (e) {$4$};
    \end{scope} 
    
     \begin{scope}[xshift=.2cm,yshift=-1cm]
        \fill[fill=black!10] (0,0) circle (1cm);
      \draw (-.25,0) coordinate (a) -- (.25,0) coordinate[pos=.5] (c) coordinate (b);
      \draw  (a) -- ++ (-.75,0) coordinate[pos=.6] (x);
\fill (a) -- (-.25,.09) arc [start angle=90, end angle=180, radius=.9mm] -- ++(.09,0);
\fill (b) circle (2pt);
    \fill[white] (a) -- (b) -- ++(0,1) --++(-.5,0) -- cycle;
 \draw  (b) --  (a);
 \draw (a) -- ++(0,.9) coordinate[pos=.5] (d);
 \draw (b) -- ++(0,.9)coordinate[pos=.5] (e);
\node[below] at (a) {$\beta''$};
\node[above] at (x) {$\alpha''$};
\node[left] at (d) {$5$};
\node[right] at (e) {$5$};
    \end{scope}    

\begin{scope}[yshift=-3cm]
    
       \fill[fill=black!10] (2.5,-.5) circle (.7cm);    
  \draw  (2.5,-.5) --++(.7,0)  coordinate[pos=.5] (i) ;
 \filldraw[fill=white] (2.5,-.5) circle (2pt);
\node[below] at (i) {$\alpha'$};\node[above] at (i) {$\beta'$};
  
       \fill[fill=black!10] (.75,-.5) circle (.7cm);    
  \draw  (.75,-.5) --++(.7,0) coordinate[pos=.5] (j) ;
   \draw  (.75,-.5) --++(-.7,0) coordinate[pos=.5] (k) ;
\filldraw[fill=white]  (.75,-.5) circle (2pt);
\node[below] at (j) {$\beta'$};\node[above] at (j) {$\alpha'$};
\node[below] at (k) {$8$};\node[above] at (k) {$9$};
  
        \fill[fill=black!10] (-1,-.5) circle (.7cm);    
  \draw (-1,-.5) --++(-.7,0) coordinate[pos=.5] (j) ;
     \filldraw[fill=white] (-1,-.5) circle (2pt);
  \node[below] at (j) {$9$};\node[above] at (j) {$8$};
    \end{scope}  
\end{scope}

\end{tikzpicture}
 \caption{An isoresidual deformation in $\mathcal{H}(2,3,-1,-2,-4)$ converging to a multi-scale differential of cherry type, where $I=\emptyset$, $J$ contains the simple pole,  and $K$ contains the other two poles.}\label{fig:poleordresup}
\end{figure}
In the picture, the zero $z_{1}$ of order $2$ is labeled in white while the zero $z_{2}$ of order $3$ is in black.
In the deformation, both saddle connections corresponding to $\alpha$ and $\beta$ go to infinity, in such a way that the difference of their periods $\int_\alpha \omega -\int_\beta \omega$ remains constant. In the picture, we only illustrate the bottom components of the resulting multi-scale differential. Note that the component containing $z_{1}$ is unstable. Hence, its stabilization is of the form (2') in Figure~\ref{fig:cherry}.
\par
To compute the orders of the poles, we can proceed in the following way. Consider the pair of saddle connections $(\alpha,\beta)$ with very large periods as in the picture. Then start rotating them such that their difference remains constant. In the upper right of the first picture, this pair comes back to its original position after a  rotation of angle $2\pi$. In the lower part of the picture, the saddle connections come back to their initial positions only after  a rotation of angle $4\pi$. Therefore, the whole picture comes back to its initial position after a rotation of $(\alpha,\beta)$ of angle $4\pi$. This implies that the corresponding pole has order $-3$.

Moreover, we can determine the residue of the pole as follows. Note that after a rotation of angle $4\pi$, the pair $(\alpha,\beta)$ does not come to its original position. Indeed, they are both shortened by the difference of the original periods of $\beta$ and $\alpha$.  
It implies that the difference $\int_{\beta}\omega -\int_{\alpha}\omega$ 
is the residue of the corresponding pole in the isoresidual curve, which is equal to the prescribed residue of the simple pole of the differentials parameterized in the isoresidual curve. Therefore, this confirms Equation~\eqref{eq:residuecase2} of Proposition~\ref{prop:flat-boundary}. 

Finally, in Figure~\ref{fig:isores1}, we can produce another example by letting $\gamma$ go to infinity. In that case, we have $I=\emptyset$, $J$ contains the poles of orders $-1$ and $-2$, and $K$ contains the pole of order $-4$.
\end{ex}
\par 
We remark that Propositions~\ref{prop:zeros}, \ref{prop:simple-pole}, and \ref{prop:flat-boundary} fully describe the zeros, simple poles, higher-order poles, and their orders for the translation structure $\omega_\lambda$ on $\overline{\mathcal{F}}_{\lambda}$. In particular, they determine the Euler characteristic $2-2g(\overline{\mathcal{F}}_{\lambda})$. Later, in Proposition~\ref{prop:T=F}, we will recover the  Euler characteristic using the alternative method of intersection theory on the multi-scale compactification of strata of differentials. 
\par

\subsection{Cylinders in isoresidual fibers}\label{sub:cylinders}
\par 
We observe that every cylinder belonging to a generic isoresidual fiber must be infinite.
\par 
\begin{prop}\label{prop:cylinder}
For a stratum $\mathcal{H}(\mu)$ with $n=2$ zeros and every configuration of residues $\lambda \in \mathcal{R}_{p} \setminus \mathcal{A}_{p}$, any cylinder in the closure $\overline{\mathcal{F}}_{\lambda}$ of the isoresidual fiber $\mathcal{F}_{\lambda}$ bounds a simple pole.
\end{prop}
\par 
\begin{proof}
We assume that $(\overline{\mathcal{F}}_{\lambda},\omega_{\lambda})$ admits a closed geodesic $\gamma$. Up to scaling, we can assume that $\gamma$ is horizontal and has unit length. We will show that the closed geodesic $\gamma$ belongs to an infinite cylinder.
\par
In $(\overline{\mathcal{F}}_{\lambda},\omega_{\lambda})$, we introduce the flow $(A_{t})_{t \in \mathbb{R}}$ that acts on the parametrized differentials by keeping the residues unchanged while the relative period $z_{1,2}$ joining the two zeros $z_1$ and $z_2$ becomes $z_{1,2}+t$. Any differential $\omega$ parametrized in $\gamma$ is periodic with a period of $1$ under the action of $(A_{t})_{t \in \mathbb{R}}$.
\par
We first consider the case where, for some differential $\omega$ of $\mathcal{F}_{\lambda}$ contained in the closed geodesic $\gamma$, $(\mathbb{CP}^{1},\omega)$ contains a cylinder $\mathcal{C}$ of finite area foliated by horizontal closed geodesics. Since $(\mathbb{CP}^{1},\omega)$ is of genus zero, the cylinder $\mathcal{C}$ decomposes the surface into two connected components, each containing exactly one zero of $\omega$. These two components are unchanged by the action of $(A_{t})_{t \in \mathbb{R}}$, while $A_{1}$ induces a Dehn twist on the cylinder $\mathcal{C}$ (recall that $A_{1}\omega=\omega$). Taking the period of a diagonal of the cylinder as a local coordinate, we can make its imaginary part arbitrarily large in absolute value. This proves that $\gamma$ belongs to an infinite cylinder bounded by some simple pole of $\omega_{\lambda}$.
\par
In the following, we will assume that for any differential $\omega$ of $\mathcal{F}_{\lambda}$ that is contained in the closed geodesic~$\gamma$, there is no horizontal cylinder of finite area in $(\mathbb{CP}^{1},\omega)$. We then introduce $\Phi_{\omega} \subset \mathbb{R}$ as follows: $s \in \Phi_{\omega}$ if $s$ is the imaginary part of the period of a saddle connection in $(\mathbb{CP}^{1},\omega)$ joining $z_1$  and $z_2$ (with this orientation). The subset $\Phi_{\omega}$ is nonempty because any length-minimizing path between the two zeros (for the flat metric induced by $\omega$) contains a saddle connection between them.
\par
We first observe that if $0 \in \Phi_{\omega}$, then $(\mathbb{CP}^{1},\omega)$ has a horizontal saddle connection,  and the two zeros of $\omega$ collide in finite time under the action of $(A_{t})_{t \in \mathbb{R}}$. This is impossible because the differential $\omega$ belongs to the closed geodesic $\gamma$ and is supposed to be periodic under the action of $(A_{t})_{t \in \mathbb{R}}$.
\par
Since $(\mathbb{CP}^{1},\omega)$ has no horizontal cylinder of finite area, Proposition~\ref{prop:INFINITE-SC-Chara} implies that the slopes of saddle connections of $(\mathbb{CP}^{1},\omega)$ cannot accumulate in the horizontal direction. It follows that for any $s >0$, $\Phi_{\omega} \cap [-s,s]$ is a finite set (the number of saddle connections of length smaller than some bound is always finite). We deduce that the infimum $\delta_{\omega}$ of the absolute values of elements of $\Phi_{\omega}$ is realized by some saddle connection of $(\mathbb{CP}^{1},\omega)$.
\par
Along the action of the flow, a saddle connection  $\alpha_{t}$ of $A_{t} \omega$ joining the two distinct zeros is destroyed when its interior is crossed by a conical singularity (two conical singularities cannot collide because $\gamma$ remains in the regular locus of $\mathcal{F}_{\lambda}$). Consequently, our initial saddle connection is cut into two saddle connections, one of which still joins the two zeros. Observe that this latter saddle connections persists under small perturbations of $t$, and the imaginary part of its period is (in absolute value) strictly smaller than that of $\alpha_{t}$. It follows that $\delta_{A_{t}\omega}$ is constant.
\par
The same argument shows that a saddle connection $\alpha$ of $(\mathbb{CP}^{1},\omega)$ that realizes $\delta_{\omega}$ persists under the action of $(A_{t})_{t \in \mathbb{R}}$. It follows that for a saddle connection $\alpha$ in $(\mathbb{CP}^{1},\omega)$ of period $z$, there is another saddle connection in $(\mathbb{CP}^{1},\omega)$ of period $z+1$. Consequently, there is an infinite sequence of saddle connections in $(\mathbb{CP}^{1},\omega)$ whose directions approach the horizontal direction. Proposition~\ref{prop:INFINITE-SC-Chara} proves that there is a horizontal cylinder of finite area in $(\mathbb{CP}^{1},\omega)$. This case has already been settled. This concludes the proof.
\end{proof}
\par 
Translation surfaces with infinitely many saddle connections have been characterized in Corollary~5.13 of \cite{tahar} as those that contain a cylinder of finite area.
\par 
\begin{cor}\label{cor:SCFinite}
For a stratum $\mathcal{H}(\mu)$ with $n=2$ zeros and any configuration of residues $\lambda \in \mathcal{R}_{p} \setminus \mathcal{A}_{p}$, the closure $\overline{\mathcal{F}}_{\lambda}$ of the isoresidual fiber $\mathcal{F}_{\lambda}$ contains finitely many saddle connections.
\end{cor}
\par 
\subsection{Isoresidual fibers over configurations of real residues}\label{sub:FiberReal}
\par 
In a translation surface $(X,\omega)$ of genus zero with real residues, the group generated by the absolute periods of $\omega$ is contained in $\mathbb{R}$. Consequently, even though  the relative periods between two zeros are defined up to the addition of an absolute period, the imaginary part of the relative period between any two fixed zeros does not depend on the integration path.
\par 
\begin{lem}\label{lem:Imaginary}
For a stratum $\mathcal{H}(a_{1},a_{2},-b_{1},\dots,-b_{p})$ and a configuration $\lambda$ of real residues, there exists a harmonic function $\mathcal{I}\colon\mathcal{F}_{\lambda} \longrightarrow \mathbb{R}$ such that for any differential $\omega \in \mathcal{F}_{\lambda}$ and any path $\gamma$ joining the two zeros in $(\mathbb{CP}^{1},\omega)$, we have 
$$
\Im \left( \int_{\gamma} \omega \right) = \mathcal{I}(\omega)\,.
$$
\end{lem}
\par 
\begin{proof}
In $(\mathbb{CP}^{1},\omega)$, the integral of $\omega$ along two distinct paths between the two zeros $z_{1}$ and $z_{2}$ differs by an absolute period of $\omega$. Given that the latter is real, $\Im\big(\int_{\gamma} \omega\big)$ is a real function $\mathcal{I}$ of $\omega$, which is globally defined on $\mathcal F_\lambda$. We observe that for any local period coordinate of the translation structure of $\mathcal{F}_{\lambda}$, $\mathcal{I}$ coincides with its imaginary part. In other words, $\mathcal{I}$ is the imaginary part of a holomorphic function on a Riemann surface, and it is therefore a harmonic function.
\end{proof}
\par 
For an isoresidual fiber $\mathcal{F}_{\lambda}$, where $\lambda$ consists of real residues, we define the \textit{real locus} $\mathbb{R}\mathcal{F}_{\lambda} \subset \mathcal{F}_{\lambda}$ as the locus formed by differentials for which the relative period between the two zeros is also real. Equivalently, $\mathbb{R}\mathcal{F}_{\lambda} = \mathcal{I}^{-1}(0)$.
\par 
\begin{lem}\label{lem:RealLocus}
In a stratum $\mathcal{H}(a_{1},a_{2},-b_{1},\dots,-b_{p})$, for any configuration of real residues $\lambda\in \mathcal{R}_p\setminus\mathcal{A}_{p}$, every zero of $(\overline{\mathcal{F}}_{\lambda},\omega_{\lambda})$ belongs to the closure of the real locus $\mathbb{R}\mathcal{F}_{\lambda}$. Moreover, every saddle connection of $(\overline{\mathcal{F}}_{\lambda},\omega_{\lambda})$ is horizontal.
\end{lem}
\par 
\begin{proof}
When the two zeros of the parametrized translation surfaces collide, the relative period computed by integrating the differential on the shrinking saddle connection tends to zero. In particular, its imaginary part (well-defined since the absolute periods of the differential are real) also tends to zero. Additionally, a zero of order $a_1+a_2$ can be locally split via a horizontal opening direction to two zeros of order $a_1$ and $a_2$. Thus, every zero of $\omega_{\lambda}$ belongs to the closure of the real locus.
\par
Following Lemma~\ref{lem:Imaginary}, the imaginary part of any local period coordinate of the translation atlas in $\mathcal{F}_{\lambda}$ is given by the harmonic function $\mathcal{I}$. Saddle connections of $(\overline{\mathcal{F}}_{\lambda},\omega_{\lambda})$ have constant slopes in the translation atlas, and their endpoints belong to the zero set of $\mathcal{I}$. It follows that every saddle connection of $(\overline{\mathcal{F}}_{\lambda},\omega_{\lambda})$ is completely contained in the real locus $\mathbb{R}\mathcal{F}_{\lambda} = \mathcal{I}^{-1}(0)$ and is therefore horizontal.
\end{proof}
\par 
Any differential $\omega$ in $\mathbb{R}\mathcal{F}_{\lambda}$ can be described by the following data:
\begin{itemize}
    \item the decorated graph $\mathfrak{gr}(\omega)$ formed by $p$ vertices and $p$ edges (see Section~\ref{sub:RealGraph});
    \item the lengths of the $p$ saddle connections of $\omega$.
\end{itemize}
\par 
Since $\mathfrak{gr}(\omega)$ has $p$ vertices and $p$ edges, it contains a uniquely defined loop. This loop has:
\begin{itemize}
    \item \textit{coherent orientation} if all the edges of the loop have the same orientation;
    \item \textit{incoherent orientation} if two consecutive edges in the loop have opposite orientations.
\end{itemize}
\par 
Because the distinction between coherent or incoherent orientation depends only on the graph, this notion induces a dichotomy on the connected components of $\mathbb{R}\mathcal{F}_{\lambda}$.
\par 
\begin{prop}\label{prop:REALlocus}
We consider a stratum $\mathcal{H}(\mu)$ with $n=2$ zeros and any configuration of real residues $\lambda$ lying outside $\mathcal{A}_{p}$. Let $\mathcal{C}$ be a connected component of the real locus $\mathbb{R}\mathcal{F}_{\lambda}$ of the isoresidual fiber $\mathcal{F}_{\lambda}$. Then one of the following statements holds:
\begin{itemize}
    \item If the decorated graph of the differentials in $\mathcal{C}$ has coherent orientation, then $\mathcal{C}$ is an infinite horizontal trajectory between a zero and a pole of $(\overline{\mathcal{F}}_{\lambda},\omega_{\lambda})$.
    \item If the decorated graph of the differentials in $\mathcal{C}$ has incoherent orientation, then $\mathcal{C}$ is a horizontal saddle connection of $(\overline{\mathcal{F}}_{\lambda},\omega_{\lambda})$ joining two zeros of $\omega_{\lambda}$. Moreover, its length is $|\sum\limits_{j \in J} \lambda_{j}|$ for some subset $J$ of $\lbrace{ 1,\dots,p \rbrace}$.
\end{itemize}
\end{prop}
\par 
\begin{proof}
We first consider the case where the decorated graph of the differentials in $\mathcal{C}$ has coherent orientation. The fact that the residues at the poles are fixed means that the difference between the lengths of the saddle connections corresponding to two adjacent oriented edges of the loop is fixed. This is the only constraint on them. It follows that the lengths of the saddle connections corresponding to the edges of the loop (in other words, the saddle connections joining two distinct zeros) can be made arbitrarily large. Thus, $\mathcal{C}$ is an infinite arc in $(\overline{\mathcal{F}}_{\lambda},\omega_{\lambda})$. 
\par 
As the lengths of these edges become infinite, the parametrized differential degenerates in the way described in Sections~\ref{sub:boundary} and~\ref{subsub:HigherPoles}. This corresponds to a pole of $\omega_{\lambda}$. Conversely, if we decrease the lengths of these saddle connections, at some point one of them shrinks to zero (only one at a time; otherwise, a resonance equation would hold between the residues). In other words, at that moment, the two zeros collide, leading to a degeneration corresponding to a zero of $\omega_{\lambda}$.
\par
Next, we discuss the case of a decorated graph with incoherent orientation. The edges of the loop form two subsets $E_{1}$ and $E_{2}$ according to the orientation of the edges. Again, the only degree of freedom is the length $L$ of the saddle connection corresponding to an arbitrary edge in $E_{1}$. Indeed, the lengths of the closed saddle connections are determined by the sums of the residues of the poles they enclose. As $L$ increases, the lengths of the other saddle connections in $E_{1}$ also increase, while the lengths of the saddle connections in $E_{2}$ decrease. Since all these lengths must remain positive, $L$ is constrained to an interval. The connected component $\mathcal{C}$ is thus a horizontal segment whose endpoints correspond to the collision of the two zeros. Therefore, $\mathcal{C}$ is a horizontal saddle connection of $(\overline{\mathcal{F}}_{\lambda},\omega_{\lambda})$.
\par
It remains to compute the length of this saddle connection. At the left (resp. right) endpoint of $\mathcal{C}$, exactly one saddle connection from $E_{1}$ (resp. $E_{2}$) shrinks to zero. It is impossible for two saddle connections to shrink at the same time because their lengths are related by equations involving the residues, which would lead to a resonance equation. We denote by $e_{1}$ and $e_{2}$ the edges of the loop corresponding to the saddle connections that shrink at the left and right endpoints of $\mathcal{C}$, respectively. By removing $e_{1}$ and $e_{2}$ from the decorated graph, we decompose it into exactly two connected components, $G_{1}$ and $G_{2}$. Denote by $I \sqcup J$ the corresponding bipartition of the set of poles $\lbrace{ 1,\dots,p \rbrace}$. We find that the lengths of the two saddle connections corresponding to the edges $e_{1}$ and $e_{2}$ sum to $|\sum\limits_{i \in I} \lambda_{i}|=|\sum\limits_{j \in J} \lambda_{j}|$. Since these  represent the only saddle connections that can shrink as the parametrized differential is deformed along $\mathcal{C}$, this gives the length of $\mathcal{C}$.
\end{proof}
\par 
\subsection{Periods of saddle connections}
\par 
Deformations of translation surfaces are governed by the deformation of the periods of their saddle connections. Since one-dimensional isoresidual fibers are translation surfaces, it is essential to understand how their saddle connections change in relation to the underlying configuration of residues. In the following, we will prove the linear dependence of the periods of an isoresidual fiber in terms of $\lambda$.
\par 
\begin{prop}\label{prop:IntegerWeights}
For a stratum $\mathcal{H}(\mu)$ of differentials on $\mathbb{C}\mathbb{P}^1$ with two zeros, for any configuration of residues $\lambda \in \mathcal{R}_{p} \setminus \mathcal{A}_{p}$, and for any relative homology class $[\gamma]$ of $H_{1}(\overline{\mathcal{F}}_{\lambda} \setminus P_{\omega_{\lambda}}, Z_{\omega_{\lambda}})$, we have
$$
\int_{[\gamma]} \omega_{\lambda} = \sum\limits_{j=1}^{p} w_{j}\lambda_{j}\,,
$$
where $w_{j} \in \mathbb{Z}$ for all $1 \leq j \leq p$.
\end{prop}
\par 
\begin{proof}
We first consider the case of an isoresidual fiber $(\overline{\mathcal{F}}_{\lambda},\omega_{\lambda})$ where $\lambda$ is a configuration of real residues. Following Lemma~\ref{lem:RealLocus}, all the saddle connections of $(\overline{\mathcal{F}}_{\lambda},\omega_{\lambda})$ belong to the real locus $\mathbb{R}\mathcal{F}_{\lambda}$. Proposition~\ref{prop:REALlocus} shows that their periods are given by partial sums of residues. As with any translation surface, saddle connections generate the relative homology group $H_{1}(\overline{\mathcal{F}}_{\lambda} \setminus P_{\omega_{\lambda}},Z_{\omega_{\lambda}})$. It follows that the integral of $\omega_{\lambda}$ along any relative homology class is a linear combination of the residues $\lambda_{1},\dots,\lambda_{p}$ with integer coefficients. Using analytic continuation, we deduce that the same claim holds for any $\lambda \in \mathcal{R}_{p} \setminus \mathcal{A}_{p}$.
\end{proof}
\par 
Actually, the fact that the period of any saddle connection of the isoresidual fiber is a partial sum of residues holds for any fiber, not just those corresponding to configurations of real residues.
\par 
\begin{prop}\label{prop:SCperiod}
Consider a stratum $\mathcal{H}(a_{1},a_{2},-b_{1},\dots,-b_{p})$ and a configuration of residues 
$\lambda \in \mathcal{R}_{p} \setminus \mathcal{A}_{p}$. For any saddle connection $\gamma$ in the closure $\overline{\mathcal{F}}_{\lambda}$ of the isoresidual fiber $\mathcal{F}_{\lambda}$, we have
$$
\int_{\gamma} \omega_{\lambda} = \sum\limits_{j \in J} \lambda_{j}
$$
for some subset $J$ of $\lbrace{ 1,\dots,p \rbrace}$.
\end{prop}
\par 
\begin{proof}
Following Corollary~\ref{cor:SCFinite}, the isoresidual fiber $(\overline{\mathcal{F}}_{\lambda},\omega_{\lambda})$, as a translation surface, has finitely many saddle connections. Each of these saddle connections has a period that is a linear combination of $\lambda_{1},\dots,\lambda_{p}$ with integer coefficients. We will prove that this linear combination is, in fact, a partial sum.
\par
Up to a small perturbation of $\lambda$, we assume that saddle connections of $(\overline{\mathcal{F}}_{\lambda},\omega_{\lambda})$ have the same slope if and only if their periods are given by the same coefficients (up to multiplication by a constant real factor). Any saddle connection $\gamma_{1}$ belongs to a maximal family $\gamma_{1},\dots,\gamma_{t}$ of saddle connections of $(\overline{\mathcal{F}}_{\lambda},\omega_{\lambda})$ that have the same slope $\theta \in \mathbb{RP}^{1}$.
\par
Since $(\overline{\mathcal{F}}_{\lambda},\omega_{\lambda})$ has finitely many saddle connections, there is an open interval $U \in \mathbb{RP}^{1}$ that contains $\theta$ but no other slopes  of saddle connections. Up to the action of ${\rm GL}_{2}^{+}(\mathbb{R})$ (recall that this action commutes with the isoresidual map; see Section~\ref{subsub:action}), we can assume that $\theta$ is the vertical direction, while the other saddle connections have slopes arbitrarily close to the horizontal direction. 
\par 
By using the subgroup of ${\rm GL}_{2}^{+}(\mathbb{R})$ that preserves the horizontal direction while contracting exponentially fast a direction arbitrarily close to the vertical direction, we can conjugate $(\overline{\mathcal{F}}_{\lambda},\omega_{\lambda})$ to the fiber corresponding to a configuration $\lambda'$ that is arbitrarily close to a configuration of real residues, where the saddle connections conjugate to $\gamma_{1},\dots,\gamma_{t}$ are the smallest saddle connections of the surface. 
\par 
It follows that these saddle connections are obtained by deformations of the saddle connections of a fiber corresponding to a configuration with real residues. Since such saddle connections have periods given by partial sums of residues (see Proposition~\ref{prop:REALlocus}), the same holds for $\gamma_{1},\dots,\gamma_{t}$. 
\end{proof}
\par 
\begin{rmk}
Observe that Proposition~\ref{prop:SCperiod} imposes highly nontrivial constraints on the geometry of isoresidual fibers. Indeed, when three saddle connections $\gamma_{1},\gamma_{2},\gamma_{3}$ form the sides of a triangle, if the periods of the homology classes $[\gamma_{1}]$ and $[\gamma_{2}]$ are partial sums of the underlying residues, there is no reason for the period $[\gamma_{3}]=-[\gamma_{1}]-[\gamma_{2}]$ to also be a partial sum of residues. 
\end{rmk}
\par 
\subsection{Formal proof of Theorem~\ref{thm:MAIN1}}\label{sub:formalMain1}
\par 
We summarize all the results of Section~\ref{sec:Translation} in Theorem~\ref{thm:MAIN1}.
\par 
\begin{proof}[Proof of Theorem~\ref{thm:MAIN1}]
The translation structure of $(\mathcal{F}_{\lambda},\omega_{\lambda})$ is described in Section~\ref{sub:translationStructure}, while its compactification is given in Section~\ref{sub:boundary}. It is proved in Sections~\ref{subsub:zeros},~\ref{subsub:SimplePoles}, and~\ref{subsub:HigherPoles} that $\omega_{\lambda}$ extends to $\overline{\mathcal{F}}_{\lambda}$ as a meromorphic differential. The periods of saddle connections are computed in Proposition~\ref{prop:SCperiod}, and we know from Corollary~\ref{cor:SCFinite} that there are finitely many of them. 
\par
We deduce that for a configuration of real residues in $\mathcal{R}_{p} \setminus \mathcal{A}_{p}$, all the saddle connections of $(\overline{\mathcal{F}}_{\lambda},\omega_{\lambda})$ are horizontal. Such a translation surface cannot correspond to a holomorphic differential, as this would imply that the area of the holomorphic translation surface is zero. Therefore,   $\omega_{\lambda}$ must have at least one pole.
\par
The fact that the isoresidual fibration is nonsingular outside $\mathcal{A}_{p}$ follows from Proposition~\ref{prop:SCperiod}. Since the periods of saddle connections are given by partial sums of residues, the translation structure of an isoresidual fiber $(\overline{\mathcal{F}}_{\lambda},\omega_{\lambda})$ cannot degenerate as long as $\lambda$ stays away from the resonance hyperplanes.
\par
Finally, if $\mathcal{F}_{\lambda}$ is not connected, then the monodromy of the fibration acts transitively on the set of connected components; otherwise, the whole stratum would not be connected. The singularity pattern of the translation structure of each connected component of 
$(\overline{\mathcal{F}}_{\lambda},\omega_{\lambda})$ is invariant under deformation. It follows that these patterns are the same for each connected component of $\mathcal{F}_{\lambda}$.
\end{proof}
\par 
\section{The Euler characteristic of generic isoresidual fibers}\label{sec:Euler}
In this section, we study the combinatorial structure for the Euler 
characteristic of generic isoresidual fibers. In particular, we will prove Theorem~\ref{thm:chambers}. 
\par 
\subsection{The wall and chamber structures}\label{sub:Discretecombinatorics}
\par 
We first recall the wall and chamber structures described in Definition \ref{defn:SPS}.
Given $p \geq 1$, the singularity pattern space $\mathcal{SP}_{p}$ is  $\left\{ (x_{1},x_{2},y_{1},\dots,y_{p}) \in \mathbb{R}_{>0}^{p+2}~\mid~x_{1}+x_{2}= \sum y_{j} \right\}$ endowed with the family $\mathcal{W}_{p}$ of hyperplanes $W_{1}(I),W_{2}(I)$ of Equation~\eqref{eq:hyptype1} and $W_{3}(I,K,L),W_{4}(J,K,M)$ of Equations~\eqref{eq:hyptype2} and~\eqref{eq:hyptype3}, where $I \sqcup J \sqcup K$ is a partition of the index set of the poles $\lbrace{ 1,\dots,p \rbrace}$ into three disjoint subsets, and $L$ and $M$ are (possibly intersecting) arbitrary subsets of $K$.
\par 
In the above, $x_i$ corresponds to the zero order plus one, i.e., $a_i+1$ for $i=1,2$, while $y_j$ corresponds to the pole order $b_j$ for $j=1,\ldots,p$. Now, we can show that in each chamber of the singularity pattern space $\mathcal{SP}_p$, separated by the walls in $\mathcal{W}_{p}$, the Euler characteristic of the isoresidual curve $\overline{\mathcal{F}}_\lambda$ has the desired piecewise homogeneous structure. 
\par 
We say that a function is {\em homogeneous of degree $k$} if scaling all variables by $\lambda$ scales the function value by $\lambda^k$. For example, $\gcd(x_1 - y_2, x_2 - y_3)$ is homogeneous of degree one for 
$x_i, y_j\in \mathbb{Z}$. Now, we can  prove Theorem~\ref{thm:chambers} through the following more detailed statement. 
\par
\begin{prop}\label{prop:chambers}
For $\lambda \in \mathcal{R}_{p}\setminus \mathcal{A}_{p}$ and $(a_1+1, a_2+1, b_1, \ldots, b_p)$ in any given chamber of $\mathcal{SP}_p$, the Euler characteristic of $\overline{\mathcal F}_\lambda$ is a sum of homogeneous components of degree from $0$ up to $p-1$ in terms of the variables $a_{1}+1,a_{2}+1,b_{1},\dots,b_{p}$. Moreover, the terms containing the $\gcd$ have degree at most $p-2$. 
\end{prop}
\par 
\begin{proof}
    According to Propositions~\ref{prop:zeros}, \ref{prop:simple-pole}, and~\ref{prop:flat-boundary} (see also Proposition $\ref{prop:T=F}$ below), we can express the Euler characteristic of $\overline{\mathcal F}_\lambda$ in terms of its boundary points and the residueless locus. It then  suffices to verify that the contributions from these loci to the Euler characteristic satisfy the desired claim in each chamber of $\mathcal{SP}_p$.  \par 
    Recall that the boundary of $\overline{\mathcal{F}}_\lambda$ consists of three types of graphs. The first type is given by two-level graphs $\Gamma_{z_1,z_2}$, where both zeros are in the unique bottom component and all the poles are in the unique top component, with the residues of the poles determined by $\lambda$. By Proposition~\ref{prop:zeros}, their total contribution to the Euler characteristic of $\overline{\mathcal F}_\lambda$ is a polynomial of degree $p-1$ in the variables $a_1+1$ and $a_2+1$, whose expression is independent of the chambers. 
    \par 
    The second type is given by horizontal graphs satisfying the equations $(a_1+1)-\sum_{i\in I}b_i=0=(a_2+1)-\sum_{j\in I^c}b_j$ of type $\beta_1$ and $\beta_2$ 
    in $\mathcal{W}_p$, where $I\subsetneq \{1,\ldots,p\}$ parametrizes the poles in the same component that contains the zero $z_1$, and the residues of the poles are determined by $\lambda$. 
    By Proposition~\ref{prop:simple-pole}, their  contribution to the Euler characteristic of $\overline{\mathcal F}_\lambda$ is a polynomial of degree $p-2$. 
    \par 
    The third type, which we call cherry graphs as described in Proposition~\ref{prop:flat-boundary}, is given by a partition of the poles $I\sqcup J \sqcup K=\{1,\ldots,p\}$, where the zero $z_1$ belongs to a bottom component with the poles indexed by~$I$, the zero $z_2$ belongs to the other bottom component with the poles indexed by $J$, and the remaining poles are in the top component, indexed by $K$. These cherry graphs satisfy the conditions $(a_1+1)-\sum_{i\in I}b_i\geq 1$ and $(a_2+1)-\sum_{j\in J}b_j\geq 1$ as a consequence of the level-graph definition at the edges. Moreover, the multi-scale differential restricted to the top component is residueless, and, according to Proposition~\ref{prop:residueless}, the enumeration of such residueless differentials is given by the function
    \begin{eqnarray*}
        \Xi\left(c_{1,I}-1,c_{2,J}-1, \{-b_i\}_{i\in K}\right) & = & \sum_{c_{1,I\sqcup L}-|K\setminus L|\geq 0}(-1)^{|L|}f(c_{1,I\sqcup L}-1+|L|,|K|+1)\\
        & = & \sum_{c_{2,J\sqcup M}-|K\setminus M|\geq 0}(-1)^{|M|}f(c_{2,J\sqcup M}-1+|M|,|K|+1)\,,
    \end{eqnarray*}
    where $c_{j,I}=(a_j+1)-\sum_{i\in I}b_i$ and $f(a,n)=\frac{a!}{(a-(n-2))!}$. Note that the conditions in the above summation correspond to the walls of type $\beta_3$ and $\beta_4$ in $\mathcal{W}_p$. Moreover, this description also includes the residueless locus in the interior of $\overline{\mathcal{F}}_\lambda$, corresponding to the case $I=J=\emptyset$. By examining  the cases in Proposition~\ref{prop:flat-boundary}, we see that in each case, the contribution to the Euler characteristic of $\overline{\mathcal{F}}_\lambda$ consists of homogeneous components of  degree at most $p-1$. 
    \par 
    Finally, note that the term $\gcd (c_1, c_2)$ appears in Proposition~\ref{prop:flat-boundary}
    when $I$ and $J$ are both nonempty in the cherry graph. In the contribution to the Euler characteristic, it is multiplied by $1+\lcm (c_1, c_2)$. Since $\gcd(c_1, c_2) \lcm(c_1, c_2) = c_1c_2$, the remaining terms that contain the $\gcd$ only have degree strictly smaller than the main terms.  
\end{proof}
\par 
In what follows, we provide an example to illustrate the description in Proposition~\ref{prop:chambers}.  
\par
\begin{ex}
Consider the stratum $\mathcal{H}(x_{1}-1,x_{2}-1,-y_{1},-y_{2},-y_{3})$, where $y_{2}+y_{3}<x_{1},x_{2}<y_{1}$. The singularities of $\omega_{\lambda}$ on $\overline{\mathcal F}_\lambda$ are:
\begin{itemize}
    \item $x_{1}+x_{2}-2$ zeros of order $x_{1}+x_{2}-2$;
    \item $y_{2}-1$ poles of order $x_{2}-y_{3}+1$ (where $K=\lbrace{ 1,2 \rbrace}$, $I=\emptyset$, and $J=\lbrace{ 3 \rbrace}$);
    \item $y_{3}-1$ poles of order $x_{2}-y_{2}+1$ (where $K=\lbrace{ 1,3 \rbrace}$, $I=\emptyset$, and $J=\lbrace{ 2 \rbrace}$);
    \item $y_{2}-1$ poles of order $x_{1}-y_{3}+1$ (where $K=\lbrace{ 1,2 \rbrace}$, $I=\lbrace{ 3 \rbrace}$, and $J=\emptyset$);   
    \item $y_{3}-1$ poles of order $x_{1}-y_{2}+1$ (where $K=\lbrace{ 1,3 \rbrace}$, $I=\lbrace{ 2 \rbrace}$, and $J=\emptyset$);      
    \item $\gcd(x_{1}-y_{2},x_{2}-y_{3})$ poles of order $\lcm(x_{1}-y_{2},x_{2}-y_{3})+1$ (where $K=\lbrace{ 1 \rbrace}$, $I=\lbrace{ 2 \rbrace}$, and $J=\lbrace{ 3 \rbrace}$);
    \item $\gcd(x_{1}-y_{3},x_{2}-y_{2})$ poles of order $\lcm(x_{1}-y_{3},x_{2}-y_{2})+1$ (where $K=\lbrace{ 1 \rbrace}$, $I=\lbrace{ 3 \rbrace}$, and $J=\lbrace{ 2 \rbrace}$);
    \item $x_{1}-1$ poles of order $x_{1}-y_{2}-y_{3}+1$ (where $K=\lbrace{ 1 \rbrace}$, $I=\lbrace{2,3 \rbrace}$,  and $J=\emptyset$);   
    \item $x_{2}-1$ poles of order $x_{2}-y_{2}-y_{3}+1$ (where $K=\lbrace{ 1 \rbrace}$, $I=\emptyset$, and $J=\lbrace{2,3 \rbrace}$);      
    \item $2(y_{2}-1)(y_{3}-1)$ doubles poles (where $I=J=\emptyset$).
\end{itemize}
Then, the total sum of the orders of the singularities is:
$$
(x_{1}+x_{2})(y_{2}+y_{3})
-2y_{2}y_{3}
-2x_{1}
-2x_{2}
-2y_{2}
-2y_{3}
-\gcd(x_{1}-y_{2},x_{2}-y_{3})
-\gcd(x_{1}-y_{3},x_{2}-y_{2})
+6\,.
$$
Note that the principal term is a homogeneous polynomial of degree $2$, which equals $p-1$ for $p=3$, as predicted by Proposition~\ref{prop:chambers}. 
\par 
Moreover, for the special case $y_{2}=y_{3}=1$ and $x_1 = x_2 = a+1$, the above formula simplifies to $-2\gcd(x_{1}-1,x_{2}-1) = -2a$. By Theorem~\ref{thm:MAIN5}, 
a generic isoresidual fiber of the stratum $\mathcal{H}(a,a,-2a,-1,-1)$ consists of $a$ connected components. Since these connected components have the same topology, each of them has Euler characteristic $-2$. Therefore, a generic isoresidual fiber of the stratum $\mathcal{H}(a,a,-2a,-1,-1)$ is a disjoint union of $a$ spheres. 
\end{ex}
\par
\subsection{The top term of the Euler characteristic of $\overline{\mathcal F}_\lambda$}\label{sub:Euler}
\par 
Following Proposition~\ref{prop:chambers}, the Euler characteristic of the generic isoresidual fiber for a given stratum is a sum of homogeneous components of degree from $0$ up to $p-1$ in terms of the variables $a_{1}+1,a_{2}+1,b_{1},\dots,b_{p}$. We conjecture that the term of degree $p-1$ does not vanish in any chamber and therefore provides the asymptotics for the Euler characteristic of the generic isoresidual fiber when the orders of singularities become large (for a fixed number of poles).
\par 
\begin{conj}\label{conj:Euler}
For $\lambda \in \mathcal{R}_{p} \setminus \mathcal{A}_{p}$, the top term of the Euler characteristic of $\overline{\mathcal F}_\lambda$ is a piecewise homogeneous polynomial of degree $p-1$. 
\end{conj}
\par 
In the chambers of $\mathcal{SP}_{p}$ where the order of the smallest zero is smaller than the order of every pole, we confirm Conjecture~\ref{conj:Euler} by providing an explicit formula for the term of degree $p-1$ in the Euler characteristic.  
\par 
\begin{prop}
\label{prop:one-chamber}
For any $p \geq 2$, the top term of the Euler characteristic of the generic isoresidual fiber in the chamber of $\mathcal{SP}_{p}$ defined by the inequalities $x_{1}>p$ and $x_{1}<y_{j} +1-p$ for any $j \in \lbrace{ 1,\dots,p \rbrace}$,  is the homogeneous polynomial
$$
x_{1}^{p-1}+x_{2}^{p-1}-(x_{1}+x_{2})^{p-1}\,.
$$
\end{prop}
\par 
\begin{proof}
Using the results of Sections~\ref{subsub:zeros} and~\ref{subsub:HigherPoles}, we compute the term of degree $p-1$ in the Euler characteristic of $\overline{\mathcal F}_\lambda$ by adding the terms of degree $p-1$ in the total order of the zeros and poles of $\omega_{\lambda}$.
\par
Since $x_{1}<y_{j} -p$ for any $j \in \lbrace{ 1,\dots,p \rbrace}$, the poles of $\omega_{\lambda}$ correspond to degenerations where $I=\emptyset$ while~$J$ and $K$ form an arbitrary partition of $\lbrace{ 1,\dots,p \rbrace}$ where $K \neq \emptyset$.
\par
For each such partition $J \cup K$ where $J,K \neq \emptyset$, the total order of the poles (computed using the explicit formula for $\Xi$ given in Proposition~\ref{prop:residueless}) is $x_{2}^{|J|-1}x_{1}^{|K|-1}(x_{2}-\sum\limits_{j \in J} y_{j})$ (plus terms of lower order). Summing on the set of partitions where $|J|=j$ and $|K|=k$, we obtain
$$
\binom{p}{j}
x_{2}^{|J|-1}x_{1}^{|K|-1} \left(x_{2}
- \frac{j}{p}\sum\limits_{j=1}^{p} y_{j}
\right)
=
\binom{p}{j}
x_{2}^{|J|-1}x_{1}^{|K|-1} \left(\frac{k}{p}x_{2}
- \frac{j}{p}x_{1}
\right)\,.
$$
For the (unique) partition where $J$ is empty, the total order is $2x_{1}^{p-1}$ (plus terms of lower order). Summing these contributions on the possible values of $j$, we obtain a telescopic sum equal to $x_{1}^{p-1}+x_{2}^{p-1}$.
\par
Finally, the term of degree $p-1$ in the total order of the zeros is $(x_{1}+x_{2})^{p-1}$. Therefore, the term of degree $p-1$ in the Euler characteristic of $\overline{\mathcal F}_\lambda$ is $x_{1}^{p-1}+x_{2}^{p-1}-(x_{1}+x_{2})^{p-1}$.
\end{proof}

\section{Intersection calculations on the multi-scale compactification}
\label{sec:intersection}
\par 
In this section, we interpret and recover some of the previous results by using degeneration and intersection theory on the multi-scale compactification of strata of differentials. 
\par 
\subsection{The genus of a complete intersection curve}
\label{subsec:intersection}
\par 
First, we review how to use the adjunction formula successively to calculate the arithmetic genus $p_a$ of a complete intersection curve. Note that the arithmetic genus and the geometric genus coincide for smooth curves, although they can differ for singular curves. 
\par 
Let $C = H_1\cap \cdots \cap H_{n-1}$ be a complete intersection curve of arithmetic genus $p_a$ in an $n$-dimensional variety $X$, where each $H_i$ is a hypersurface (i.e., an effective Cartier divisor) in $X$. Then,  
$$ K_{H_1\cap\cdots \cap H_k} = \left(K_{H_1\cap \cdots \cap H_{k-1}} + H_k\right)|_{H_1\cap \cdots \cap H_{k}} $$
where $K$ stands for the canonical divisor class. 
It follows that 
$$2p_a(C)-2 = \left(K_X + \sum_{i=1}^{n-1} H_i\right) \prod_{i=1}^{n-1} H_i\,. $$
\par 
Recall that $\mu = (a_1, a_2, -b_1, \ldots, -b_p)$ is a partition of $-2$ with two positive entries and $p$ negative entries. The projectivized moduli space of differentials of type $\mu$ on $\mathbb{CP}^1$ is isomorphic to $\mathcal M_{0,p+2}$, and let $\MS(\mu)$ be the multi-scale compactification. We denote by $z_1$ and $z_2$ the two zeros, and by $q_1,\ldots,q_p$ the poles. 
\par 
We use $\Gamma$ to denote the level graph of a multi-scale differential. If $\Gamma$ has exactly two levels, then the locus of multi-scale differentials with the level graph $\Gamma$ is a boundary divisor class denoted by $D_\Gamma$. We use $\Gamma^{\top}$ and $\Gamma^{\bot}$ to denote the top and bottom levels of $\Gamma$, and we denote by $N_{\Gamma}^{\top}$ and $N_{\Gamma}^{\bot}$ the (unprojectivized) dimensions of the top and bottom level strata, respectively. We also denote by $\ell_\Gamma$ the $\lcm$ of the prong numbers of the vertical edges in $\Gamma$. The set of two-level graphs in $\MS(\mu)$ is denoted by $\LG_1(\mu)$.  
\par 
If $\Gamma$ has only one horizontal edge (i.e., with two simple poles at the branches of the corresponding node), we denote this horizontal boundary divisor class by $D_h$. 
\par 
Given a generic configuration of residues $\lambda = (\lambda_{1}, \ldots, \lambda_{p})\in \mathcal{R}_{p} \setminus \mathcal{A}_{p}$, let $D_i$ be the divisorial locus in $\MS(\mu)$ 
where $\lambda_i \Res_{q_p}\omega - \lambda_{p} \Res_{q_i}\omega = 0$ for $i=1,\ldots, p-1$. From this description, the divisor class of $D_i$ is given by 
$$ D_i = -\eta\,, $$
where $\eta$ is the first Chern class of the tautological line bundle $\mathcal O(-1)$, whose fibers are generated by the top level differentials on $\MS(\mu)$. 
\par 
We remark that $D_i$ can contain extraneous boundary divisor components that arise in the following two ways. First, $q_i$ and $q_p$ can both go to the lower levels of a level graph, resulting in  their residues becoming zero (before twisting to obtain the multi-scale differential). Alternatively, suppose one of them is on the top level. If the top level vertex containing this pole has it as the unique pole, then the residue is still automatically zero.
\par 
Recall that $\overline{\mathcal F}_{\lambda}$ is the closure of the isoresidual fiber $\mathcal F_\lambda$ in $\MS(\mu)$ parametrizing differentials with $\lambda$ as the prescribed configuration of residues at the poles (up to simultaneous scaling). Let 
$$\mathcal D_\lambda = D_1\cap \cdots \cap D_{p-2}\,.$$ 
From the above construction, it is clear that 
$\overline{\mathcal F}_\lambda\subset \mathcal D_\lambda$. In what follows, we will study the differences between them. Recall also that in Section~\ref{sub:boundary}, we described the multi-scale differentials contained in the boundary of $\overline{\mathcal F}_{\lambda}$. Next, we study multi-scale differentials that are contained in the boundary of $\mathcal D_\lambda$. 
\par 
\begin{prop}
\label{prop:boundary}
Let $\lambda$ be a generic tuple of residues. 
Then a multi-scale differential $(X,\omega)$ contained in the boundary of $\mathcal D_\lambda$ satisfies the following conditions: 
\begin{itemize}
\item If the level graph of $(X,\omega)$ has only one level with a horizontal edge, then the horizontal edge separates two components $X_1$ and $X_2$, where $X_1$ contains $z_1$ and the subset of poles labeled in $I$, and $X_2$ contains $z_2$ and the subset of poles labeled in $J$ (with $I\sqcup J = \{1,\ldots,p\}$), and $a_1 - \sum\limits_{i\in I} b_i + 1 = 0$. Moreover, the residues of the marked poles are determined by $\lambda$. 
\item If the level graph has multiple levels and if the lower levels contain a marked pole, then the residue of every top-level marked pole is zero.
\end{itemize}
\end{prop}
\par 
\begin{proof}
In the first case, since there are only two zeros and each vertex (of genus zero) must contain a zero, it follows that $X$ has two components joined by the horizontal edge. Since there is no lower level, the residues do not degenerate to zero; hence, they are still determined by $\lambda$. 
\par 
In the second case, suppose a marked pole $q_i$ is on the top level for $i\leq p-2$. If $q_p$ is in a lower level, then since the residue at $q_p$ is zero (before twisting to obtain the multi-scale differential), it follows that the residue at $q_i$ is zero. If $q_p$ is also on top, then using some $q_j$ in a lower level implies that the residue at $q_p$ is zero, and hence the residue of each top-level pole is zero. Finally,  if $q_{p-1}$ is the only top-level pole, then by the residue theorem, its residue is zero.  
\end{proof}
\par 
Note that a simple pole cannot have zero residue. Therefore, the case when all the poles are simple is easier to treat  using Proposition~\ref{prop:boundary}. For the case of higher-order poles, we need to subtract extraneous boundary components from each $D_i$ before taking the complete intersection. In the following two sections, we will deal with the two cases separately.   
\par 
\subsection{Strata with only simple poles}
\label{sec:simple}
\par 
In this section, we will prove Theorem~\ref{thm:simple} using intersection theory on $\MS(\mu)$. 
\par 
\begin{cor}
\label{cor:boundary-simple}
Let $\lambda$ be a generic tuple of residues. If all poles are simple, i.e., $b_i = 1$ for all $i$, then $\mathcal D_\lambda $ does not contain any multi-scale differential whose level graph has more than one level with a pole below the top level. In this case, $\mathcal D_\lambda = \overline{\mathcal F}_\lambda$ and 
\begin{eqnarray} 
\label{eq:simple}
2g (\overline{\mathcal F}_\lambda) - 2 & = & \left( 2\eta - D_h + \sum_{\Gamma\in \LG_1(\mu)} (N_{\Gamma}^{\bot} \ell_\Gamma - 1) D_\Gamma  \right) (-\eta)^{p-2}\,.
\end{eqnarray}
\end{cor}
\par 
\begin{proof}
By Proposition~\ref{prop:boundary}, for a multi-scale differential contained in $\mathcal D_\lambda$ with multiple levels, if there is a pole below the top level, then every top-level pole has zero residue. Note that every vertex (of genus zero) on top level must contain at least one pole. Additionally, a simple pole on the top level cannot have zero residue. This thus justifies the first claim. We remark that $D_\lambda$ can contain multi-scale differentials with a two-level, two-vertex graph, where all simple poles are on the top-level vertex with prescribed generic residues. Note that the locus of such multi-scale differentials is $0$-dimensional and thus does not produce any extraneous boundary component in $\mathcal D_\lambda$ but not in $\overline{\mathcal F}_\lambda$. 

Since $\overline{\mathcal F}_\lambda$ is smooth by Theorem~\ref{thm:closureMSD}, we have 
\begin{eqnarray}
2g(\overline{\mathcal F}_\lambda) - 2 & = & 2p_a (\mathcal D_\lambda) - 2 \\ \nonumber
& = & \left(K_{\MS(\mu)} + \sum_{i=1}^{p-2}D_i\right) \prod_{i=1}^{p-2} D_i \\ \nonumber 
& = & \left( 2\eta - D_h + \sum_{\Gamma\in \LG_1(\mu)} (N_{\Gamma}^{\bot} \ell_\Gamma - 1) D_\Gamma  \right)(-\eta)^{p-2}\,,
\end{eqnarray}
where we used the canonical divisor class formula of $\MS(\mu)$ from \cite[Theorem 1.1]{CMZ}. 
\end{proof}
\par 
Consider $\mu = (a_1, a_2, \rec[-1][a_1+a_2+2])$ where $a_1, a_2 > 0$ and all the poles are simple. By Corollary~\ref{cor:boundary-simple}, in this case, the generic isoresidual fiber $\overline{\mathcal F}_\lambda = \mathcal D_\lambda = D_1\cap \cdots \cap D_{p-2}$; hence, we can apply Formula~\eqref{eq:simple} to compute the genus of $\overline{\mathcal F}_\lambda$.
\par 
\begin{proof}[Proof of Theorem~\ref{thm:simple}]
In Formula~\eqref{eq:simple}, we need to evaluate 
$$D_h (-\eta)^{a_1+a_2}, \quad (-\eta)^{a_1+a_2+1},\quad D_\Gamma (-\eta)^{a_1+a_2}$$ 
where $\Gamma$ is a two-level graph with $N_\Gamma^{\bot} = 1$ (otherwise, the intersection of $D_\Gamma$ with $\eta^{a_1+a_2}$ is zero since the top level does not have enough dimension). For simplicity of notation, we denote 
$$ N(a_1, a_2) = \int_{\MS(\mu)} (-\eta)^{a_1+a_2+1}\,. $$
We also denote and recall 
$$ N(a) = \int_{\MS(a, \rec[-1][a+2])} (-\eta)^{a} = a! \, $$
from \cite{ChPrIso}. 
\par 
First, a general differential in $D_h$ is of type $(a_1,\rec[-1][a_1+1]; -1)\times (a_2, \rec[-1][a_21+1]; -1)$ with matching residues at the horizontal node. Let $\pi_1$ and $\pi_2$ be the projections from the product of the strata of the two horizontal components to each of them, respectively.   Then $D_h$ can be identified with the zero locus of a section of $-\eta$ over $\PP(\pi_1^{*}\mO(-1) \oplus \pi_2^{*}\mO(-1))$, using the matching residue condition at the horizontal node.  We also need to choose $a_1+1$ simple poles to be in the component that contains $z_1$. It follows that
\begin{eqnarray*} 
\int_{\MS(\mu)}D_h (-\eta)^{a_1+a_2} & = & \binom{a_1 + a_2 + 2 }{a_1 + 1} a_1 ! a_2 ! \\ 
& = & \frac{(a_1+a_2+2)!}{(a_1+1)(a_2+1)}\,. 
\end{eqnarray*}
\par 
Next, we recall the divisor class relation 
$$ \eta = (-1+1) \psi_{q_1} -\sum_{q_1\in \Gamma^\bot}  \ell_\Gamma D_\Gamma $$
from \cite[Proposition 8.2]{CMZ}. 
Then we have 
\begin{eqnarray*}
N (a_1, a_2) & = & \left( \sum_{q_1\in \Gamma^\bot,\ N_\Gamma^{\bot} = 1} \ell_\Gamma D_\Gamma \right) (-\eta)^{a_1+a_2} \\
& = & a_1 N(a_1-1, a_2) + a_2 N(a_1, a_2-1)\,.  
\end{eqnarray*}
In the above, the two terms come from the two cases of the bottom vertex of $\Gamma$ containing $q_1$ and $z_1$, or containing $q_1$ and $z_2$, respectively. Note that $N(0,0) = 0$. By induction, we conclude that $N(a_1, a_2) = 0$ for all $a_1$ and $a_2$. 
\par 
Finally, we compute $\sum\limits_{N_\Gamma^{\bot} = 1} \ell_\Gamma D_\Gamma (-\eta)^{a_1+a_2}$, where the sum runs over the two-level graphs with (unprojectivized) bottom dimension $N_\Gamma^{\bot} = 1$. Since every top vertex must carry a marked pole, the global residue condition (GRC) is void. There are two cases, depending on whether a simple pole (say $q_1$) is in the lower level or not. First, suppose $q_1\in \Gamma^\bot$.  Then the top vertex contains one marked zero and one zero edge. In this case, we obtain $D_\Gamma (-\eta)^{a_1+a_2} = 0$ as before. Next, suppose all simple poles are on the top level. Then there are two sub-cases. If both $z_1$ and $z_2$ are on the bottom, then we obtain 
$D_\Gamma (-\eta)^{a_1+a_2} = N(a_1+a_2)$ with $\ell_\Gamma = a_1+a_2+1$. If $z_1$ is on the bottom and there are two top vertices---one containing $k$ many simple poles and the other containing $z_2$ with the remaining simple poles---then the corresponding contribution is  
$$ \sum_{k=2}^{a_1} \binom{a_1+a_2+2}{k} N(k-2) N(a_1-k, a_2) = 0\,. $$
Similarly, if $z_2$ is on the bottom and there are two top vertices, the contribution is 
$$ \sum_{k=2}^{a_2} \binom{a_1+a_2+2}{k} N(k-2) N(a_1, a_2-k) = 0\,. $$
\par 
In summary, we conclude that 
\begin{eqnarray*}
2g (\overline{\mathcal F}_\lambda) -2 & = & - D_h (-\eta)^{a_1+a_2} - 2 (-\eta)^{a_1+a_2+1} + 
\sum_{N_\Gamma^{\bot} = 1} (N_\Gamma^{\bot}\ell_\Gamma - 1) D_\Gamma (-\eta)^{a_1+a_2} \\
&= & - \frac{(a_1+a_2+2)!}{(a_1+1)(a_2+1)} - 2\cdot 0 + (a_1+a_2+1-1) \cdot (a_1+a_2)! \\
& = & (a_1+a_2)! \left(  a_1+a_2  - \frac{(a_1+a_2+2)(a_1+a_2+1)}{(a_1+1)(a_2+1)} \right)\,.  
\end{eqnarray*}
This completes the proof of Theorem~\ref{thm:simple}. 
\end{proof}
\par 

\subsection{Strata with arbitrary poles}
\par
Now, we consider a partition $\mu = (a_1, a_2, -b_1, \ldots, -b_p)$ of $-2$, where the $b_i$'s can be arbitrary positive integers. 
\par 
We denote $T_0= \MS(\mu)$. For $i=1,\ldots,p-2$, we define the spaces $T_i$ successively, each as the closure inside $T_{i-1}$ of the locus of differentials $(\mathbb{C}\mathbb{P}^1,\omega)$ satisfying the {\em residue equation} 
$$\lambda_{p} \Res_{q_i}\omega-\lambda_{i} \Res_{q_p}\omega=0$$ 
for $q_i$. 
\par 
We denote and recall 
$$ f(a,n) = \frac{a!}{(a-n+2)!}\,.$$ 
Moreover, we denote and recall 
$$ c_{1,I} = a_1 + 1 - \sum_{i\in I}b_i \quad {\rm and} \quad c_{2,J} = a_2 + 1 - \sum_{j\in J}b_j$$
for subsets $I$ and $J$ of $\{1,\ldots, p\}$, as used in Proposition~\ref{prop:flat-boundary}. We also recall the quantity $\Xi(a_1,a_2;b_1,\ldots, b_n)$, which is described in Proposition~\ref{prop:residueless} and counts the number of residueless differentials of type $\mu$. 
\par 
\begin{prop}
    \label{prop:T=F}
    For any generic configuration of residues $\lambda\in \mathcal{R}_p\setminus \mathcal{A}_p$, we have $T_{p-2}=\overline{\mathcal F}_\lambda$. Moreover, the genus $g_{\lambda}$ of $\overline{\mathcal F}_\lambda$ is determined by
    \begin{eqnarray*}
        2g_\lambda-2 & = & (a_1+a_2)f(a_1+a_2, p) \\
      & & - \sum_{c_{1,I}=0} f(a_1, |I|+1)f(a_2,|I^c|+1)\\
      & & - \sum_{c_{1,I}>0} (c_{1,I}+1)f(a_1,|I|+1)\Xi(c_{1,I}-1,a_{2};\{b_i\}_{i\in I})\\
      & & - \sum_{c_{2,J}>0} (c_{2,J}+1)f(a_2,|J|+1)\Xi(a_{1}, c_{2,J}-1;\{b_j\}_{j\in J})\\
      & & -\sum_{\substack{c_{1,I}>0, \\ c_{2,J}>0}}(1+\lcm(c_{1,I},c_{2,J}))f(a_1,|I|+1)f(a_2,|J|+1)\gcd(c_{1,I},c_{2,J})\Xi(c_{1,I}-1,c_{2,J}-1;\{b_i\}_{i\in K})\\
      & &-2\Xi(a_1,a_2;b_1,\ldots, b_n)\,.
    \end{eqnarray*}
\end{prop}
\par 
We note that the terms in the above formula correspond precisely to the description of the zeros, poles, and their orders for the translation structure $\omega_\lambda$ on $\overline{\mathcal F}_\lambda$, as given in Propositions~\ref{prop:zeros}, \ref{prop:simple-pole}, and~\ref{prop:flat-boundary}. Therefore, this provides an alternative cross-check for the computation of the Euler characteristic of $\overline{\mathcal F}_\lambda$. 
\par
\begin{proof}
    Every element $(X,\omega)\in T_i$ satisfies the residue equations $\lambda_{p} \Res_{q_j}\omega-\lambda_{j} \Res_{q_p}\omega=0$ for every $1\leq j \leq i$. In particular, every element in the interior of $T_{p-2}$ satisfies the residue equations for the first $p-2$ poles; hence, by the residue theorem, all the residues are determined by $\lambda$ up to scaling. Therefore, the interior of $T_{p-2}$ coincides with the open isoresidual fiber curve $\mathcal{F}_\lambda$. Since at each step we remove the extraneous boundary components, the closure $T_{p-2}$ of its interior inside the closed space $T_{p-3}$ is exactly the closure of $\mathcal{F}_\lambda$ inside $\MS(\mu)$. It follows that $T_{p-2}=\overline{\mathcal F}_\lambda$. 
    \par 
    To compute the genus of $T_{p-2}$, we observe that the divisor class of $T_i$ inside $T_{i-1}$ satisfies 
    \begin{equation*}
        [T_i]=-\eta-\sum_{\Gamma\in \mathcal{G}_i}\ell_{\Gamma}D_{\Gamma}
    \end{equation*}
    where $\mathcal{G}_i$ is the collection of two-level graphs $\Gamma$ that automatically guarantee the residue equation for the pole $q_i$ by using the residue equations of the previous poles $q_1,\ldots, q_{i-1}$ as well as the associated global residue condition for $\Gamma$. 
    \par 
    Let $K_{T_{i}}$ be the canonical divisor class of  $T_i$. By the adjunction formula, we obtain that
    \begin{equation*}
        K_{T_i} = \left(K_{T_{i-1}} + [T_i]\right)[T_i]\,.
    \end{equation*}
    This implies that 
    \begin{eqnarray*}
        2g (T_{p-2})-2 & = & \left(K_{\MS(\mu)} + \sum_{i=1}^{p-2}\left( -\eta-\sum_{\Gamma\in \mathcal{G}_i}\ell_{\Gamma}D_{\Gamma} \right)\right)[T_{p-2}] \\
        & = & \left( p\eta -D_h + \sum_{\Gamma \in \LG_1}(N_{\Gamma}^\bot\ell_{\Gamma}-1)D_{\Gamma} - (p-2)\eta - \sum_{i=1}^{p-2}\sum_{\Gamma \in \mathcal{G}_i}\ell_{\Gamma} D_{\Gamma} \right)\left[\overline{\mathcal F}_\lambda\right]\,.
    \end{eqnarray*}
    Therefore, we only need to consider those two-level graphs $\Gamma$ such that $D_{\Gamma}$ intersects $\overline{\mathcal F}_\lambda$ as described in Propositions~\ref{prop:F-boundary} and~\ref{prop:F-horizontal}. 
    \par 
    We denote by $\Gamma_{z_1,z_2}$ the two-level graph with both zeros in the unique bottom component and all the poles in the unique top component. We denote by $\Gamma_{I,J}$ the two-level ``cherry'' graphs with the zero $z_1$ in one bottom component, along with some marked poles indexed by $I$, the zero $z_2$ in the other bottom component with some marked poles indexed by $J$, and the rest of the marked poles in the unique top component indexed by $K$. Note that either $I$ or $J$ can be empty, but not both. Moreover, neither $\Gamma_{z_1,z_2}$ nor the horizontal graphs corresponding to the components of  $D_h$ belong to any $\mathcal{G}_i$. 
    \par 
    To describe the set $\left\{\Gamma_{I,J} \in \sqcup_{i=1}^{p-2}\mathcal{G}_i\right\}$, depending on the positions of $q_{p-1}$ and $q_p$, we analyze the following four cases: 
    \par 
    \begin{itemize}
        \item[\textit{Case 1:}] Suppose $q_{p-1}, q_p\in \Gamma_{I,J}^{\top}$. First, we cannot determine the residue equation for any of the poles in the top component using only the residue equations for the previous ones, due to the presence of~$q_{p-1}$. Next, let $q_{i_0}$ be the pole with the smallest index at the bottom level. Imposing its residue equation forces the residue at $q_p$ to be zero because the differential $\omega$ (before twisting) vanishes at the bottom level. This guarantees the residue equations for all the other poles at the bottom level, other than~$q_{i_0}$. Consequently, $\Gamma_{I,J}\in \mathcal{G}_i$ for each $i$ such that $i_0<i\leq p-2$ and $q_i\in \Gamma_{I,J}^\bot$.
        \par 
        \item[\textit{Case 2:}] Suppose $q_{p-1} \in \Gamma_{I,J}^{\top}$ and $q_p \in \Gamma_{I,J}^{\bot}$. Similarly, we cannot determine the residue equation for any of the poles in the top component using only the residue equations for the previous ones. However, all the residue equations for the other poles at the bottom level are guaranteed, due to the vanishing of $\omega$ at the bottom level. This means that $\Gamma_{I,J}\in \mathcal{G}_i$ for each $i\leq p-2$ such that $q_i\in \Gamma_{I,J}^\bot$.
        \par 
        \item[\textit{Case 3:}] Suppose $q_{p-1} \in \Gamma_{I,J}^{\bot}$ and $q_p \in \Gamma_{I,J}^{\top}$. First, consider the case where the other poles in the top component are the first $i$ poles, $q_1,\ldots, q_i$. We cannot determine the residue equation for any of these poles using the residue equations of the previous ones. However, if we impose the residue equations for all of them, the residue theorem for the top component and the generality of $\lambda$ imply that all the poles in the top component are residueless, including in particular $q_p$. This guarantees the residue equations for the remaining poles at the bottom level, except for $q_{p-1}$. In this case, $\Gamma_{I,J}\in \mathcal{G}_i$ for each $i\leq p-2$ such that $q_i\in \Gamma_{I,J}^\bot$.
        \par 
        Next, consider the case where $q_{i_0}$ is the pole with the smallest index at the bottom level and  there exists at least one pole $q_i$ at the top level such that $i_0<i<p-1$. Imposing the residue equation for $q_{i_0}$ forces the residue at $q_p$ to be zero. This ensures that the residue equations for the remaining poles at the bottom level are satisfied, except for $q_{p-1}$. Additionally, let $q_{i_1}$ be the pole with the largest index in the top component. By imposing the residue equations for the poles with indices smaller than $i_1$, and since the residue at $q_p$ is zero, all the poles in the top component become residueless, and the residue equation for $q_{i_1}$ is guaranteed. In this case, we conclude that $\Gamma_{I,J}\in \mathcal{G}_i$ for $i=i_1$, and 
       for all other $i\leq p-2$ such that $i\neq i_0$ and $q_i\in \Gamma_{I,J}^\bot$. 
        \par 
        \item[\textit{Case 4:}] Suppose $q_{p-1},q_p \in \Gamma_{I,J}^{\bot}$. The residue equations for all the poles at the bottom level, except for~$q_{p-1}$ and $q_p$, are guaranteed. For the top level, only the residue equation for the pole with the largest index $i_1$ is guaranteed by imposing the residue equations for the previous ones. In this case, $\Gamma_{I,J}\in \mathcal{G}_i$ for $i=i_1$, and for all other $i\leq p-2$ such that $q_i\in \Gamma_{I,J}^\bot$.
    \end{itemize}
    \par 
    Note that for any of these four cases, the graph $\Gamma_{I,J}$ appears $|I|+|J|-1$ times in $\sqcup_{i=1}^{p-2}\mathcal{G}_i$, with $N_{\Gamma_{I,J}}^\bot=|I|+|J|$. Therefore, 
    \begin{eqnarray*}
        2g (\overline{F}_\lambda)-2 & = & \left( 2\eta -D_h + (\ell_{\Gamma_{z_1,z_2}}-1)D_{\Gamma_{z_1,z_2}} +\sum_{\Gamma_{I,J}}(\ell_{\Gamma_{I,J}}-1)D_{\Gamma_{I,J}} \right)\left[\overline{\mathcal F}_\lambda\right]\\
        & = & (\ell_{\Gamma_{z_1,z_2}}-1)D_{\Gamma_{z_1,z_2}}\left[\overline{\mathcal F}_\lambda\right] - D_h\left[\overline{\mathcal F}_\lambda\right]-2\left(-\eta-\sum_{\Gamma_{I,J}}\ell_{\Gamma_{I,J}}D_{\Gamma_{I,J}}\right)\left[\overline{\mathcal F}_\lambda\right]\\
        & & -\sum_{\Gamma_{I,J}}(\ell_{\Gamma_{I,J}}+1)D_{\Gamma_{I,J}}\left[\overline{\mathcal F}_\lambda\right]\,,
    \end{eqnarray*}
where
    \begin{eqnarray*}
(\ell_{\Gamma_{z_1,z_2}}-1)D_{\Gamma_{z_1,z_2}}\left[\overline{\mathcal F}_\lambda\right] & = & (a_1+a_2)f(a_1+a_2,p)\,,\\ 
& & \\
 D_h \left[\overline{\mathcal F}_\lambda\right] & = & \sum_{c_{1,I}=0} f(a_1,|I|+1)f(a_2,|I^c|+1)\,,\\
-2\left(-\eta-\sum_{\Gamma_{I,J}} \ell_{\Gamma_{I,J}} D_{\Gamma_{I,J}}\right)\left[\overline{\mathcal F}_\lambda\right] & = & -2 \Xi(a_1,a_2; b_1,\ldots, b_p) \ \ \text{by Proposition \ref{prop:zero-res-taut} below,} 
\end{eqnarray*}
and
\begin{equation*}
D_{\Gamma_{I,J}}\left[\overline{\mathcal F}_\lambda\right] = 
\left\{
    \begin{array}{lr}
        f(a_1,|I|+1)\Xi(c_{1,I}-1,a_2;\{b_i\}_{i\in K}) & \text{if } J=\emptyset,\\
        f(a_2,|J|+1)\Xi(a_1,c_{2,J}-1;\{b_i\}_{i\in K}) & \text{if } I=\emptyset,\\
        f(a_1,|I|+1)f(a_1,|J|+1)\text{gcd}(c_{1,I},c_{2,J})\Xi(c_{1,I}-1,c_{2,J}-1;\{b_i\}_{i\in K)}) & \text{otherwise}.
    \end{array}\right.
\end{equation*}
\end{proof}
It remains to prove the following identity: 
\par 
\begin{prop}
\label{prop:zero-res-taut}
The number of residueless differentials $\Xi(a_1,a_2;b_1,\ldots,b_p)$ in stratum $\mathcal{H}(a_1, a_2, -b_1,\ldots, -b_p)$ is
\begin{equation}
    \label{eqn:zero-res-taut}
    \Xi(a_1,a_2;b_1,\ldots,b_p)=\left(-\eta-\sum_{\Gamma_{I,J}}\ell_{\Gamma_{I,J}} D_{\Gamma_{I,J}}\right)\left[\overline{\mathcal{F}}_{\lambda}\right]\,.
\end{equation}
\end{prop}
\par 
\begin{proof}
    We rewrite the right-hand side of Equation \eqref{eqn:zero-res-taut} with respect to the pole $q_p$ as
    \begin{equation*}
        \left(-\eta-\sum_{\Gamma_{I,J}}\ell_{\Gamma_{I,J}} D_{\Gamma_{I,J}}\right)\left[\overline{\mathcal{F}}_{\lambda}\right]=\left((b_p-1)\psi_{q_p}-\sum_{q_p\in \Gamma_{I,J}^{\top}}\ell_{\Gamma_{I,J}}D_{\Gamma_{I,J}}\right)\left[\overline{\mathcal{F}}_{\lambda}\right]\,.
    \end{equation*}
    \par 
Consider the forgetful morphism $\pi\colon \MS(\mu)\rightarrow \overline{\mathcal M}_{0,p+2}$, which remembers only the $p+2$ marked zeros and poles. We have 
$$\pi_{*}(\psi_{q_p})=\sum_{\Delta \in B^{q_p}_{z_1,z_2}}\Delta\,,$$ 
where $B^{q_p}_{z_1,z_2}$ is the set of boundary divisors in $\overline{\mathcal M}_{0,p+2}$ whose dual graphs have the pole $q_p$ in one component and the zeros $z_1$ and $z_2$ in the other component. 
\par 
Let $\pi^*(B^{q_p}_{z_1,z_2}) =\left\{D_{\Gamma}\ | \ \pi_{*}(D_{\Gamma})\in B^{q_p}_{z_1,z_2}\right\}$. Then, the only boundary divisor $D_{\Gamma}\in \pi^*(B^{q_p}_{z_1,z_2})$ that intersects~$\overline{\mathcal{F}}_{\lambda}$ is the one for which $z_1$ and $z_2$ are in the bottom component of $\Gamma$ and all the poles are in the unique top component. We conclude that
\begin{equation}
 \begin{aligned}
    \Bigg((b_p-1)\psi_{q_p}&-\sum_{q_p\in \Gamma^{\top}}\ell_{\Gamma}D_{\Gamma}\Bigg)\left[\overline{\mathcal{F}}_{\lambda}\right]\\
    =& \ (b_p-1)f(a_1+a_2,p) \\
    &-\sum_{\substack{c_{1,I'}>0 \\ c_{2,J'}>0 \\ |I'|+|J'|\geq 1 \\ p\in K'}}c_{1,I'}c_{2,J'}f(a_1,|I'|+1)f(a_2,|J'|+1)\Xi(c_{1,I'}-1,c_{2,J'}-1;\{b_i\}_{i\in K'})\,.
\end{aligned}
\end{equation}
\par 
Recall that the number of residueless differentials is given by Proposition \ref{prop:residueless}, where Equation \eqref{eqn:residueless} can be rewritten in two different ways with respect to the zeros $z_1$ and $z_2$, as
\begin{eqnarray*}
    \Xi(a_1,a_2;b_1,\ldots,b_p) & = & \sum_{c_{1,I}+|I|-p\geq 0}(-1)^{|I|}f(c_{1,I}-1+|I|,p+1)\\
    & = & \sum_{c_{2,J}+|J|-p\geq 0}(-1)^{|J|}f(c_{2,J}-1+|J|,p+1)\,.
\end{eqnarray*}
If we set $b_p=1$, every term $(-1)^{|I|}f(c_{1,I}-1+|I|,p+1)$ such that $p\in I$ cancels out with the term $(-1)^{|I\setminus\{p\}|}f(c_{1,I\setminus\{p\}}-1+|I\setminus\{p\}|,p+1)$, and the same holds for $z_2$ and $J$. Therefore,  
\begin{eqnarray*}
    \Xi(a_1,a_2;b_1,\ldots,b_p)|_{b_p=1} & = & \sum_{\substack{b_p-1>c_{1,I}+|I|-p\geq 0\\ p\notin I}}(-1)^{|I|}f(c_{1,I}-1+|I|,p+1)\\
    & = & \sum_{\substack{b_p-1>c_{2,J}+|I|-p\geq 0\\ p\notin J}}(-1)^{|J|}f(c_{2,J}-1+|J|,p+1)\,.   
\end{eqnarray*}
\par 
Note that 
$$(-1)^{|I|}f(c_{1,I}-1+|I|,p+1)|_{b_p=1}=(-1)^{|J|}f(c_{2,J}-1+|J|,p+1)|_{b_p=1}$$ 
for $I\sqcup J=\{1,\ldots,p-1\}$. 
Moreover, if $b_p-1>c_{1,I}+|I|-p\geq 0$ and $b_p-1>c_{2,J}+|J|-p\geq 0$ for $I\sqcup J=\{1,\ldots,p-1\}$, then
\begin{eqnarray*}
   c_{1,I'}, c_{2,J'} & > & 0\\
   b_p-1>c_{1,I}+|I|-(p-|J'|) & \geq & 0\\ 
   b_p-1>c_{2,J}+|J|-(p-|I'|) & \geq & 0
\end{eqnarray*}
 for every $I'\subset I$ and $J'\subset J$. This implies that 
$$
\begin{aligned}
    \Biggl( &(b_p-1)\psi_{q_p}- \sum_{q_p\in \Gamma^{\top}}\ell_{\Gamma_{I,J}}D_{\Gamma_{I,J}}\Biggr)\left[\overline{\mathcal{F}}_{\lambda}\right]\Big|_{b_p=1}\\
    &= -\sum_{\substack{c_{1,I'}>0 \\ c_{2,J'}>0 \\ |I'|+|J'|\geq 1 \\ p\in K'}}c_{1,I'}c_{2,J'}f(a_1,|I'|+1)f(a_2,|J'|+1)\Xi(c_{1,I'}-1,c_{2,J'}-1,\{b_i\}_{i\in K'})|_{b_p=1}\\
    &=\sum_{\substack{b_p-1>c_{1,I}+|I|-|I'|-|K'|\geq 0\\ I'\subset I , \ J' \subset \{1,\ldots,p-1\}-I \\ |I'|+|J'|\geq 1 \\ p\in K'}}(-1)^{|I|-|I'|-1}c_{1,I'}c_{2,J'}f(a_1,|I'|+1)f(a_2,|J'|+1)f(c_{1,I}-1+|I|-|I'|,|K'|+1)\\
    &=\sum_{\substack{b_p-1>c_{1,I}+|I|-p\geq 0\\ p \notin I}}(-1)^{|I|-1}\sum_{i=0}^{|I|}\sum_{\substack{j=0\\i+j\geq 1}}^{p-1-|I|}\Bigg[(-1)^i f(a_1,i+1)f(a_2,j+1)f(c_{1,I}-1+|I|-i,p+1-i-j)\\
    & \ \ \ \ \ \ \ \ \ \ \ \ \ \ \ \ \ \ \ \ \ \ \ \ \ \ \ \ \ \ \ \ \ \ \ \ \ \ \ \ \ \ \ \ \ \ \ \ \ \ \ \ \ \ \ \ \ \ \ \ \cdot \left( \binom{|I|-1}{i-1}c_{1,I}+\binom{|I-1|}{i}(a_1+1) \right)\\
    & \ \ \ \ \ \ \ \ \ \ \ \ \ \ \ \ \ \ \ \ \ \ \ \ \ \ \ \ \ \ \ \ \ \ \ \ \ \ \ \ \ \ \ \ \ \ \ \ \ \ \ \ \ \ \ \ \ \ \ \  \cdot\left( \binom{p-2-|I|}{j-1}(1-c_{1,I})+\binom{p-2-|I|}{j}(a_2+1) \right)\Bigg]\\
    &=\sum_{\substack{b_p-1>c_{1,I}+|I|-p\geq 0\\ p\notin I}}(-1)^{|I|}f(c_{1,I}-1+|I|,p+1)\\
    &\ \  \ \ \ \ \ \ \ \ \ \ \ \ \ \ \ \ \ \ \ \ \ \ \ \  +(-1)^{|I|-1} \sum_{i=0}^{|I|}\Bigg[(-1)^i f(a_1,i+1)\left( \binom{|I|-1}{i-1}c_{1,I}+\binom{|I|-1}{i}(a_1+1) \right)\\
    &\ \ \ \ \ \ \ \ \ \ \ \ \ \ \ \ \ \ \ \ \ \ \ \ \ \ \ \ \ \ \ \ \ \ \ \ \ \ \ \ \ \ \ \ \ \ \ \ \ \cdot\Bigg(-\sum_{j=1}^{p-1-|I|}f(a_2,j+1)f(c_{1,I}-1+|I|-i,p+2-i-j)\\
    &\ \ \ \ \ \ \ \ \ \ \ \ \ \ \ \ \ \ \ \ \ \ \ \ \ \ \ \ \ \ \ \ \ \ \ \ \ \ \ \ \ \ \ \ \ \ \ \ \ \ \ \ \  +\sum_{j=0}^{p-2-|I|}f(a_2,j+2)f(c_{1,I}-1+|I|-i,p+1-i-j)\Bigg)\Bigg]\\
    &=\sum_{\substack{b_p-1>c_{1,I}+|I|-p)\geq 0\\ p\notin I}}(-1)^{|I|}f(c_{1,I}+|I|,p+1)\,.
\end{aligned}$$
\par 
Finally, since $a_2=\sum_{i=1}^p b_i -a_1-2$, the left-hand side and right-hand side of Equation \eqref{eqn:zero-res-taut} are polynomials $L,R\in \mathbb{Z}[a_1,b_1,\ldots,b_p]$ of degree $p-1$. The computations above imply that $(b_p-1) \mid  (L-R)$. Furthermore, since we can use any other pole instead of $q_p$, it follows that  $\prod_{i=1}^p(b_i-1)\mid (L-R)$. Since $L-R$ has degree at most $p-1$, we conclude that $L-R$ is the zero polynomial.
\end{proof}

\section{Monodromy of the isoresidual fibration}\label{sec:GaussManin}

\subsection{Gauss--Manin connection}
\par 
For a stratum $\mathcal{H}(a_{1},a_{2},-b_{1},\dots,-b_{p})$, we introduce the vector bundle $\pi\colon\mathcal{E} \rightarrow \mathcal{B}$, where:
\begin{itemize}
    \item $\mathcal{B}$ is the complement $\mathcal{R}_{p} \setminus \mathcal{A}_{p}$ of the resonance arrangement in the residual space;
    \item for each $\lambda \in \mathcal{B}$, the fiber $\mathcal{E}_{\lambda}$ is $\mathbb{C} \otimes H_{1}(\overline{\mathcal{F}}_{\lambda} \setminus P_{\omega_{\lambda}}, Z_{\omega_{\lambda}},\ZZ)$.
\end{itemize}
The vector bundle $\pi\colon\mathcal{E} \rightarrow \mathcal{B}$ is endowed with a connection $\nabla_{GM}$ that is locally trivial on the elements of  $\calH_{\lambda} := H_{1}(\overline{\mathcal{F}}_{\lambda} \setminus P_{\omega_{\lambda}}, Z_{\omega_{\lambda}},\ZZ)$ (which is a lattice in the fiber). This connection is called the \textit{Gauss--Manin connection} of $\mathcal{E} $.
\par 
\subsection{Period central charge}
\par 
\begin{defn}\label{defn:DualLattice}
In the dual space $(\mathcal{R}_{p})^{\ast}$ of the residual space $\mathcal{R}_{p}$, we denote by $L^{\vee}$ the lattice generated by the partial sums $\lambda \mapsto \sum\limits_{j \in J} \lambda_j$ for subsets $J\subset \{1,\ldots,p\}$. 
\end{defn}
\par 
It follows from Proposition~\ref{prop:IntegerWeights} that for each relative homology class $[\gamma]$ of $\calH_{\lambda}$, its period depends on $\lambda$ as an element of $L^{\vee}$. Since the saddle connections generate the relative homology group of any translation surface, the period of any relative homology class $[\gamma]$ in $\calH_{\lambda}$ is given by an element $w_{[\gamma]}\in L^{\vee}$. Moreover, the map $[\gamma]\mapsto w_{[\gamma]}$ is locally constant under small variations of $\lambda$. We describe this structure as a morphism from $\mathcal{E}$ to the trivial bundle $(\mathcal{R}_{p})^{\ast} \times \mathcal{B}$.
\par 
\begin{defn}
The \textit{period central charge} is the bundle morphism $\Psi\colon\mathcal{E} \rightarrow \mathcal{B}\times (\mathcal{R}_{p})^{\ast} $ such that for any configuration $\lambda \in \mathcal{B}$ and any relative homology class $[\gamma] \in \calH_{\lambda}$, we have $\Psi(\lambda,[\gamma])=(\lambda,w)$, where $w$ is the element of $L^{\vee}$ prescribing the deformation of $\int_{[\gamma]} \omega_{\lambda}$ as $\lambda$ changes. The morphism $\Psi$ extends linearly to arbitrary elements of $\mathcal{E}_{\lambda} \coloneqq \mathbb{C} \otimes H_{1}(\overline{\mathcal{F}}_{\lambda} \setminus P_{\omega_{\lambda}}, Z_{\omega_{\lambda}},\ZZ)$.
\end{defn}
\par 
Isoresidual fibers are translation surfaces, and deformations of translation surfaces correspond to deformations of the periods of their relative homology classes. The period central charge describes how the periods of the relative homology classes of $\calH_{\lambda}$ change in terms of $\lambda$. It follows that the geometry of the isoresidual fibration is entirely described by the Gauss--Manin connection of $\calE$, obtained by replacing the fiber with its homology group (up to a change of coefficients).
\par 
\subsection{Monodromy}\label{sub:Monodromy}
\par 
For any based loop $\alpha$ in $\mathcal{B}$, the Gauss--Manin connection along $\alpha$ induces an automorphism 
$T_{\alpha}$ of $\mathcal{E}_{\lambda} =\mathbb{C} \otimes 
 H_{1}(\overline{\mathcal{F}}_{\lambda} \setminus P_{\omega_{\lambda}}, Z_{\omega_{\lambda}})$. A priori, computing the monodromy of the Gauss--Manin connection on $\mathcal{E}_\lambda$ is a difficult problem. Nevertheless, the isoresidual fibration enjoys additional structures that impose significant constraints on the possible form of this monodromy. In this section, we explain the following three results:
\begin{enumerate}
    \item The period central charge $\Psi$ commutes with the monodromy automorphisms.
    \item The action of the monodromy on the singular points of the isoresidual fiber can be deduced from the study of the monodromy in the case $n=1$ investigated in \cite[Section~5]{GeTaIso}.
    \item Saddle connections that shrink as $\lambda$ approaches a resonance hyperplane survive as saddle connections along the monodromy around this hyperplane.
\end{enumerate}
\par 
\subsubsection{The period central charge commutes with monodromy}\label{subsub:commute}
\par 
\begin{prop}\label{prop:CommuteWithMonodromy}
For any $\lambda \in \mathcal{B}$, any based loop $\alpha$ in $\mathcal{B}$, and any homology class $[\gamma] \in \calH_{\lambda}$, we have:
$$
\Psi(\lambda,T_{\alpha}[\gamma]) = \Psi(\lambda,[\gamma]).
$$
\end{prop}
\par 
\begin{proof}
The Gauss--Manin connection identifies nearby relative homology classes (with integer coefficients) and linear dependence of their periods in terms of $\lambda$, which can be shown to be the same through analytic continuation.
\end{proof}
\par 
Proposition~\ref{prop:CommuteWithMonodromy} provides the construction of many invariant subbundles of $\mathcal{E}$.
\par 
\begin{cor}\label{cor:InvariantSubbundle}
For any linear subspace $V$ of $(\mathcal{R}_{p})^{\ast}$, $\mathcal{E}^{V}=\Psi^{-1}(\mathcal{B} \times V)$ is an invariant subbundle of $\mathcal{E}$.
\end{cor}
\par 
In particular, the subbundle $\mathcal{E}^{\lbrace{ 0 \rbrace}}$ generated by the relative homology classes with vanishing periods is invariant under the monodromy of the Gauss--Manin connection.
\par 
\subsubsection{Action of the monodromy on the zeros and poles of $\omega_{\lambda}$}\label{subsub:MonodromyZerosPoles}
\par 
The relative homology group $\calH_{\lambda}$ is  part of the exact sequence
$$
0 \longrightarrow
H_{1}(\overline{\mathcal{F}}_{\lambda} \setminus P_{\omega_{\lambda}}) 
\longrightarrow H_{1}(\overline{\mathcal{F}}_{\lambda} \setminus P_{\omega_{\lambda}}, Z_{\omega_{\lambda}})
\longrightarrow
H_{0}(Z_{\omega_{\lambda}}) 
\longrightarrow
H_{0}(\overline{\mathcal{F}}_{\lambda} \setminus P_{\omega_{\lambda}})
\longrightarrow 0\,.
$$
\par 
Following Section~\ref{subsub:zeros}, the zeros of $\omega_{\lambda}$ are the elements of the isoresidual fiber of the minimal stratum $\mathcal{H}(a_{1}+a_{2},-b_{1},\dots,-b_{p})$. The action of the monodromy on $H_{0}(Z_{\omega})$ partially determines the action of the monodromy on $\calH_{\lambda}$. This monodromy on $H_{0}(Z_{\omega})$ is fully described in Section~5 of \cite{GeTaIso} for the case $p=3$ and partially for some other cases.
\par
We can now describe part of the action of the monodromy on $\calH_{\lambda}$ by the action on the poles of $\omega_{\lambda}$. As in the case of the zeros, it follows from the study of the action of the monodromy on some discrete isoresidual fibers.
\par
Following the results of Sections~\ref{subsub:SimplePoles} and~\ref{subsub:HigherPoles}, in the multi-scale compactification, the poles of $\omega_{\lambda}$ correspond to nodal curves with at most three irreducible components. The top-level component corresponds to a differential with zero residues, making it invariant under the monodromy. In contrast, the monodromy usually acts nontrivially on the other components. Since each of the latter components contains exactly one zero, they are described by the isoresidual fibration for the case $n=1$ as in \cite{GeTaIso}.
\par
This action of the monodromy, via permutations on $P_{\omega_{\lambda}}$, determines the action of the monodromy on the subspace of 
$H_{1}(\overline{\mathcal{F}}_{\lambda} \setminus P_{\omega_{\lambda}})$ generated by the homology classes of the loops around the poles.
\par 
\subsubsection{Action of the monodromy on saddle connections}
\par 
In $\mathcal{B}$, we consider a positively oriented simple loop $\alpha$ around the resonance hyperplane $_{A_{J}}$ corresponding to the resonance equation $\sum\limits_{j \in J} \lambda_{j} =0$. We will use the geometry of the saddle connections of $(\overline{\mathcal{F}}_{\lambda},\omega_{\lambda})$ as $\lambda$ approaches $A_{J}$ to construct a block decomposition of the monodromy automorphism $T_{\alpha}$ on $H_{1}(\overline{\mathcal{F}}_{\lambda} \setminus P_{\omega_{\lambda}})$.
\par
Denoting by $V$ the complex line in  $(\mathcal{R}_{p})^{\ast}$ generated by $\lambda \mapsto \sum\limits_{j  \in J} \lambda_{j}$, Corollary~\ref{cor:InvariantSubbundle} shows that $\mathcal{E}^{V}$ is an invariant subbundle. In particular, $\mathcal{E}_{\lambda}^{V}$ is invariant under the action of $T_{\alpha}$.
\par
For $\lambda$ close enough to $A_{J}$, we introduce the bundle $\mathcal{F}$, where $\mathcal{G}_{\lambda}$ is the subspace of $\mathcal{E}_{\lambda}^{V}$ generated by the saddle connections whose period central charges coincide with $\sum\limits_{j  \in J} \lambda_{j}$. These saddle connections become arbitrarily small (in contrast to the other saddle connections in the fiber) as $\lambda$ approaches $A_{J}$. Therefore, they cannot cross any singularity and persist as saddle connections as $\lambda$ moves along $\alpha$.
\par 
\begin{prop}\label{prop:MonoSC}
For $\lambda$ close enough to $A_{J}$, the automorphism $T_{\alpha}$ acts by permutation on a basis of the subbundle $\mathcal{G}_{\lambda} \subset \mathcal{E}_{\lambda}$. Similarly, $T_{\alpha}$ acts by permutation on a basis of the quotient bundle $\mathcal{E}_{\lambda}/\mathcal{G}_{\lambda}$.
\end{prop}
\par 
\begin{proof}
In an arbitrarily small neighborhood of a point in the regular locus of $A_{J}$, the Gauss-Manin connection around $A_{J}$ preserves saddle connections whose period is given by $\sum\limits_{j  \in J} \lambda_{j}$. Indeed, since these saddle connections are arbitrarily small compared to the other saddle connections of $(\overline{\mathcal{F}}_{\lambda},\omega_{\lambda})$, they cannot be crossed by them. It follows that they are permuted by $T_{\alpha}$.
\par
Since the relative homology group $\calH_{\lambda}$ is generated by the homology classes of saddle connections, the fiber $\mathcal{E}_{\lambda}/\mathcal{G}_{\lambda}$ of the quotient bundle is generated by saddle connections whose periods are not given by $\sum\limits_{j \in J} \lambda_{j}$.
\par
Along a loop in $\mathcal{B}$ that is close enough to the resonance hyperplane, the saddle connections that do not shrink can only be crossed by the vanishing saddle connections. It follows that their classes in the quotient fiber $\mathcal{E}_{\lambda}/\mathcal{G}_{\lambda}$ are the same, up to permutation.
\end{proof}
\par 
\subsection{Example: the case of $\mathcal{H}(1,1,-2,-1,-1)$}\label{sub:monoExample}
\par 
Note first that when $p=2$, the monodromy around the unique resonance hyperplane is trivial because the generic isoresidual fibers are related to each other by the scaling action of $\mathbb{C}^{\ast}$. Hence, the stratum $\mathcal{H}(1,1,-2,-1,-1)$ is the first nontrivial example.
\par
Here, $\mathcal{B} = \mathcal{R}_{3} \setminus \mathcal{A}_{3}$ is the complement of three resonance hyperplanes $A_{1}$, $A_{2}$, and $A_{3}$,  corresponding respectively to the vanishing of $\lambda_{1}$ (the residue at the double pole), and $\lambda_{2}$ and $\lambda_{3}$ (the residues at the simple poles).
\par
The zeroes of $\omega_{\lambda}$ are described in Theorem~\ref{thm:MAIN1} while the poles are described in Sections~\ref{subsub:SimplePoles} and~\ref{subsub:HigherPoles}. We obtain that the generic isoresidual fiber is a punctured sphere endowed with a translation structure belonging to the stratum $\mathcal{H}(2,2,-2,-2,-1,-1)$. In particular, it is of genus zero. The residues at the double poles are $\pm (\lambda_{2}-\lambda_{3})$, while the residues at the simple poles are $\pm\lambda_{1}$. The vector bundle $\mathcal{E}$ has rank $4$.
\par
We deduce from Section~\ref{subsub:MonodromyZerosPoles} that the monodromy of the isoresidual fibration preserves each pole of~$\omega_{\lambda}$. Therefore, $\mathcal{E}$ admits a subbundle $\mathcal{E}_{p}$ of rank $3$, on which the monodromy acts trivially. In particular, $\mathcal{E}_{p}$ contains the subbundle $\mathcal{E}^{\lbrace{ 0 \rbrace}}$ of rank $2$,  generated by the classes whose period charge vanishes (see Corollary~\ref{cor:InvariantSubbundle}).
\par
The relative homology class of any arc joining the two zeros of $\omega_{\lambda}$ generates the quotient bundle $\mathcal{E}/\mathcal{E}_{p}$ of rank $1$. Following Section~\ref{subsub:MonodromyZerosPoles}, we can verify that the monodromy of loops around $A_{2}$ and $A_{3}$ permutes the two zeros and therefore acts as $-\Id$ on $\mathcal{E}/\mathcal{E}_{p}$. In contrast, the monodromy of loops around $A_{1}$ preserves the two zeros individually and therefore acts trivially on $\mathcal{E}/\mathcal{E}_{p}$. The eigenvalues of the monodromy operators are then:
\begin{itemize}
    \item $(1,1,1,1)$ for the loops around $A_{1}$;
    \item $(1,1,1,-1)$ for the loops around $A_{2}$ and $A_{3}$).
\end{itemize}
In order to compute these monodromy operators, we will describe explicitly the translation structure of $\mathcal{F}_{\lambda}$ for a generic configuration of real residues~$\lambda$ (see Figure~\ref{fig:Sphere}). This can be accomplished by exhaustively writing down the ribbon graphs associated to~$\omega_\lambda$ as defined in Section~\ref{sub:RealGraph}. We observe that as $\lambda$ approaches any of the three resonance hyperplanes, exactly two of the four saddle connections of $(\mathcal{F}_{\lambda},\omega_{\lambda})$ shrink. According to Proposition~\ref{prop:MonoSC} there exists of a subbundle $\mathcal{G}_{\lambda,A_{i}}$ of rank $2$ (defined in a neighborhood of the resonance hyperplane $A_{i}$) on which the monodromy around that hyperplane acts as a permutation, while it also acts by permuting the quotient bundle. 

\subsubsection{Computation of the monodromy around $A_{1}$}

The eigenvalues of the monodromy operator are $(1,1,1,1)$ so these two permutations are trivial. Looking at the translation structure for a generic configuration of real residues close to $A_{1}$, we observe that each of the two shrinking saddle connections is a closed saddle connection enclosing a simple pole of $\omega_{\lambda}$. We deduce that the subbundle $\mathcal{G}_{\lambda,A_{1}}$ of rank $2$ (defined in a neighborhood of $A_{1}$) is contained in the rank-$3$ subbundle $\mathcal{E}_{p}$ where the monodromy acts trivially.
\par
Using the block decomposition of the matrix together with its commutation with the period charge, we obtain a matrix of the following form:
$$
\begin{pmatrix}1&0&1&-1\\0&1&0&0\\0&0&1&0\\0&0&0&1\end{pmatrix}.
$$
\subsubsection{Computation of the monodromy around $A_{2}$ and $A_{3}$}
The eigenvalues for the monodromy around $A_{2}$ and $A_{3}$ are $(1,1,1,-1)$. Again, by looking at the translation structure for a generic configuration of real residues close to $A_{2}$ (or $A_{3}$), we observe that the two shrinking saddle connections join to the two zeros. Since the monodromy around these hyperplanes transposes the two zeros, it follows that the monodromy acts as a transposition on $\mathcal{G}_{\lambda,A_{i}}$ and trivially on the quotient bundle $\mathcal{E}_{p} / \mathcal{G}_{\lambda,A_{i}}$ for $i =2$ and $3$.
\par
In order to write down explicitly the monodromy action corresponding to $A_{2}$ (or equivalently $A_{3}$) up to conjugacy, we factorize the matrix into two \textit{half-monodromy transformations} corresponding to the two half-loops around the resonance hyperplane connecting the two chambers of the space of configurations of real residues, where we have $\lambda_{2}>0$ and $\lambda_{2}<0$, respectively. We obtain that the monodromy matrix is given by 
$$
\begin{pmatrix}1&0&1&-1\\0&1&-1&1\\0&0&0&-1\\0&0&1&0\end{pmatrix}
$$
in a basis $(u_{1},u_{2},u_{3},u_{4})$ of $\mathcal{H}_{\lambda}$, formed by the four saddle connections of $(\mathcal{F}_{\lambda},\omega_{\lambda})$. Here, $\lambda_{2},\lambda_{3}>0$ while $\lambda_{1}<0$, $u_{3}$ and $u_{4}$ generate $\mathcal{G}_{\lambda,A_{2}}$, and $\Psi(u_{1})=\Psi(u_{2})=\lambda_{3}$ while $\Psi(u_{3})=\Psi(u_{4})=\lambda_{2}$ (see Figure~\ref{fig:Sphere} for a picture of the translation structure). We can check that the monodromy commutes with the period central charge.
\begin{figure}[ht]
 \centering
 \begin{tikzpicture}[,decoration={
    markings,
    mark=at position 0.45 with {\arrow[very thick]{>}}}]
\coordinate (a) at (-2,0);
\coordinate (b) at (2,0);

\draw[postaction={decorate}] (a)  .. controls ++(-90:1.2) and  ++(180:.5) .. (0,-1.5) .. controls ++(0:.5) and ++(-90:1.2) ..node[below right] {$u_{4}$} (b);
\draw[postaction={decorate}] (b)  .. controls ++(90:1.2) and  ++(0:.5) .. (0,1.5) .. controls ++(180:.5) and ++(90:1.2) ..node[above left] {$u_{2}$}  (a);
\draw[postaction={decorate}] (b)  .. controls ++(135:.7) and  ++(0:.5) .. (0,.6) .. controls ++(180:.5) and ++(45:.7) ..node[above] {$u_{3}$}  (a);
\draw[postaction={decorate}] (a)  .. controls ++(-45:.7) and  ++(180:.5) .. (0,-.6) .. controls ++(0:.5) and ++(-135:.7) ..node[above right] {$u_{1}$}  (b);
\foreach \i in {0,1}
\draw (a) --++(-15+30*\i:.3);
\foreach \i in {0,1}
\draw (b) --++(-15+30*\i:.3);

\fill[] (a) circle (2pt);\node[left] at (a) {$z_{2}$};
\filldraw[fill=white] (b)  circle (2pt);\node[left] at (b) {$z_{1}$};
\node at (0,0) {$d_{1}$};
\node at (0,1) {$s_{1}$};
\node at (0,-1) {$s_{2}$};
\node at (2.2,.7) {$d_{2}$};
\end{tikzpicture}
    \caption{The graph associated to the differential $\omega_{\lambda}$ where $z_{1},z_{2}$ are the zeros, $s_{1},s_{2}$ are the simple poles and $d_{1},d_{2}$ are the double poles.}\label{fig:Sphere}
\end{figure}
\par 
\begin{rmk}
The lattice of absolute periods of $\omega_{\lambda}$ is $(\lambda_{2}+\lambda_{3})\mathbb{Z}+(\lambda_{2}-\lambda_{3})\mathbb{Z}$. Indeed, the isoresidual fiber has genus zero, and we know the residues at the poles. In contrast, we can find saddle connections with periods $\lambda_{2}$ or $\lambda_{3}$ between the two zeros of $\omega_{\lambda}$. The relative period lattice of $\omega_{\lambda}$ is therefore $\lambda_{2}\mathbb{Z}+\lambda_{3}\mathbb{Z}$. These two lattices do not coincide.
\end{rmk}
\par 
\section{Connected components of generic isoresidual fibers}\label{sec:Connected}
\par 
The goal of this section is to prove the classification of the connected components of generic isoresidual fibers in genus zero, as stated in Theorem~\ref{thm:MAIN5}. To achieve this, we introduce topological invariants in Section~\ref{sub:ConnectedInvariant}, prove that they are complete for one-dimensional isoresidual fibers in Section~\ref{sub:ConnectedUni}, and then use the incidence relations between strata to deduce the result in full generality in Section~\ref{sub:Incidence}.
\par 
\subsection{Topological invariants}\label{sub:ConnectedInvariant}
\par 
Recall that the two exceptional families of strata where generic isoresidual fibers are disconnected are as follows:
\begin{itemize}
    \item $\mathcal{H}(kc_{1},\dots,kc_{m},-1,-1)$ with $k \geq 2$ and  $\gcd(c_{1},\dots,c_{m})=1$;
    \item $\mathcal{H}(c_{1},\dots,c_{m},\rec[-1][p])$ with $p,c_{1},\dots,c_{m}$ even.
\end{itemize}
In each case, we introduce topological invariants of the connected components in an indirect way as follows:
\par
1) We consider isoresidual fibers corresponding to configurations of residues where, for each pair of simple poles, one pole has a real positive residue while the other has a real negative residue. Notice that this condition can be satisfied by some residue configurations that do not belong to any resonance hyperplane (see Definition~\ref{defn:resonanceintro}).
\par
2) For each translation surface $(X,\omega)$ belonging to any such generic isoresidual fiber $\mathcal{F}_{\lambda}$, we glue the cylinders corresponding to the pairs of simple poles using a homothety. The resulting object is not a translation surface, but rather a dilation surface of genus $g \geq 1$.
\par
3) The moduli space of dilation surfaces is stratified in exactly the same way as that of translation surfaces. The charts of the complex affine structure around singular points are given by branches of $\int z^{m} g(z)dz$, where 
$g$ is a nonzero holomorphic function and $m \in \mathbb{C}$, with $m$ representing the order of the singularities. The Gauss--Bonnet identity $\sum m_{i}=2g-2$ also holds for complex affine structures. When the complex affine structure reduces to a dilation structure, its singularities satisfy $\Re(m) \in \mathbb{Z}$, and the moduli space is stratified according to these integer values (see \cite{ABW} for details).
\par
4) Any stratum $\mathcal{H}(a_{1},\dots,a_{n},-b_{1},\dots,-b_{p})$ of meromorphic differentials naturally embeds into the corresponding stratum
$\mathcal{D}(a_{1},\dots,a_{n},-b_{1},\dots,-b_{p})$ of dilation surfaces. The topological invariants deduced from the rotation numbers (which still make sense in a dilation surface, since slopes are globally defined) that we described in Section~\ref{sub:CCClassification} distinguish the connected components of $\mathcal{D}(a_{1},\dots,a_{n},-b_{1},\dots,-b_{p})$.

\subsubsection{The first exceptional family of strata}\label{subsub:FirstSurgery}

For a stratum of the form $\mathcal{H}(kc_{1},\dots,kc_{m},-1,-1)$, where $k \geq 2$ and $c_{1},\dots,c_{m}$ are positive or negative coprime integers satisfying $\sum\limits_{i=1}^m c_{i} = 0$, we consider a generic isoresidual fiber $\mathcal{F}_{\lambda}$ corresponding to a configuration $\lambda=(\lambda_{1},\dots,\lambda_{p})$ of residues,  where $\lambda_{p-1}$ is a positive real number and  $\lambda_{p}$ is a negative real number.
\par
We introduce the following surgery: for any translation surface $(X,\omega)$ in $\mathcal{F}_{\lambda}$, we glue the two cylinders corresponding to the two simple poles together. We choose one regular closed geodesic on each cylinder and identify them with each other using a homothety. As a result, we obtain a dilation surface $X'$ of genus one with a distinguished closed geodesic $\gamma$. 
\par
There is a well-defined notion of topological index (see Section~\ref{sub:rotation}) in the dilation surface $X'$. For any pair of simple loops $\alpha,\beta$ in $X'$ crossing $\gamma$ exactly once (with a positive orientation), the topological indices of $\alpha$ and $\beta$ differ by a combination of the orders $kc_{1},\dots,kc_{m}$. We thus define the \textit{surgical rotation number} of $(X,\omega)$ as the class of ${\rm Ind}_{\alpha}$ in $\mathbb{Z}/k\mathbb{Z}$.
\par
We can easily check that:
\begin{itemize}
    \item the surgical rotation number of $(X,\omega)$ does not depend on the choice of closed geodesics in the surgery performed on $(X,\omega)$;
    \item deformations of $(X,\omega)$ in the isoresidual fiber $\mathcal{F}_{\lambda}$ preserve the surgical rotation number.
\end{itemize}
Therefore, for each connected component $\mathcal{C}$ of $\mathcal{F}_{\lambda}$, there is a well-defined \textit{surgical rotation number} in $\mathbb{Z}/k\mathbb{Z}$. We will see in Proposition~\ref{prop:MAIN5} that the number of connected components in a generic isoresidual fiber is indeed $k$, proving in particular that the  \textit{surgical rotation number} is a complete invariant of the connected components for the fibers where it is defined (where $\lambda_{p-1}$ and $\lambda_{p}$ are respectively a positive and a negative real number).
\par 
Algebro-geometrically speaking, gluing the two simple poles yields a rational nodal curve as a degeneration of elliptic curves. The torsion number for the connected components of the stratum $\mathcal{H}(kc_1, \ldots, kc_m)$ in genus one (see \cite[Section 3.4]{chge}) can also be used to induce the surgical rotation number.  

\subsubsection{The second exceptional family of strata}\label{subsub:SecondSurgery}

For a stratum of the form $\mathcal{H}(c_{1},\dots,c_{m},\rec[-1][2g])$, where $g\geq 1$ and $c_{1},\dots,c_{m}$ are positive or negative even integers satisfying $\sum\limits_{i=1}^m c_{i} = 2g-2$, we consider a generic isoresidual fiber $\mathcal{F}_{\lambda}$ corresponding to a configuration of residues, where $g$ simple poles have pairwise distinct real positive residues while the other $g$ simple poles have pairwise distinct negative real residues.
\par
The surgery we introduce is as follows: We choose a regular closed geodesic on each cylinder corresponding to a simple pole. We denote by $\gamma_{1},\dots,\gamma_{g}$ (resp. $\delta_{1},\dots,\delta_{g}$) the closed geodesics chosen on the cylinders corresponding to the simple poles with negative (resp. positive) residues. Without loss of generality, we assume that $\gamma_{1},\dots,\gamma_{g}$ (resp. $\delta_{1},\dots,\delta_{g}$) are ordered by increasing length. Using homotheties to identify $\gamma_{i}$ with $\delta_{i}$ for each $i \in \lbrace{ 1,\dots,g \rbrace}$, we obtain a dilation surface $X'$ of genus $g \geq 1$.
\par
We know from \cite{ABW,Kawazumi} that the parity of the spin structure (or equivalently, the Arf invariant) generalizes from translation surfaces to dilation surfaces. We can then check that the parity of
$$
\sum\limits_{i=1}^{g} ({\rm Ind}_{\alpha_{i}}+1)({\rm Ind}_{\beta_{i}}+1)
$$
does not depend on the choice of geodesics $\gamma_{1},\dots,\gamma_{g}$ and $\delta_{1},\dots,\delta_{g}$. It is also constant along deformations of $(X,\omega)$ inside the isoresidual fiber, since the closed geodesics of the cylinders retain their directions.
\par
Then, for each connected component $\mathcal{C}$ of $\mathcal{F}_{\lambda}$, there is a well-defined \textit{surgical Arf invariant} in $\mathbb{Z}/2\mathbb{Z}$. We will see in Proposition~\ref{prop:MAIN5} that the number of connected components in a generic fiber is indeed $2$, proving in particular that the  \textit{surgical Arf invariant} is a complete invariant of the connected components for the fibers where it is defined (where the $g$ residues are positive real numbers while the other $g$ residues are negative real numbers).

\subsection{Incidence relations}\label{sub:Incidence}

In this section, we consider strata with the same pattern of poles $b_{1},\dots,b_{p}$, but the pattern of zeros can vary. To avoid confusion between isoresidual fibers of distinct strata, we introduce the notation $\mathcal{F}_{\lambda}(a_{1},\dots,a_{n})$ for the isoresidual fiber of $\mathcal{H}(a_{1},\dots,a_{n}) =
\mathcal{H}(a_{1},\dots,a_{n}, -b_1,\ldots, -b_p)$ corresponding to a configuration of residues $\lambda=(\lambda_{1},\dots,\lambda_{p})$.
\par 
We deduce from the surgery of breaking up zeros (see Corollary~\ref{cor:breakingup}) the following lemma about incidence relations between these isoresidual fibers.
\par 
\begin{lem}\label{lem:incidence}
Consider a stratum $\mathcal{H}(a_{1},\dots,a_{n})$ with $n \geq 2$, any configuration of residues $\lambda$, and a subset $I\subset \left\{1,\dots,n\right\}$. For any connected component of $\mathcal{F}_{\lambda}(a_{1},\dots,a_{n})$, there exists a connected component of $\mathcal{F}_{\lambda}(\sum_{i\in I} a_{i},\lbrace a_{j}\rbrace_{j\notin I})$ contained in its closure. Moreover, each connected component of $\mathcal{F}_{\lambda}(\sum_{i\in I} a_{i},\lbrace a_{j}\rbrace_{j\notin I})$ is adjacent to a unique connected component of $\mathcal{F}_{\lambda}(a_{1},\dots,a_{n})$.
\end{lem}
\par 
The proof of this result is a refinement of Propositions~7.1 and~7.2 of \cite{Bo}.
\par 
\begin{proof}
We first treat the case where $I=\left\{1,\dots,n\right\}$ and show that, given $\mathcal{F}_{\lambda}(a_{1},\dots,a_{n})$, there is an element in its closure that belongs to $\mathcal{F}_{\lambda}(\sum_{i\in I} a_{i})$.
\par 
We associate to $\omega\in \mathcal{F}_{\lambda}(a_{1},\dots,a_{n})$ an oriented graph in the following way. Consider a generic direction and the associated decomposition into basic domains (see \cite[section 3.3]{Bo}). Each pole corresponds to a vertex in the graph. For each saddle connection, we draw an arrow between the poles corresponding to the basic domains it bounds, such that the orientation goes from the lower domain to the upper domain.
\par 
The dimension of $\mathcal{H}(a_{1},\dots,a_{n})$ is $p+n-2$, so the graph has at least $p$ edges, since $n\geq2$. Therefore, it contains at least one simple closed path. Let $\gamma$ be such a path. If $\gamma$ is a loop (i.e., it contains only one vertex), then we are done by letting the corresponding saddle connection degenerate. 
\par
If $\gamma$ is a sum of edges, consider one of the shortest edges. Shrink it by a quantity $d$. Then, consider the next edge in the loop. If both edges point in the same direction, we change the latter saddle connection by~$-d$. If one edge points toward the vertex and the other points away from it, we add $d$ to the corresponding saddle connection. More generally, if the change at an edge is $\pm d$, then the change at the next edge is $\pm d$ if both edges point toward or away from the vertex, and $\mp d$ if the edges point in opposite directions.
\par 
Now, this can be done in a coherent way. Indeed, if $\gamma$ is a directed graph, then we add $d$ to all the corresponding saddle connections. If we reverse the orientation of exactly one edge in $\gamma$, then this only changes its contribution to $-d$ and makes no other change. Since any orientation of the edges can be obtained from the directed one by such changes, this proves the result.
\par 
\smallskip
\par
Now we can show uniqueness and the case $I \neq \left\{1,\dots,n\right\}$ at the same time. The multi-scale compactification is smooth at the points corresponding to $\mathcal{F}_{\lambda}(a_{1},\dots,a_{n})$ on the boundary. Hence, we can break the zeros not in $I$, and then those in $I$. Since breaking the zeros does not change the residues, this shows (by reversing this construction) that we can merge only the zeros in $I$. Finally, the smoothness implies the uniqueness statement,  since if two components were to meet, this would create a singularity.
\end{proof}

\subsection{Connected components of generic one-dimensional isoresidual fibers}\label{sub:ConnectedUni}

In this section, we prove that the surgical Arf invariant and the surgical rotation number are complete invariants for the connected components of generic isoresidual fibers in strata with $n=2$ zeros.
\par 
\begin{prop}\label{prop:MAIN5}
In strata of meromorphic differentials of genus zero, generic isoresidual fibers are connected, except for the following two families of strata:
\begin{enumerate}
    \item $\mathcal{H}(ka_{1},ka_{2},-kb_{1},\dots,-kb_{s},-1,-1)$ for some $k \geq 2$ and $a_{1},a_{2},b_{1},\dots,b_{s}$ positive integers such that $\gcd(a_{1},a_{2},b_{1},\dots,b_{s})=1$, where generic isoresidual fibers have $k$ connected components;
    \item $\mathcal{H}(2a_{1},2a_{2},-2b_{1},\dots,-2b_{s},\rec[-1][2g])$ with $a_{1},a_{2},b_{1},\dots,b_{s}\geq 1$ and $g \geq 2$, where generic isoresidual fibers have two connected components.
\end{enumerate}
\end{prop}
\par 
This entire Section~\ref{sub:ConnectedUni} is dedicated to the proof of this proposition. We begin with some general considerations.
\smallskip
\par
We know from Theorem~\ref{thm:MAIN1} that the type of $(\overline{\mathcal{F}}_{\lambda}, \omega_\lambda)$ as a translation structure is the same for any $\lambda \in \mathcal{R}_{p} \setminus \mathcal{A}_{p}$. Therefore, it suffices to determine the number of connected components of generic isoresidual fibers in the special case of an isoresidual fiber $\mathcal{F}_{\lambda}$, where $\lambda=(\lambda_{1},\dots,\lambda_{p})\in \mathcal{R}_{p} \setminus \mathcal{A}_{p}$ satisfies $\lambda_{1} \in \mathbb{R}_{>0}$ and $\lambda_{j} \in \mathbb{R}_{<0}$ for $2 \leq j \leq p$ as shown in Figure~\ref{fig:exgraph}.
\par
Following Section~\ref{subsub:zeros}, the zeros of $(\overline{\mathcal{F}}_{\lambda}, \omega_\lambda)$ correspond to elements of stratum $\mathcal{H}(a_{1}+a_{2},-b_{1},\dots,-b_{s},\rec[-1][t])$ that realize the configuration $\lambda$. We know from Lemma~\ref{lem:incidence} that every connected component of $(\overline{\mathcal{F}}_{\lambda}, \omega_\lambda)$ contains a zero of $\omega_{\lambda}$. To classify the connected components of $\mathcal{F}_{\lambda}$, we determine which pairs of zeros of $\omega_{\lambda}$ can be joined by a chain of saddle connections of $(\overline{\mathcal{F}}_{\lambda}, \omega_\lambda)$.
\par
It is proved in Lemma~\ref{lem:RealLocus} that when $\lambda$ is a configuration of real residues, every saddle connection of 
$(\overline{\mathcal{F}}_{\lambda}, \omega_\lambda)$ is horizontal. Differentials that belong to such a saddle connection define the same decorated graph (see Section~\ref{sub:RealGraph}), which we describe in the following lemma.
\par 
\begin{lem}\label{lem:Sec8SaddleConnection}
Consider a stratum $\calH(a_{1},a_{2},-b_{1},\dots,-b_{p})$ and a configuration of real residues $\lambda\in \mathcal{R}_{p} \setminus \mathcal{A}_{p}$ that satisfies $\lambda_{1} \in \mathbb{R}_{>0}$ and $\lambda_{j} \in \mathbb{R}_{<0}$ for $2 \leq j \leq p$. 
Saddle connections of $(\overline{\mathcal{F}}_{\lambda}, \omega_\lambda)$ correspond to decorated graphs that satisfy the following two sets of properties.
\par
The first set of properties ensures that the decorated graph is compatible with the data of $\calH(a_{1},a_{2},-b_{1},\dots,-b_{p})$: 
\begin{enumerate}
    \item Faces are labeled according to the zeros (of orders $a_{1}$ and $a_{2}$);
    \item Vertices are labeled according to the poles (with orders and residues $b_{1},\dots,b_{p}$ and $\lambda_{1},\dots,\lambda_{p}$);
    \item Each face corresponding to a zero of order $a_{i}$ contains $2a_{i}+2$ corners;
    \item For each vertex corresponding to a pole of order $b_{j}$, there are $b_{j}-1$ incoming half-edges and $b_{j}-1$ outgoing half-edges attached to the vertex.
\end{enumerate}
The second set of properties ensures that the orientations of the edges are compatible with the signs of the partial sums of residues (which are positive if and only if they include $\lambda_{1}$): 
\begin{enumerate}
    \item The graph contains a unique loop formed by $r+s+2$ vertices called $T,V_{1},\dots,V_{r},B,$ $U_{1},\dots,U_{s}$ (in this cyclic order) for $r,s \geq 0$;
    \item The edges of the loop are oriented from $B$ (as ``bottom'') to $T$ (as ``top'') in the two branches of the loop (through $U_{1},\dots,U_{s}$ or $V_{r},\dots,V_{1}$);
    \item The vertex $P$, corresponding to the pole of order $b_{1}$ with the only positive residue $\lambda_{1}$, either coincides with $T$ or belongs to a tree attached to $T$;    
    \item In a tree attached to a vertex of the loop (provided that the tree does not contain $P$), the edges are oriented towards the loop;
    \item In a tree attached to $T$ that contains $P$ as a vertex, the edges are oriented towards $P$.
\end{enumerate}
\par 
Denote by $\lambda_{B}$ the total (negative) residue of the poles corresponding to the vertex $B$ and the vertices of the trees attached to $B$. Then, the length of the corresponding saddle connection is $|\lambda_{B}|$. Moreover, each end of the saddle connection corresponds to shrinking one of the two edges of the loop incident to $B$.
\end{lem}
\par 
\begin{proof}
We first verify the properties of a decorated graph associated with a differential parameterized by a saddle connection $\gamma$ of $(\overline{\mathcal{F}}_{\lambda}, \omega_\lambda)$. Note that the first set of properties follows directly from Sections~\ref{sub:RealGraph} and~\ref{sub:FiberReal}.
\par 
For the second set of properties, the number of edges in the decorated graph is $n+p-2=p$, while the number of vertices is $p$. It follows that the graph has only one loop. As proved in Proposition~\ref{prop:REALlocus}, the orientation of the edges in the loop is incoherent. In other words, there are some vertices of the loop where the two incident edges of the loop are incoming, and the same number of vertices where the two incident edges (of the loop) are outgoing. For a translation surface associated with $\gamma$, cutting along all the saddle connections corresponding to the edges of the loop in the decorated graph decomposes the translation surface into connected components, each  having two boundary saddle connections. It is clear that the connected component corresponding to the vertex of the loop with two incoming edges has a positive total residue. It follows that:
\begin{itemize}
    \item there is only one vertex, called $T$, in the loop with two incoming edges;
    \item the vertex $P$ of the graph, corresponding to the unique pole with a positive residue, either coincides with $T$ or belongs to a tree attached to the loop at $T$;
    \item there is only one vertex, called $B$, in the loop with two outgoing edges.
\end{itemize}
The orientations of the edges in the loop are now fixed. The other edges correspond to closed saddle connections that decompose $\mathbb{CP}^{1}$ into two connected components. The orientations of the corresponding edges are fully determined by the signs of the partial sums of the residues.
\par
Since there is only one positive residue, we observe that the lengths of the saddle connections corresponding to the edges of the loop are decreasing. Specifically, the lengths of the saddle connections corresponding to the edges between $T,V_{1},\dots,V_{r},B$ (or similarly $T,U_{s},\dots,U_{1},B$) decrease and differ by constant partial sums of negative residues. The only degree of freedom in the saddle connection $\gamma$ is the length of the saddle connection corresponding to the edges of the loop, as it represents the relative period of the differential between the two zeros. Thus, the two saddle connections that can shrink correspond to the edges of the loop incident to $B$. The parametrization of $\gamma$ is given by the lengths of these two saddle connections, which sum to $\lambda_{B}$ (up to sign), and this is the length of $\gamma$. In particular, the two ends of $\gamma$ correspond to the shrinking of one of these saddle connections and the removal of the corresponding edges from the decorated graph (see Lemma~\ref{lem:defengraph}).
\par
Conversely, for any such configuration of real residues with a unique positive residue, the orientation of the edges we prescribed to the graph ensures that we can construct the translation surface corresponding to the decorated graph in such a way that the lengths of the saddle connections corresponding to the edges are positive.
\end{proof}
\par 
We describe the decorated graph of a differential $\omega$ parameterized by a saddle connection of $(\overline{\mathcal{F}}_{\lambda}, \omega_\lambda)$ in the case of strata with only simple poles.
\begin{ex}
Consider a stratum $\calH(a_{1},a_{2},\rec[-1][t])$ in genus zero with only simple poles. A decorated graph parameterized in a saddle connection of $\omega_\lambda$
has $t$ vertices $V_{1},\dots,V_{t}$ (corresponding to the poles $p_{1},\dots,p_{t}$) and $t$ edges pointing towards the vertex~$V_{1}$. It follows that there is a unique pole of residue $\lambda_{j}<0$ whose corresponding vertex $V_{j}$ has valency $2$. The loop is formed by the two edges from $V_{j}$ to $V_{1}$ and decomposes the underlying sphere into two connected components. The other $t-2$ vertices are attached to $V_{1}$ each through a unique edge. They are split into two families having respectively $a_{1}$ and $a_{2}$ vertices according to which face they belong to. This example is illustrated in Figure~\ref{fig:exgraph} for the case $t = 5$ and $j=4$.
\end{ex}
\par 
In what follows, we first prove Proposition~\ref{prop:MAIN5} in the case of strata without simple poles.
\begin{lem}\label{lem:MAIN5nosimple}
In strata $\calH(a_{1},a_{2},-b_{1},\dots,-b_{p})$ of meromorphic differentials of genus zero with $b_{i}\geq2$ for all $i$, the generic isoresidual fibers are connected.
\end{lem}
\begin{proof}
We consider the case of a unique positive residue at a pole of lowest order and negative residues at all the other poles. The case with two poles is trivial since there is only one decorated graph. Now, suppose there are $p\geq3$ poles, and let $p_i$ denote the poles such that their orders satisfy $-b_{1}\geq -b_{2}\geq \dots \geq -b_{p}$. We will assume that the residue of $p_{1}$ is positive, and the residue of the other poles are negative.
\par
Now consider the graph associated with a differential in the generic isoresidual fiber of the stratum $\calH(a_{1}+a_{2},-b_{1},\dots,-b_{p})$. This graph has a special vertex $V_{1}$, such that all the other vertices point towards it. First, we deform the graph so that all vertices are directly adjacent to $V_{1}$. Consider a leaf of the graph. Move it by connecting its first quadrant to the rest of the graph, such that the bounded region corresponds to the zero of order $a_{1}$, and then delete the original edge (see the top of Figure~\ref{fig:operation2}). Repeat this operation until we return to the original place. If the vertex is attached to~$V_{1}$ at some step, we are done. Otherwise, we start from the second quadrant as pictured at the bottom of Figure~\ref{fig:operation2}  and repeat the operation until we return to the original place. Note there is a forbidden quadrant between the first and the second one due to the compatibility condition of the directions of the edges: this will always be the case and we will forget these quadrant in our counts.  With this operation, we can visit all the places of the previous case, shifted by one quadrant. We encounter the same dichotomy, and if we do not meet $V_{1}$, we can repeat the procedure to eventually reach this vertex.
\begin{figure}[ht]
\begin{tikzpicture}[scale=1,decoration={
    markings,
    mark=at position 0.45 with {\arrow[very thick]{>}}}]

\begin{scope}[xshift=0cm,yshift=0cm]
\node[circle,draw] (b) at  (0,0) {$p_{1}$};
 \node[circle,draw] (d) at  (-140:1.5) {$p_{2}$};
  \node[circle,draw] (e) at  (1.5,0) {$p_{3}$};
\draw[postaction={decorate}] (d) -- (b);
\draw[postaction={decorate},dotted] (e)  .. controls ++(120:.7) and ++(30:.7) .. (-.5,.7) .. controls ++(-150:.5) and
++(-150:1) .. (b);
\draw[postaction={decorate}] (e)  --  (b);
\draw (e) --++(60:.5);
\draw (e) --++(-60:.5);
   \foreach \i in {0,1,...,5}
\draw (b) --++(90+36*\i:.5);

\draw[->] (2.3,0) -- (3.6,0);
\end{scope}

\begin{scope}[xshift=5cm,yshift=0cm]
\node[circle,draw] (b) at  (0,0) {$p_{1}$};
 \node[circle,draw] (d) at  (-140:1.5) {$p_{2}$};
  \node[circle,draw] (e) at  (1.5,0) {$p_{3}$};
\draw[postaction={decorate}] (d) -- (b);
\draw[postaction={decorate}] (e)  .. controls ++(120:.7) and ++(30:.7) .. (-.5,.7) .. controls ++(-150:.5) and
++(-150:1) .. (b);
\draw (e) --++(60:.5);
\draw (e) --++(-60:.5);
   \foreach \i in {0,1,...,5}
\draw (b) --++(90+36*\i:.5);

\node at (2.5,0) {$\simeq$};
\end{scope}

\begin{scope}[xshift=10cm,yshift=0cm]
\node[circle,draw] (b) at  (0,0) {$p_{1}$};
 \node[circle,draw] (d) at  (-160:1.5) {$p_{2}$};
  \node[circle,draw] (e) at  (160:1.5) {$p_{3}$};
\draw[postaction={decorate}] (d) -- (b);
\draw[postaction={decorate}] (e)  -- (b);
\draw (e) --++(120:.5);
\draw (e) --++(-120:.5);
   \foreach \i in {0,1,...,5}
\draw (b) --++(-90+36*\i:.5);
\end{scope}
\begin{scope}[xshift=0cm,yshift=-3cm]
\node[circle,draw] (b) at  (0,0) {$p_{1}$};
 \node[circle,draw] (d) at  (-140:1.5) {$p_{2}$};
  \node[circle,draw] (e) at  (1.5,0) {$p_{3}$};
\draw[postaction={decorate}] (d) -- (b);
\draw[postaction={decorate},dotted] (e)  .. controls ++(-80:.9) and ++(-30:1.7) .. (1.5,.7)  .. controls ++(150:.4) and ++(0:.4) .. (.2,1).. controls ++(180:.5) and ++(140:.9) .. (b);
\draw[postaction={decorate}] (e)  --  (b);
\draw (e) --++(60:.5);
\draw (e) --++(-60:.5);
   \foreach \i in {0,1,...,5}
\draw (b) --++(90+36*\i:.5);

\draw[->] (2.3,0) -- (3.6,0);
\end{scope}

\begin{scope}[xshift=5cm,yshift=-3cm]
\node[circle,draw] (b) at  (0,0) {$p_{1}$};
 \node[circle,draw] (d) at  (-140:1.5) {$p_{2}$};
  \node[circle,draw] (e) at  (1.5,0) {$p_{3}$};
\draw[postaction={decorate}] (d) -- (b);
\draw[postaction={decorate}] (e)  .. controls ++(-80:.9) and ++(-30:1.7) .. (1.5,.7)  .. controls ++(150:.4) and ++(0:.4) .. (.2,1).. controls ++(180:.5) and ++(140:.9) .. (b);
\draw (e) --++(60:.5);
\draw (e) --++(-60:.5);
   \foreach \i in {0,1,...,5}
\draw (b) --++(90+36*\i:.5);

\node at (2.6,0) {$\simeq$};
\end{scope}

\begin{scope}[xshift=10cm,yshift=-3cm]
\node[circle,draw] (b) at  (0,0) {$p_{1}$};
 \node[circle,draw] (d) at  (-140:1.5) {$p_{2}$};
  \node[circle,draw] (e) at  (140:1.5) {$p_{3}$};
\draw[postaction={decorate}] (d) -- (b);
\draw[postaction={decorate}] (e) -- (b);
\draw (e) --++(120:.5);
\draw (e) --++(-120:.5);
   \foreach \i in {0,1,...,5}
\draw (b) --++(90+36*\i:.5);
\end{scope}
\end{tikzpicture}
 \caption{A move of a leaf from the first quadrant on the top and from the second quadrant on the bottom.} \label{fig:operation2}
\end{figure}
\smallskip
\par
So, the graph we consider is given by a root vertex  $V_{1}$ and $p-1$ leaves $V_{2},\dots,V_{p}$ pointing towards it. We show that this graph can be put into the following standard form: the vertices point towards the same quadrant of the vertex $V_{1}$ and are arranged in a cyclic increasing order. Take $V_{3}$ and move it along the graph by adding an edge starting from its first quadrant, as described in the preceding paragraph (and shown at the top of Figure~\ref{fig:operation2}). Repeating this operation, either it reaches the first quadrant of $V_{1}$ after $V_{2}$, or it returns to its original position. In the latter case, we proceed as in the preceding paragraph, and eventually, it reaches the desired position. We apply the same procedure to every edge: this is possible since all the vertices have half-edges and thus at least two quadrants. 
\end{proof}
\par 
Next, we prove Proposition~\ref{prop:MAIN5} for strata with only simple poles.
\par 
\begin{lem}\label{lem:MAIN5Simple}
In a stratum $\mathcal{H}(a_{1},a_{2},\rec[-1][t])$ of differentials in genus zero with only simple poles, the generic isoresidual fibers have:
\begin{itemize}
    \item two connected components, classified by the surgical Arf invariant, if $a_{1}$ and $a_{2}$ are even;
    \item only one connected component, otherwise.
\end{itemize}
\end{lem}
\par 
\begin{proof}
Since $\lambda$ is a configuration of real residues with a unique positive residue, the associated decorated trees are easy to describe. The $(t-2)!$ zeros of the translation structure $\omega_{\lambda}$ (see Proposition~\ref{prop:zeros}) correspond to trees with one vertex $V_{1}$ of valency $t-1$
and $t-1$ vertices
attached to it according to some cyclic order.
Each zero of $\omega_{\lambda}$ thus corresponds to a permutation $\sigma$ on $\lbrace{2, \dots, t \rbrace}$ such that $\sigma(j)$ follows $j$ in the cyclic order.
\par
Suppose there exists a saddle connection $\gamma$ on $(\overline{\mathcal{F}}_{\lambda}, \omega_\lambda)$ joining two zeros of $\omega_{\lambda}$, corresponding respectively to permutations $\sigma$ and $\tau$ of $\lbrace{ 2, \dots, t \rbrace}$. Then, $\sigma$ and $\tau$ differ only by the location of $j_{0}$ in the cyclic order, depending on which edge incident to $V_{j_{0}}$ is deleted, as explained in Lemma~\ref{lem:defengraph} and Figure~\ref{fig:operation}. The location of $j_{0}$ in the cyclic order changes either  by $a_{1}$ or by $a_{2}$.  In other words, $\sigma \circ \tau^{-1}$ is a cycle of order $a_{1}+1$ or $a_{2}+1$. 
Since each such cycle corresponds to a saddle connection of $\omega_\lambda$, these cycles generate the whole alternate group $\mathfrak{A}_{t-1}$ if $a_{1},a_{2}$ are even, and the whole symmetric group $\mathfrak{S}_{t-1}$ otherwise.
 \begin{figure}[ht]
 \centering
 \begin{tikzpicture}[,decoration={
    markings,
    mark=at position 0.4 with {\arrow[very thick]{>}}}]
\begin{scope}[xshift=0cm]
 \node[circle,draw] (a) at  (-2,1) {$2$};
\node[circle,draw] (b) at  (0,0) {$1$};
\node[circle,draw] (c) at  (-2,-1) {$3$};
 \node[circle,draw] (d) at  (1.5,0) {$4$};
  \node[circle,draw] (e) at  (3,0) {$j_{0}$};
\draw [postaction={decorate}]  (a) -- (b);
\draw[postaction={decorate}] (c) -- (b);
\draw[postaction={decorate}] (d) -- (b);
\draw[postaction={decorate},dotted] (e)  .. controls ++(120:1.4) and
++(60:1.4) ..  (b);
\draw[postaction={decorate}] (e)  .. controls ++(-120:1.4) and
++(-60:1.4) ..  (b);


\draw[->] (4,0) -- (6,0);
\end{scope}

\begin{scope}[xshift=8cm]
 \node[circle,draw] (a) at  (-2,1) {$2$};
\node[circle,draw] (b) at  (0,0) {$1$};
\node[circle,draw] (c) at  (-2,-1) {$3$};
 \node[circle,draw] (d) at  (1.5,0) {$4$};
  \node[circle,draw] (e) at  (3,0) {$j_{0}$};
\draw [postaction={decorate}]  (a) --  (b);
\draw[postaction={decorate}] (c) -- (b);
\draw[postaction={decorate}] (d) -- (b);
\draw[postaction={decorate}] (e)  .. controls ++(120:1.4) and
++(60:1.4) ..  (b);
\end{scope}
\end{tikzpicture}
    \caption{Going from one tree to another.}\label{fig:operation}
\end{figure}
\par
It follows that, provided $a_{1}$ and $a_{2}$ are even, the two zeros corresponding to the permutations $\sigma$ and $\tau$ belong to the same connected component of $(\overline{\mathcal{F}}_{\lambda}, \omega_\lambda)$ if and only if $\sigma$ and $\tau$ have the same signature. Indeed, the decomposition of $\sigma \circ \tau^{-1}$ into a product of cycles of length $a_{1}+1$ provides the chain of saddle connections joining the corresponding zeros. If either $a_{1}$ or $a_{2}$ is odd, the same argument shows that every pair of zeros can be joined by a chain of saddle connections. Therefore, the isoresidual fiber $\overline{\mathcal{F}}_{\lambda}$ is connected.
\end{proof}
\par 
We now prove the case in which there is a unique non-simple pole.
\par 
\begin{lem}\label{lem:MAIN51big}
In a stratum $\mathcal{H}(a_{1},a_{2},-b,\rec[-1][t])$ of differentials in genus zero with $b\geq2$, the generic isoresidual fibers have:
\begin{itemize}
    \item $k$ connected components, where $k=\gcd(a_{1},a_{2},b)$, if $t=2$;
    \item two connected components, if $a_{1},a_{2},b$ are even and $t \geq 4$;
    \item only one connected component, otherwise.
\end{itemize}
\end{lem}
\par 
\begin{proof}
We consider an isoresidual fiber $\mathcal{F}_{\lambda}$,  where the residues at the simple poles are negative real numbers and the residue at the pole of order $b$ is positive. The decorated graphs corresponding to the zeros of $(\overline{\mathcal{F}}_{\lambda}, \omega_\lambda)$ have a vertex $V_{1}$ with $2(b-1)$ half-edges and $t-1$ vertices pointing towards it. We analyze which zeros can be connected by a saddle connection.
\smallskip
\par 
Since the case with a unique simple pole is trivial (as there is a unique graph), we begin with the case of two simple poles. We fix the vertex $V_{2}$ corresponding to the first simple pole and move $V_{3}$ along the graph. The move consists of drawing a new edge that partitions the vertices and half-edges into two sets of cardinality $a_{1}$ and $a_{2}$, and then deleting the original edge (see Figure~\ref{fig:operation}). There are $b$ possible positions for $V_{3}$, and the distance between   consecutive positions is either $a_{1}$ or $a_{2}$. Therefore, there are $\gcd(a_{1},a_{2},b)$ possible orbits for the position of $V_{3}$. 
\smallskip
\par 
Now, suppose there are $t \geq 3$ simple poles. Again, we fix the position of $V_{2}$ and want to place the other vertices $V_{i}$ for $3\leq i\leq t-1$ in the same quadrant and in increasing order. Consider $V_{3}$ and move it along the graph using the operation shown in  Figure~\ref{fig:operation}. If it moves next to $V_{2}$, we are done. Otherwise, continue moving it until it jumps  above $V_t$. Then, we make one move of $V_{t}$ so that it jumps over $V_{3}$ and let $V_{3}$ jump over it again. In this way, the position that $V_{3}$ will visit will be shifted by one quadrant relative to the previous moves. Either it will reach the correct position, or we can repeat the operation until it does. Thus, we can place all the $V_{i}$ with $i \leq t-1$ in their correct positions. 
\par
Next, consider the last vertex $V_{t}$. We first place it after the vertices $V_{t-2}$ and $V_{t-1}$, possibly switching the order of these two vertices. After moving $V_{t}$ along the graph, if it reaches the correct quadrant, we are done. Otherwise, either we jump over both  $V_{t-2}$ and $V_{t-1}$, or we end up between them. In both cases, first move $V_{t-2}$, then $V_{t}$, and finally place $V_{t-2}$ back. This will switch the position of $V_{t}$ by one quadrant and swap the relative positions of $V_{t-2}$ and $V_{t-1}$. In the second case, we are done, and in the first case, we can repeat this operation after rotating $V_{t}$ along the graph. Eventually, the vertex $v_{t}$ can be placed in the correct quadrant.
\par
Finally, it remains to show that if one of $a_{1},a_{2}$ or $b$ is odd, then we can switch the positions of $V_{t-1}$ and $V_{t-2}$. In the case with three simple poles, this is achieved by first moving $V_{t-1}$, then $V_{t-2}$, and finally $V_{t}$ in the negative direction. If there are more simple poles, we first perform this move to switch the positions of $V_{t-1}$ and $V_{t-2}$. This, however, changes the position of the triple of vertices by $-a_{1}+1$ or $-a_{2}+1$ quadrants. Then, by going back, we see that this triple is shifted by one quadrant. Thus, by repeating this operation, we obtain the triple in the desired order in any quadrant with an odd distance from the original one. Since the total number of quadrants where this triple can be is $b+t-3$, if $b+t$ is odd, then we are done. If either $a_{1}$ or $a_{2}$ is odd, say $a_{1}$, we can first move the triple by $a_{1}$ quadrants and then perform the same operation as in the previous case to place it in the desired position. Now, suppose both $b$ and $t$ are odd, and $a_{1},a_{2}$ are even. In that case, we first move all the remaining $t-3$ simple poles by $-a_{1}$. Then, the triple lands in a new quadrant an odd distance from the desired quadrant. Thus, we can repeat the previous steps from this new position to reach the desired final configuration. 
\end{proof}
\par 
Now, we can prove the general case.
\par 
\begin{proof}[Proof of Proposition~\ref{prop:MAIN5}]
By the previous results, we can assume that there are at least two poles of higher orders and at least one simple pole. If there is a unique simple pole, we can assume that it has the unique positive real residue. Then, all the other poles can be moved freely, as shown in the proof of Lemma~\ref{lem:MAIN5nosimple}. Hence, the generic isoresidual fiber is connected. 
\par
Now, suppose there are $s= 2$ simple poles. In this case, we fix one simple pole to be the unique pole with a positive residue. Then, the non-simple poles can be placed freely, as in the proof of Lemma~\ref{lem:MAIN5nosimple}. We place them such that they all point towards the simple pole with the positive residue. Now, consider the second simple pole, called $q$. Note that there are are $\sum b_{i}$ possible positions, so by performing the operation of Figure~\ref{fig:operation}, we can visit all the quadrants at a distance of $a_{1}$ or $a_{2}$ from the previous one. Hence, there are $d:=\gcd(a_{1},a_{2},\sum b_{i})$ distinct classes of positions. If $d=k$ then we are done. Otherwise, take $b_{j}$ such that $\gcd(d,b_{j})<d$. Move the corresponding vertex $V_{j}$ parallel to $q$  such that they do not interact.  After $q$ passes the quadrant where $V_{j}$ was attached, stop it and place $b_{j}$ back in its original position. Then $q$ is switched by $b_{j}$, and we obtain all the quadrants with a distance of $\gcd(a_{1},a_{2},b_{j},\sum b_{i})$. We complete the result by considering all the $b_{i}$'s.
 \par
Finally, we consider the case where there are $s\geq3$ simple poles. We first place the graph without the three simple poles in the desired form, as in the proof of Lemma~\ref{lem:MAIN51big}. Then, the three simple poles can be placed in any position, up to order, if all the other orders are divisible by $2$, as in the proof of Lemma~\ref{lem:MAIN5nosimple}. Hence, we obtain the desired upper bound on the number of connected components.
\end{proof}

\subsection{Proof of Theorem~\ref{thm:MAIN5}}

By combining Proposition~\ref{prop:MAIN5} with Lemma~\ref{lem:incidence}, we can extend the classification of the connected components of the generic isoresidual fibers to strata in genus zero with an arbitrary number of zeros.
\par 
\begin{proof}[Proof of Theorem~\ref{thm:MAIN5}]
We first consider the case of an isoresidual fiber $\mathcal{F}_{\lambda}(a_{1},\dots,a_{n})$ for which there is a permutation $\sigma$ such that $\mathcal{F}_{\lambda}(a_{\sigma(1)}+\dots+a_{\sigma(a_{n-1})},a_{\sigma(n)})$ is connected. Then, any connected component $\mathcal{C}$ of $\mathcal{F}_{\lambda}(a_{1},\dots,a_{n})$ contains $\mathcal{F}_{\lambda}(a_{\sigma(1)}+\dots+a_{\sigma(a_{n-1})},a_{\sigma(n)})$ in its closure. Lemma~\ref{lem:incidence} then shows that $\mathcal{C}$ is the only connected component of $\mathcal{F}_{\lambda}(a_{1},\dots,a_{n})$.
\par
In the remaining cases, we consider isoresidual fibers  $\mathcal{F}_{\lambda}(a_{1},\dots,a_{n})$ with $n \geq 3$, such that for any permutation $\sigma$, the stratum $\mathcal{H}(a_{\sigma(1)}+\dots+a_{\sigma(a_{n-1})},a_{\sigma(n)},-b_{1},\dots,-b_{p})$ belongs to one of the two exceptional families described in 
Proposition~\ref{prop:MAIN5}. In particular, the orders of the poles are either $(-1,\dots,-1)$, or of the form $(-kd_{1},\dots,-kd_{p-2},-1,-1)$ with $k \geq 2$.
\par
We first consider the case where $b_{1}=\dots=b_{p}=1$. If, for every permutation $\sigma$, the isoresidual fiber $\mathcal{F}_{\lambda}(a_{\sigma(1)}+\dots+a_{\sigma(a_{n-1})},a_{\sigma(n)})$ fails to be connected, then, by Proposition~\ref{prop:MAIN5}, it follows that all the $a_{1},\dots,a_{n}$ are even. The surgical Arf invariant, described in Section~\ref{subsub:SecondSurgery}, is a well-defined topological invariant for the connected components of $\mathcal{F}_{\lambda}(a_{1},\dots,a_{n})$. The argument that works in the general case then proves that the subset of $\mathcal{F}_{\lambda}(a_{1},\dots,a_{n})$ consisting of differentials with an even (resp. odd) surgical Arf invariant (see Section~\ref{subsub:SecondSurgery}) is connected. Since every fiber of the form $\mathcal{F}_{\lambda}(a_{\sigma(1)}+\dots+a_{\sigma(a_{n-1})},a_{\sigma(n)})$ has exactly one (non-empty) connected component for each parity of the invariant, we deduce that $\mathcal{F}_{\lambda}(a_{1},\dots,a_{n})$ has exactly two connected components.
\par
Finally, we consider the case of a stratum $\mathcal{H}(a_{1},\dots,a_{n},-Ld_{1},\dots,-Ld_{p-2},-1,-1)$, where two poles are simple and the others have orders divisible by $L \geq 2$. Since for every permutation $\sigma$ and any $t$, the isoresidual fiber $\mathcal{F}_{\lambda}(\sum\limits_{1 \leq i \leq t} a_{i}, \sum\limits_{t< i\leq n} a_{i})$ fails to be connected, it follows from Proposition~\ref{prop:MAIN5} that for any $i$, we have $L_{\sigma,t} \neq 1$, where $L_{\sigma,t} = \gcd(\sum\limits_{1 \leq i \leq t} a_{i},L)$.
\par
For each stratum $\mathcal{H}(\sum\limits_{1 \leq i \leq t} a_{i}, \sum\limits_{t< i\leq n} a_{i})$, the generic isoresidual fibers have $L_{\sigma,t}$ connected components, which are classified by the surgical rotation number in $\mathbb{Z}/L_{\sigma,t}\mathbb{Z}$ (see Section~\ref{subsub:FirstSurgery}). We observe that the monodromy of the isoresidual fibration around the resonance hyperplane given by $\lambda_{p}=0$ changes the surgical rotation number by one. Since every connected component of $\mathcal{F}_{\lambda}(\sum\limits_{1 \leq i \leq t} a_{i}, \sum\limits_{t< i\leq n} a_{i})$ belongs to the closure of exactly one connected component of $\mathcal{F}_{\lambda}(a_{1},\dots,a_{n})$, it follows that the number of connected components of $\mathcal{F}_{\lambda}(a_{1},\dots,a_{n})$ is at most the greatest common divisor of the $L_{\sigma,t}$. In particular, $\mathcal{F}_{\lambda}(a_{1},\dots,a_{n})$ has at most $k$ connected components, where  $k=\gcd(a_{1},\dots,a_{n},L)$. 
\par
Conversely, we know that $\mathcal{F}_{\lambda}(a_{1},\dots,a_{n})$ has at least $k$ connected components because the surgical rotation number in $\mathbb{Z}/k\mathbb{Z}$ is a topological invariant (see Section~\ref{subsub:FirstSurgery}), and every number in $\mathbb{Z}/k\mathbb{Z}$ can be realized by some connected component of $\mathcal{F}_{\lambda}(a_{1},a_{2}+\dots+a_{n})$. For example, the connected components may be classified by a surgical rotation number in a larger cyclic group, but the class modulo $k$ remains unchanged by the splitting of the zero. Therefore, $\mathcal{F}_{\lambda}(a_{1},\dots,a_{n})$ has exactly $k$ connected components, and they are classified by their surgical rotation number in $\mathbb{Z}/k\mathbb{Z}$.
\end{proof}

\section{The orders of singularities in strata and isoresidual fibers}\label{sec:Global}

We know from Section~\ref{sub:CCClassification} that some strata of meromorphic differentials can be disconnected if the singularity pattern satisfies some arithmetic properties. The formulas in Section~\ref{sub:boundary} provide a way to compute the number and order of the zeros and poles of the translation structure of generic isoresidual fibers in terms of the singularity pattern of the stratum. In this section, we use these formulas to derive general combinatorial relations between these orders.
\par
In our notation, we distinguish between the singularities of a generic isoresidual fiber $\mathcal{F}$ and the orders of singularities in the pattern that define the stratum $\mathcal{H}$. By a slight abuse of notation, we will refer to the  singularities of $\mathcal{H}$ (which, in fact, are the singularities of the translation surfaces parametrized by $\mathcal{H}$).
\par 
\subsection{Simple poles in strata and  isoresidual fibers}\label{sub:7Simple}
\par 
We will prove that every pole of $(\overline{\mathcal{F}}_{\lambda},\omega_{\lambda})$ is a simple pole if and only if the stratum is of the form $\mathcal{H}(a_{1},a_{2},-1,\dots,-1)$.
\par 
\begin{proof}[Proof of the first part of Theorem~\ref{thm:MAIN3}]
We first consider a stratum $\mathcal{H}$ parametrizing differentials with only simple poles. Then, every pole of a generic isoresidual fiber of $(\overline{\mathcal{F}}_{\lambda},\omega_{\lambda})$ corresponds to a degeneration where $K$ is empty (see Section~\ref{subsub:SimplePoles}). It follows that every pole of $\omega_{\lambda}$ is a simple pole.
\par
Conversely, we assume that the partition $\mu$ contains a number $c \leq -2$. Suppose, by contradiction, that every pole of $\omega_{\lambda}$ is a simple pole. For such a simple pole of $\omega_{\lambda}$, there is a partition $I \sqcup J$ (where both $I$ and~$J$ are nonempty) of the set of poles of the parametrized differentials. Without loss of generality, we assume that the pole of order $c$ belongs to $J$.
\par
If $I$ contains a simple pole, then we can transfer this simple pole from $I$ to $J$, and the pole of order $c$ from~$J$ to $K$. This gives a degeneration corresponding to a nonsimple pole of $\omega_{\lambda}$.
\par
If $I$ does not contain any simple pole, then it contains a nonsimple pole (since we know by hypothesis that $I \neq \emptyset$). We then  transfer this pole from $I$ to $K$, while transferring the pole of order $c$ to $K$ at the same time. We obtain a degeneration corresponding to a nonsimple pole of $\omega_{\lambda}$. Thus, $\omega_{\lambda}$ must have a nonsimple pole.
\par
Notice that the case where $\omega_{\lambda}$ has no poles at all cannot occur because, when the configuration $\lambda$ is real, all the periods of $\omega_{\lambda}$ are real. This would be impossible for a translation surface of finite area (which corresponds to a holomorphic one-form).
\end{proof}
\par 
\subsection{Even singularities in strata and isoresidual fibers}\label{sub:7even}
\par 
The goal of this section is to prove the second part of Theorem~\ref{thm:MAIN3}. We first show  that when the singularity pattern of $\mathcal{H}(\mu)$ contains only singularities of even order, the translation structure $\omega_\lambda$ of generic isoresidual fibers also have only zeros and poles of even order.
\par 
\begin{prop}\label{prop:Sec7Direct}
For a stratum $\mathcal{H}(a_{1},a_{2},-b_{1},\dots,-b_{p})$ where $a_{1},a_{2},b_{1},\dots,b_{p}$ are even, every zero or pole of a generic isoresidual fiber $(\overline{\mathcal{F}}_{\lambda},\omega_{\lambda})$ has even order.
\end{prop}
\par 
\begin{proof}
For the zeros of $\omega_{\lambda}$, the claim follows from Proposition~\ref{prop:zeros}. Any pole of $\omega_{\lambda}$ corresponds to a tripartition $I \sqcup J \sqcup K$ of $b_{1},\dots,b_{p}$. The order of the pole can be $a_{1}-b_{I}+2$ (if $J$ is empty), $a_{2}-b_{J}+2$ (if $I$ is empty), $2$ (if both $I$ and $J$ are empty) or $1+\lcm(a_{1}-b_{I}+1,a_{2}-b_{J}+1)$ otherwise. By hypothesis, $a_{1}-b_{I}+1$ and $a_{2}-b_{J}+1$ are odd, so a pole of $\omega_{\lambda}$ is of even order in all these cases.
\end{proof}
\par 
The converse of Proposition~\ref{prop:Sec7Direct} (which forms the second half of Theorem~\ref{thm:MAIN3}) does not hold for $p=2$. Using the results from  Section~\ref{sub:boundary}, we compute the following example.
\par 
\begin{ex}
Generic isoresidual fibers $(\overline{\mathcal{F}}_{\lambda},\omega_{\lambda})$ of the stratum $\mathcal{H}(3,3,-1,-7)$ are translation surfaces parametrized by $\mathcal{H}(6,-4,-4)$.
\end{ex}
\par 
In the following, we will focus on strata of differentials with at least three poles (i.e., $p\geq 3$). 
\par 
\begin{lem}\label{lem:nosimplepole}
Consider a stratum $\mathcal{H}(a_{1},a_{2},\rec[-1][s],-b_{1},\dots,-b_{t})$ with $p=s+t \geq 3$, $b_{1},\dots,b_{t} \geq 2$,  and $a_{1} \leq a_{2}$. If every singularity of $(\overline{\mathcal{F}}_{\lambda},\omega_{\lambda})$ has even order, then $s=0$, and additionally, $a_{1}$ and $a_{2}$ must have the same parity.
\end{lem}
\par 
\begin{proof}
The zeros of $\omega_{\lambda}$ are of even order, so $a_{1}+a_{2}$ is even. It follows that $a_{1}$ and $a_{2}$ must have the same parity. We also have $t \geq 1$, because otherwise,  Theorem~\ref{thm:MAIN3} would imply that every pole of $\omega_{\lambda}$ is simple.
\par
We first assume $s \geq 1$ and consider a degeneration corresponding to a nonsimple pole of $\omega_{\lambda}$. This degeneration  corresponds to a tripartition $I \sqcup J \sqcup K$. Without loss of generality, we assume that $I$ contains a simple pole (which cannot belong to $K$). The pole of $\omega_{\lambda}$ corresponding to this degeneration is of order $a_{1}-b_{I}+2$ or $1+\lcm(a_{1}-b_{I}+1,a_{2}-b_{J}+1)$, depending on whether $J$ is empty or not. Since every pole of $\omega_{\lambda}$ is of even order, it follows that $a_{1}-b_{I}+1$ is odd.
\par
By moving a simple pole from $I$ to $J$, we obtain a tripartition $I' \sqcup J' \sqcup K$ corresponding to another pole of $\omega_{\lambda}$. If $I'$ is not empty, the pole is of order $1+\lcm(a_{1}-b_{I}+2,a_{2}-b_{J})$, which is odd (since $a_{1}-b_{I}+1$ is odd), leading to a contradiction. Therefore, $I'$ must be empty, and the pole of $\omega_{\lambda}$ corresponding to this degeneration is of order $a_{2}-b_{J}+1$, which is thus even.
\par
Now we consider two cases depending on whether $J$ is empty or not. If $J$ is nonempty, then the pole corresponding to the tripartition $I \sqcup J \sqcup K$ is of order $1+\lcm(a_{1}-b_{I}+1,a_{2}-b_{J}+1)$, and thus both $a_{1}-b_{I}+1$ and $a_{2}-b_{J}+1$ are odd. However, we already know that $a_{2}-b_{J}+1$ is even, which leads to a contradiction. Consequently, the partitions corresponding to degenerations in the fiber $\mathcal{F}_{\lambda}$ are very specific. One subset of $I$ and $J$ contains a unique simple pole, while the other subset is empty, and every other pole belongs to the subset $K$. In particular, there is only one simple pole. However, since $p \geq 3$, at least one pole satisfies $b_{j} \leq a_{2}+1$. We can then find a tripartition $I \sqcup J \sqcup K$ where this pole belongs to $J$, which again leads to a contradiction.
\end{proof}
\par 
We treat separately the case in which one zero has a very small order.
\par 
\begin{prop}\label{prop:EvenEven1}
Consider a stratum $\mathcal{H}(a_{1},a_{2},-b_{1},\dots,-b_{p})$ with $a_{1} \leq p-2$. If every singularity of $(\overline{\mathcal{F}}_{\lambda},\omega_{\lambda})$ is of even order, then $a_{1},a_{2},b_{1},\dots,b_{p}$ are even.
\end{prop}
\par 
\begin{proof}
Following Lemma~\ref{lem:nosimplepole}, we have $b_{1},\dots,b_{p} \geq 2$. We consider partitions where $K$ contains exactly $a_{1}+1$ poles, while $J$ contains the remaining $p-a_{1}-1$ poles (with $I$ being empty). We can freely choose the orders of the poles in $J$ and $K$. Such a partition corresponds to a pole of $\omega_{\lambda}$ of order $2+a_{2}-b_{J}$ (which is even by hypothesis). Thus, $b_{J}$ must have the same parity as $a_{2}$. Unless all the $b_{1},\dots,b_{p}$
have the same parity, this is not possible for every partition corresponding to a degeneration of $\mathcal{F}_{\lambda}$. Therefore, all the $b_{1},\dots,b_{p}$ must have the same parity.
\par
If all the $b_{1},\dots,b_{p}$ are odd, then $a_{2}$ must have the same parity as $p-a_{1}-1$. Since $a_{1}$ and $a_{2}$ have the same parity (see~Lemma~\ref{lem:nosimplepole}), we conclude that $p$ is odd. Thus, $a_{1}+a_{2}-\sum\limits_{j=1}^{p} b_{j}$ is odd, which is a contradiction.
\par
If all the $b_{1},\dots,b_{p}$ are even, then $a_{2}$ and $2+a_{2}-b_{J}$ must have the same parity. Since $2+a_{2}-b_{J}$ is even (it is the order of the corresponding pole of $\omega_{\lambda}$), we deduce that $a_{2}$, and therefore $a_{1}$, are even. Consequently, all singularities of $\omega_{\lambda}$ are of even order.
\end{proof}
\par 
We will now prove the general case.
\par 
\begin{proof}[Proof of the second part of Theorem~\ref{thm:MAIN3}]
Following Propositions~\ref{prop:Sec7Direct},~\ref{prop:EvenEven1}, and Lemma~\ref{lem:nosimplepole}, it remains to prove that in the case of a stratum $\mathcal{H}(a_{1},a_{2},-b_{1},\dots,-b_{p})$ with $p-1 \leq a_{1} \leq a_{2}$ and $b_{1},\dots,b_{p} \geq 2$, if every singularity of  $(\overline{\mathcal{F}}_{\lambda},\omega_{\lambda})$  is of even order, then $a_{1},a_{2},b_{1},\dots,b_{p}$ are even.
\par
We know from Lemma~\ref{lem:nosimplepole} that $a_{1}$ and $a_{2}$ have the same parity. We first assume that both $a_{1}$ and $a_{2}$ are odd. We then split the case based on whether all the $b_{1},\dots,b_{p}$ are odd.
\par
First, we assume that one of the pole orders $b_{j}$ is even. If $b_{j} \leq a_{2}+1$, then there is a degeneration where $J$ is a singleton formed by this pole. In this case, $a_{2}-b_{j}+1$ is even, so the corresponding pole of $\omega_{\lambda}$ would be of odd order. This is a contradiction, so we assume that every even order $b_{j}$ among $b_{1},\dots,b_{p}$ satisfies $b_{j} \geq a_{2}+3$. Consequently, there can be at most one such pole, $b_{1} $ (if there were two, their sum would exceed $a_{1}+a_{2}+2$). Thus, the sum of the orders of the poles of odd order,  $b_{2},\dots,b_{p}$, is at most $a_{1}-1$. Therefore, there exists a tripartition $I \sqcup J \sqcup K$ such that $I$  contains all the poles of odd order, while $K$ contains at least the pole of order $b_{1}$. Since the number $p-1$ of poles of odd order is even, the pole corresponding to this tripartition is automatically of odd order (because $a_{1}-b_{I}+2$ is odd).
\par
Now, we consider the case where all the $b_{1},\dots,b_{p}$ are odd. Since $a_{1}+a_{2}$ is even, $p$ is also even and satisfies $p \geq 4$. The two poles with the smallest orders have a total order of at most $1+a_{2}$. There is a partition where these two poles are the only poles in $J$. This partition also corresponds to a pole of odd order for  $\omega_{\lambda}$, which is impossible. Thus, we have eliminated all cases where the two zeros are of odd order.
\par
Now, we consider the case where both $a_{1}$ and $a_{2}$ are even. The number $\alpha$ of odd numbers among $b_{1},\dots,b_{p}$ is even. We need to prove that $\alpha=0$. Since $p \geq 3$, if there are poles of odd order among $b_{1},\dots,b_{p}$, the smallest of them has order at most $a_{2}$ (because we would have $\alpha \geq 2$). There is a partition where this pole is the only one in $J$. Such a partition corresponds to a pole of odd order. Therefore, all the $b_{1},\dots,b_{p}$ must be even when the singularities of $\omega_{\lambda}$ are of even order.
\end{proof}
\par 
\subsection{Isoresidual fibers lying in disconnected strata of translation surfaces}
\par 
It is possible for the connected components of isoresidual fibers to belong to strata that are not connected. In such cases, the connected components of $\mathcal F_\lambda$ are classified using two types of topological invariants: hyperellipticity and invariants derived from the winding numbers of loops in the translation structure (see Section~\ref{sub:CCClassification}).
\par
Recall that strata with hyperelliptic components have at most two zeros and two poles. We will prove the following: if an isoresidual fiber has so few singularities, then it is automatically of genus zero (and therefore belongs to a connected stratum). \par 
\begin{prop}\label{prop:AdditionalTwozeros}
The only strata $\mathcal{H}(a_{1},a_{2},-b_{1},\dots,-b_{p})$ such that every connected component of the generic isoresidual fiber $(\overline{\mathcal{F}}_{\lambda},\omega_{\lambda})$ has at most two zeros are:
\begin{itemize}
    \item strata for which $p=2$;
    \item stratum $\mathcal{H}(1,1,\rec[-1][4])$, for which  $(\overline{\mathcal{F}}_{\lambda},\omega_{\lambda})$ belongs to $\mathcal{H}(2,2,\rec[-1][6])$;
    \item strata $\mathcal{H}(k,k,-2k,-1,-1)$, where  $(\overline{\mathcal{F}}_{\lambda},\omega_{\lambda})$ consists of $k$ connected components, each belonging to $\mathcal{H}(2k,2k,-k,-k,-1-k,-1-k)$.
\end{itemize}
In each of these cases, every connected component of $\mathcal{F}_{\lambda}$ has genus zero.
\end{prop}
\par 
\begin{proof}
In the case where $\mathcal{F}_{\lambda}$ is connected, the claim follows from Proposition~\ref{prop:zeros} and the enumeration of poles provided in Sections~\ref{subsub:SimplePoles} and~\ref{subsub:HigherPoles}.
\par
If $\mathcal{F}_{\lambda}$ is disconnected, then Theorem~\ref{thm:MAIN1} shows that all of them belong to the same stratum in the moduli space of translation surfaces. Theorem~\ref{thm:MAIN5} precisely characterizes the strata $\mathcal{H}(a_{1},a_{2},-b_{1},\dots,-b_{p})$ where this can occur.
\par
In the first family of exceptional strata, $a_{1}$ and $a_{2}$ are even so we have $a_{1}+a_{2} \geq 4$. The number of simple poles is even and positive so we have $p \geq 3$. It follows that the generic isoresidual fiber contains at least four zeros. The only case where it contains exactly four zeros implies $a_{1}+a_{2}=4$ and $p=3$. This characterizes stratum $\mathcal{H}(2,2,-4,-1,-1)$ which also belongs to the second exceptional family.
\par
In the second exceptional family, strata are of the form $\mathcal{H}(kc_{1},kc_{2},-kd_{1},\dots,-kd_{p-2},-1,-1)$ for some $k \geq 2$ and positive coprime integers $c_{1},c_{2},d_{1},\dots,d_{p-2}$. In this case, the generic isoresidual fiber has $k$ connected components and therefore contains at most $2k$ zeros. Since $p \geq 3$, we must have $c_{1}=c_{2}=1$. Such a stratum is either $\mathcal{H}(k,k,-k,-k,-1,-1)$ or $\mathcal{H}(k,k,-2k,-1,-1)$. In the first case, Proposition~\ref{prop:zeros} proves that the isoresidual fiber contains $2k(2k-1)$ zeros, which is impossible. In the second case, an explicit enumeration of the possible degenerations of the isoresidual fiber $\mathcal F_\lambda$ shows that it consists of $k$ spheres. 
\end{proof}
\par 
Recall that for strata of translation surfaces of genus one, the number of connected components is determined by the greatest common divisor of the orders of the singularities. In the following, we will prove that for a generic isoresidual fiber, this common divisor is either $1$ or $2$.
\par 
\begin{prop}\label{prop:Factor}
Consider a stratum $\mathcal{H}(a_{1},a_{2},-b_{1},\dots,-b_{p})$ with $p \geq 3$. If the orders of the singularities of $(\overline{\mathcal{F}}_{\lambda},\omega_{\lambda})$ are integer multiples of a positive integer $k$, then $k \in \lbrace{ 1,2 \rbrace}$.
\end{prop}
\par 
\begin{proof}
In the following, we assume that $k \neq 1$, and we will then prove that $k = 2$. In particular, we have $b_{1},\dots,b_{p} \geq k$. By convention, we assume that $a_{1} \leq a_{2}$. If $a_{1} \geq p-1$, then there is a tripartition $I \sqcup J \sqcup K$,  where both $I$ and $J$ are trivial. In this case, formulas from  Section~\ref{subsub:HigherPoles} show that the corresponding pole of~$\omega_{\lambda}$ is a double pole, and thus $k=2$. In the rest of the proof, we will assume that $a_{1} \leq p-2$.
\par
We first prove that every pole order $b_{j}$ is of the form $r+km_{j}$, where $r$ is a constant number satisfying $rp \equiv 2\pmod{k}$. Since $b_{1},\dots,b_{p} \geq k$, there are degenerations in $\mathcal{F}_{\lambda}$ such that $a_{1}+1$ poles belong to $K$, while the others belong to $J$ (with $I$ being empty). Since $a_{1} \leq p-2$, $J$ is nonempty, and the order $a_{2}-b_{J}+2$ of the pole of $\omega_{\lambda}$ corresponding to this degeneration changes when we change the choice of the poles of~$J$. More precisely, $a_{1}-b_{J}+1$ changes by a multiple of~$k$ if we permute one pole of $J$ with a pole of $K$. Consequently, there exists a number $r$ such that every pole order $b_{j}$ is of the form $r+kb_{j}$. The sum $a_{1}+a_{2}$ is also an integer multiple of $k$. Since $a_{1}+a_{2} - \sum\limits_{j=1}^{p}=-2$, we obtain that $rp \equiv 2 \pmod{k}$.
\par
We first assume that some $b_{j}$ among $b_{1},\dots,b_{p}$ satisfies $b_{p} \geq 3$. We consider a tripartition where $K$ contains $a_{1}+1$ poles (including the pole of order $b_{p}$), while $I$ is empty. The inequality $a_{1} \leq p-2$ then implies that $J$ is nonempty, and the corresponding pole of $\omega_{\lambda}$ has order $a_{2}-b_{J}+2$. We construct a new tripartition by moving a pole of order $b_{i}$ (with $i \neq p$) from $K$ to $J$. This results in a degeneration corresponding to a pole of order $a_{2}-b_{J}-b_{i}+2$. The fact that the order of every pole of $\omega_{\lambda}$ is a multiple of $k$ implies that $b_{i}$ is a multiple of $k$. We already know that every order $b_{j}$ is of the form $r+km_{j}$, where $r$ is a constant, so $r=0$ and all the  $b_{1},\dots,b_{p}$ are integer multiples of $k$. Since $a_{1}+a_{2}$ is also an integer multiple of $k$, and $-2$ is a multiple of $k$, we conclude that $k=2$.
\par
In the last case, we have $b_{1}=\dots=b_{p}=2$, and therefore $a_{1}+a_{2}=2p-2$. If $a_{1}$ and $a_{2}$ are odd, then, because $a_{1} \leq p-2$, a bipartition $I \sqcup J$, where $I$ contains $\frac{a_{1}+1}{2}$ double poles, corresponds to a simple pole of~$\omega_{\lambda}$, which is impossible.
\par
We will assume that $a_{1}$ and $a_{2}$ are even. We consider the tripartition where $I$, $J$, and $K$ contain $\frac{a_{1}}{2}$, $\frac{a_{2}}{2}$, and one double poles, respectively. We have $a_{1}-b_{I}+1=1$ and $a_{2}-b_{J}+1=1$, so the corresponding pole of $\omega_{\lambda}$ is a double pole. It follows that $k=2$ in this case.
\end{proof}

It follows from Propositions~\ref{prop:AdditionalTwozeros},~\ref{prop:Factor}, and the results of Section~\ref{sub:CCClassification} that if the generic isoresidual fibers of some stratum belong to a disconnected stratum, then the latter has exactly two connected components, classified by the parity of the  spin structure. Following the second part of Theorem~\ref{thm:MAIN3}, this occurs for a generic isoresidual fiber $\mathcal{F}_{\lambda}$ of a stratum $\mathcal{H}(a_{1},a_{2},-b_{1},\dots,-b_{p})$ if and only if the following three conditions are satisfied:
\begin{enumerate}
    \item The number $p$ of poles satisfies $p \geq 3$;
    \item The orders $a_{1},a_{2},-b_{1},\dots,-b_{p}$ of the singularities of the stratum are even;
    \item The genus $g$ of the generic isoresidual fiber $\mathcal{F}_{\lambda}$ satisfies $g \geq 1$ (Theorem~\ref{thm:MAIN5} proves that for such $\mu$,   $\mathcal{F}_{\lambda}$ is always connected).
\end{enumerate}
\par 
\begin{opprob}
Show that the third condition regarding the genus of the fiber is redundant. Indeed, the simplest example of strata satisfying the first two conditions is $\mathcal{H}(2,2,-2,-2,-2)$, and we can check directly that its generic isoresidual fibers belong to $\mathcal{H}(4^{4},(-2)^{8})$ and are therefore elliptic curves.
\end{opprob}
\par 
\begin{opprob}
Provide a systematic computation of the parity of the spin structure for isoresidual fibers of strata satisfying the aforementioned three conditions.
\end{opprob}


\printbibliography

@article{ABW,
url = {https://doi.org/10.1515/crelle-2023-0005},
title = {Moduli spaces of complex affine and dilation surfaces},
author = {Paul Apisa and Matt Bainbridge and Jane Wang},
pages = {229--243},
volume = {2023},
number = {796},
journal = {J. Reine Angew. Math.},
doi = {doi:10.1515/crelle-2023-0005},
year = {2023}
}

@Article{BCGGM,
    Author = {Matt {Bainbridge} and Dawei {Chen} and Quentin {Gendron} and Samuel {Grushevsky} and Martin {M\"oller}},
    Title = {{Compactification of strata of abelian differentials.}},
    FJournal = {{Duke Mathematical Journal}},
    Journal = {{Duke Math. J.}},
    ISSN = {0012-7094; 1547-7398/e},
    Volume = {167},
    Number = {12},
    Pages = {2347--2416},
    Year = {2018},
    Publisher = {Duke University Press},
    MSC2010 = {14H10 32G15}
}

@Article{BCGGM3,
    Author = {{Bainbridge}, Matt and {Chen}, Dawei and {Gendron}, Quentin and {Grushevsky}, Samuel and {M\"oller}, Martin},
    Title = {Strata of $k$-differentials.},
    FJournal = {{Algebraic Geometry}},
    Journal = {{Algebr. Geom.}},
    Volume = {6},
    Number = {2},
    Pages = {196--233},
    Year = {2019},
    Publisher = {European Mathematical Society (EMS) Publishing House, Zurich}
}

@article {BDG,
    AUTHOR = {Benirschke, Frederik and Dozier, Benjamin and Grushevsky,
              Samuel},
     TITLE = {Equations of linear subvarieties of strata of differentials},
   JOURNAL = {Geom. Topol.},
  FJOURNAL = {Geometry \& Topology},
    VOLUME = {26},
      YEAR = {2022},
    NUMBER = {6},
     PAGES = {2773--2830},
      ISSN = {1465-3060,1364-0380},
   MRCLASS = {30F30 (14H15 32G15 37F34)},
  MRNUMBER = {4521253},
       DOI = {10.2140/gt.2022.26.2773},
       URL = {https://doi.org/10.2140/gt.2022.26.2773},
}

@Article{Bo,
    Author = {{Boissy}, Corentin},
    Title = {Connected components of the strata of the moduli space of meromorphic differentials.},
    FJournal = {{Commentarii Mathematici Helvetici}},
    Journal = {{Comment. Math. Helv.}},
    ISSN = {0010-2571; 1420-8946/e},
    Volume = {90},
    Number = {2},
    Pages = {255--286},
    Year = {2015},
    Publisher = {European Mathematical Society (EMS) Publishing House, Zurich},
    DOI = {10.4171/CMH/353},
    MSC2010 = {32G15 30F30 57R30}
}

@article {CMZ,
    AUTHOR = {{Costantini}, Matteo and {M\"oller}, Martin and {Zachhuber},
              Jonathan},
     TITLE = {The {C}hern classes and {E}uler characteristic of the moduli
              spaces of {A}belian differentials},
   JOURNAL = {Forum Math. Pi},
  FJOURNAL = {Forum of Mathematics. Pi},
    VOLUME = {10},
      YEAR = {2022},
     PAGES = {Paper No. e16, 55},
      ISSN = {2050-5086},
   MRCLASS = {14D23 (30F60 32G15 57R20)},
  MRNUMBER = {4448178},
       DOI = {10.1017/fmp.2022.10},
       URL = {https://doi.org/10.1017/fmp.2022.10},
}

@article {EMZ,
    AUTHOR = {{Eskin}, Alex and {Masur}, Howard and {Zorich}, Anton},
     TITLE = {Moduli spaces of {A}belian differentials: the principal
              boundary, counting problems, and the {S}iegel-{V}eech
              constants},
   JOURNAL = {Publ. Math. Inst. Hautes \'Etudes Sci.},
  FJOURNAL = {Publications Math\'ematiques. Institut de Hautes \'Etudes
              Scientifiques},
    VOLUME = {97},
      YEAR = {2003},
     PAGES = {61--179},
      ISSN = {0073-8301},
   MRCLASS = {32G15 (30F10 37D40 37D50 37F99)},
  MRNUMBER = {MR2010740 (2005b:32029)},
MRREVIEWER = {Serge L. Tabachnikov},
}

@misc{BCGGM2,
    title={The moduli space of multi-scale differentials},
        Author = {{Bainbridge}, Matt and {Chen}, Dawei and {Gendron}, Quentin and {Grushevsky}, Samuel and {M\"oller}, Martin},
    year={2019},
    eprint={1910.13492},
    archivePrefix={arXiv},
}

@ARTICLE{tahar,
   author = {{Tahar}, Guillaume},
    title = "{Counting saddle connections in flat surfaces with poles of higher order}",
  FJournal = {{Geometriae Dedicata}}, 
    Journal = {{Geom. Dedicata}},
    Volume = {196},
    number = {1},
    Pages = {145--186},
    Year = {2018},
 Publisher = {Springer Netherlands},
}

@InCollection{Zo,
    Author = {Anton {Zorich}},
    Title = {Flat surfaces.},
    BookTitle = {{Frontiers in number theory, physics, and geometry I. On random matrices, zeta functions, and dynamical systems. Papers from the meeting, Les Houches, France, March 9--21, 2003}},
    ISBN = {978-3-540-23189-9/hbk},
    Pages = {437--583},
    Year = {2006},
    Publisher = {Springer},
    MSC2010 = {32G15 57M50 30F30 37D50 37D40 30F60},
    Zbl = {1129.32012}
}

@Article{CMSZ,
    Author = {{Chen}, Dawei and {Möller} Martin and {Sauvaget}, Adrien and {Zagier}, Don},
    Title = {{Masur–Veech volumes and intersection theory on moduli spaces of Abelian differentials.}},
    FJournal = {{Inventiones Mathematicae}},
    Journal = {{Invent. Math.}},
    Volume = {222},
    Pages = {283--373},
    Year = {2020},
}

@article{getaab,
     author = {Quentin Gendron and Guillaume Tahar},
     title = {Diff\'erentielles ab\'eliennes \`a~singularit\'es~prescrites},
     journal = {J. Éc. Polytech., Math.},
     fjournal = {Journal de l{\textquoteright}\'Ecole polytechnique {\textemdash} Math\'ematiques},
     pages = {1397--1428},
     publisher = {\'Ecole polytechnique},
     volume = {8},
     year = {2021},
     doi = {10.5802/jep.174},
     url = {https://jep.centre-mersenne.org/articles/10.5802/jep.174/}
}

@Article{chge,
 Author = {Chen, Dawei and Gendron, Quentin},
 Title = {Towards a classification of connected components of the strata of {{\(k\)}}-differentials},
 FJournal = {Documenta Mathematica},
 Journal = {Doc. Math.},
 ISSN = {1431-0635},
 Volume = {27},
 Pages = {1031--1100},
 Year = {2022},
 DOI = {10.25537/dm.2022v27.1031-1100},
 zbMATH = {7554114}
}

@Article{GeTaIso,
 Author = {Gendron, Quentin and Tahar, Guillaume},
 Title = {Isoresidual fibration and resonance arrangements},
 FJournal = {Letters in Mathematical Physics},
 Journal = {Lett. Math. Phys.},
 ISSN = {0377-9017},
 Volume = {112},
 Number = {2},
 Note = {Id/No 33},
 Year = {2022},
 DOI = {10.1007/s11005-022-01528-z},
 zbMATH = {7513999}
}

@article{ChPrIso,
  title={Counting differentials with fixed residues},
  author={Chen, Dawei and Prado, Miguel},
  journal={Letters in Mathematical Physics},
  volume={115},
  number={3},
  pages={53},
  year={2025},
  publisher={Springer}
}

@article{SaBraid,
  title={Stratified braid groups: monodromy},
  author={Salter, Nick},
  journal={Mathematical Proceedings of the Cambridge Philosophical Society},
  volume={178},
  number={2},
  pages={259--292},
  year={2025},
}

@misc{FaTaZa,
      title={Isoperiodic foliation of the stratum $\mathcal{H}(1,1,-2)$}, 
      author={Gianluca Faraco and Guillaume Tahar and Yongquan Zhang},
      year={2023},
      eprint={2305.06761},
      archivePrefix={arXiv},
}

@article{BuRoCoun,
      title={Counting meromorphic differentials on $\mathbb{CP}^1$}, 
      author={Alexandr Buryak and Paolo Rossi},
      FJournal = {Letters in Mathematical Physics},
      Journal = {Lett. Math. Phys.},
      Volume = {114},
      Number = {97},
      Year = {2024}
}

@article{Kawazumi,
    author = {Kawazumi, Nariya},
    title = "{The mapping class group orbits in the framings of compact surfaces}",
    fjournal = {The Quarterly Journal of Mathematics},
    journal = {Q. J. Math.},
    volume = {69},
    number = {4},
    pages = {1287-1302},
    year = {2018},
    month = {05},
    issn = {0033-5606},
    doi = {10.1093/qmath/hay024},
    url = {https://doi.org/10.1093/qmath/hay024},
}

@article {Sugiyama,
    AUTHOR = {Sugiyama, Toshi},
     TITLE = {The moduli space of polynomial maps and their fixed-point
              multipliers},
   JOURNAL = {Adv. Math.},
  FJOURNAL = {Advances in Mathematics},
    VOLUME = {322},
      YEAR = {2017},
     PAGES = {132--185},
      ISSN = {0001-8708,1090-2082},
   MRCLASS = {37F10 (14C17 14D20 37C25)},
  MRNUMBER = {3720796},
MRREVIEWER = {Fabrizio\ Bianchi},
       DOI = {10.1016/j.aim.2017.10.013},
       URL = {https://doi.org/10.1016/j.aim.2017.10.013},
}

@article{Pan,
author = {Dmitri Panov},
title = {{Polyhedral Kähler manifolds}},
volume = {13},
journal = {Geom. Topol.},
number = {4},
publisher = {MSP},
pages = {2205 -- 2252},
year = {2009},
doi = {10.2140/gt.2009.13.2205},
URL = {https://doi.org/10.2140/gt.2009.13.2205}
}

@Book{AMTransSurf,
 Author = {Athreya, Jayadev and Masur, Howard},
 Title = {Translation surfaces},
 FSeries = {Graduate Studies in Mathematics},
 Series = {Grad. Stud. Math.},
 ISSN = {1065-7338},
 Volume = {242},
 ISBN = {978-1-4704-7655-7; 978-1-4704-7677-9; 978-1-4704-7676-2},
 Year = {2024},
 Publisher = {American Mathematical Society},
 DOI = {10.1090/gsm/242},
 zbMATH = {7813234}
}

@Article{KLS,
 Author = {Krichever, Igor and Lando, Sergei and Skripchenko, Alexandra},
 Title = {Real-normalized differentials with a single order 2 pole},
 FJournal = {Letters in Mathematical Physics},
 Journal = {Lett. Math. Phys.},
 ISSN = {0377-9017},
 Volume = {111},
 Number = {2},
 Pages = {19},
 Note = {Id/No 36},
 Year = {2021},
 DOI = {10.1007/s11005-021-01379-0},
 zbMATH = {7345369},
 Zbl = {1462.14032}
}

@article{CMDisoper,
  title={Isoperiodic meromorphic forms: two simple poles},
  author={Calsamiglia, Gabriel and Deroin, Bertrand},
  journal={Groups, Geometry, and Dynamics},
  volume={20},
  number={1},
  pages={107--168},
  year={2025}
}

@article{BGPA,
  title={The space of solvable Pell--Abel equations},
  author={Bogatyr{\"e}v, Andrei and Gendron, Quentin},
  journal={Compositio Mathematica},
  volume={161},
  number={7},
  pages={1483--1511},
  year={2025},
  publisher={London Mathematical Society}
}

@Article{BenBound,
 Author = {Benirschke, Frederik},
 Title = {The boundary of linear subvarieties},
 FJournal = {Journal of the European Mathematical Society (JEMS)},
 Journal = {J. Eur. Math. Soc. (JEMS)},
 ISSN = {1435-9855},
 Volume = {25},
 Number = {11},
 Pages = {4521--4582},
 Year = {2023},
 DOI = {10.4171/JEMS/1287},
 zbMATH = {7774913},
 Zbl = {1537.32039}
}
\end{document}